\documentclass[12pt]{article}
\usepackage[includeheadfoot,top=2 cm, bottom=2 cm, left=2 cm, right=2 cm]{geometry}
\usepackage{pdfpages}
\usepackage{graphicx}
\usepackage[numbers,sort&compress]{natbib}

\usepackage{amsfonts,amsthm}
\usepackage{upgreek}
\usepackage{bm}

\usepackage{hyperref}
\usepackage{enumerate}   
\usepackage{algcompatible}
\usepackage{algorithm}
\usepackage{amssymb}
\usepackage{multirow}
\usepackage{array}
\usepackage{mathtools}
\usepackage{color}
\usepackage{xcolor}
\usepackage{float}
\usepackage{subcaption}
\hypersetup{
	colorlinks,
	linkcolor={blue!},
	citecolor={blue!},
	urlcolor={blue!}
}

\newcommand{\bl}{\mathbf{l}}
\newcommand{\ba}{\mathbf{a}}
\newcommand{\bb}{\mathbf{b}}
\newcommand{\bu}{\mathbf{u}}

\newcommand{\bv}{\mathbf{v}}
\newcommand{\by}{\mathbf{y}}
\newcommand{\bt}{\mathbf{t}}
\newcommand{\bz}{\mathbf{z}}
\newcommand{\bx}{\mathbf{x}}
\newcommand{\bw}{\mathbf{w}}

\newcommand{\br}{\mathbf{r}}
\newcommand{\bnull}{\mathbf{0}}
\newcommand{\bTheta}{\mathbf{\Theta}}

\newcommand{\bPhi}{\mathbf{\Phi}}
\newcommand{\bR}{\mathbf{R}}
\newcommand{\bA}{\mathbf{A}}

\newcommand{\bT}{\mathbf{T}}
\newcommand{\bM}{\mathbf{M}}

\newcommand{\bH}{\mathbf{H}}

\newcommand{\bOmega}{\mathbf{\Omega}}
\newcommand{\bGamma}{\mathbf{\Gamma}}

\newcommand{\bQ}{\mathbf{Q}}

\newcommand{\bP}{\mathbf{P}}

\newcommand{\bU}{\mathbf{U}}
\newcommand{\bL}{\mathbf{L}}
\newcommand{\bV}{\mathbf{V}}
\newcommand{\bW}{\mathbf{W}}

\newcommand{\bc}{\mathbf{c}}

\newcommand{\keywords}[1]{{\textit{Keywords---}} #1}

\newtheorem{lemma}{Lemma}[section]
\newtheorem{proposition}[lemma]{Proposition}
\newtheorem{corollary}[lemma]{Corollary}

\newtheorem{remark}[lemma]{Remark}
\newtheorem{definition}[lemma]{Definition}

\begin{document}

	\title{Randomized linear algebra for model reduction. \\
		Part II: minimal residual methods and dictionary-based approximation.
	}
	\author{ 
		Oleg Balabanov\footnotemark[1]~\footnotemark[2] ~and Anthony Nouy\footnotemark[1]~\footnotemark[3]
	}
	
	\renewcommand{\thefootnote}{\fnsymbol{footnote}}
	\footnotetext[1]{Centrale Nantes, LMJL, UMR CNRS 6629, France.}
	\footnotetext[2]{Polytechnic University of Catalonia, LaC\`an, Spain.}
	\footnotetext[3]{Corresponding author (anthony.nouy@ec-nantes.fr).}
	\renewcommand{\thefootnote}{\arabic{footnote}}
	
	\date{}
	\maketitle

\begin{abstract}
	A methodology for using random sketching in the context of model order reduction for high-dimensional parameter-dependent systems of equations was introduced in [Balabanov and Nouy 2019, Part I]. Following this framework, we here construct a reduced model from a small, efficiently computable random object called a sketch of a reduced model, using minimal residual methods. We introduce a sketched version of the minimal residual {based} projection as well as a novel nonlinear approximation method, where for each parameter value, the solution is approximated by minimal residual projection onto a subspace spanned by several vectors picked {(online)} from a dictionary of candidate basis vectors. It is shown that random sketching technique can improve not only efficiency but also numerical stability. A rigorous analysis of the conditions on the random sketch required to obtain a given accuracy is presented. These conditions may be ensured a priori with high probability by considering for the sketching matrix an oblivious embedding of sufficiently large size.  Furthermore, a simple and reliable procedure for a posteriori verification of the quality of the sketch is provided. This approach can be used for certification of the approximation as well as for adaptive selection of the size of the random sketching matrix.

\end{abstract}
 
\section{Introduction}
We consider large parameter-dependent systems of equations
\begin{equation} \label{eq:initialproblem}
\bA(\mu)\bu(\mu)= \bb(\mu),~\mu \in \mathcal{P},
\end{equation}
where $\bu(\mu)$ is a solution vector, $\bA(\mu)$ is a parameter-dependent {matrix}, $\bb(\mu)$ is a parameter-dependent right hand side and $\mathcal{P}$ is a parameter set. 
Parameter-dependent problems are considered for many purposes such as {design,} control, optimization, uncertainty quantification or inverse problems.

{Solving~\textup {(\ref {eq:initialproblem})} for many parameter values can be computationally {unfeasible}. Moreover, for real-time applications, a quantity of interest ($\bu(\mu)$ or a function of $\bu(\mu)$) has to be estimated on the fly in highly limited computational time for a certain value of $\mu$. Model order reduction (MOR) methods are developed for efficient approximation of the quantity of interest for each parameter value. They typically consist of two stages.} In the first so-called offline stage a reduced model is constructed from the full order model. This stage usually involves expensive computations such as evaluations of $\bu(\mu)$ for several parameter values, computing multiple high-dimensional matrix-vector and inner products, etc., but this stage is performed only once. Then, for each given parameter value, the precomputed reduced model is used for efficient approximation of the solution or an output quantity with a computational cost independent of the dimension of the initial system of equations~\textup {(\ref {eq:initialproblem})}. For {a detailed presentation of the classical MOR methods} such as Reduced Basis (RB) method and Proper Orthogonal Decomposition (POD) the reader can refer to~\cite{morbook2017}. In the present work the approximation of the solution shall be obtained with a minimal residual (minres) projection on a reduced (possibly parameter-dependent) subspace.  The minres projection can be interpreted as a Petrov-Galerkin projection where the test space is chosen to minimize some norm of the residual~\cite{bui2008model,amsallem2013}. Major benefits over the classical Galerkin projection include an improved stability (quasi-optimality) for non-coercive problems and more effective residual-based error bounds of an approximation (see e.g.~\cite{bui2008model}). In addition, minres methods are better suited to random sketching as will be seen in the present article.

In recent years randomized linear algebra (RLA) became a popular approach in the fields such as data analysis, machine learning, compressed sensing, etc.~\cite{woodruff2014sketching,mahoney2011randomized,tropp2007signal}. This {probabilistic approach for numerical linear algebra} can yield a drastic computational cost reduction in terms of classical metrics of efficiency such as complexity (number of flops) and memory consumption. Moreover, it can be highly beneficial in extreme computational environments that are typical in contemporary scientific computing. For instance, RLA can be essential when data has to be analyzed only in one pass (e.g., when it is streamed from a server) or when it is distributed on multiple workstations with expensive communication costs.  

Despite their indisputable success in fields closely related to MOR, the aforementioned techniques only recently started {to be {extensively} used in MOR community.} One of the earliest works considering RLA in the context of MOR is~\cite{zahm2016interpolation}, where the authors proposed to use RLA for interpolation of (implicit) inverse of a parameter-dependent matrix. In~\cite{buhr2017randomized} the RLA was used for approximating the range of a transfer operator and for computing a probabilistic bound for {the} approximation error.  {In~\cite{cao2004,homescu2005error,smetana2018randomized} the authors developed probabilistic error estimators based on random projections and the adjoint method. These approaches can also be formulated in RLA framework. {A randomized singular value decomposition (see~\cite{halko2011finding,woodruff2014sketching}) was used for computing the POD vectors in~\cite{hochman2014reduced}.} Efficient algorithms for Dynamic Mode Decomposition based on RLA and compressed sensing were proposed in~\cite{erichson2016randomized,bistrian2017randomized,erichson2017randomized,kramer2017sparse,brunton2015}, though with a lack of theoretical analysis.
	
	{As already shown in~\cite{balabanov2019galerkin}, random sketching can lead to  drastic reduction of the computational costs of classical MOR methods.} A random sketch of a reduced model is defined as a set of small random projections of the reduced basis vectors and the associated residuals. Its representation (i.e, affine decomposition\footnote{A parameter-dependent quantity $\bv(\mu)$ with values in a vector space $V$ over $\mathbb{K}$ is said to admit an affine representation if $\bv(\mu) = \sum \bv_i \lambda_i(\mu)$ with $\lambda_i(\mu) \in \mathbb{K}$ and $\bv_i \in V$.}) {can be efficiently precomputed} in basically any computational architecture. The random projections should be chosen according to the metric of efficiency, e.g., number of flops, memory consumption, communication cost between distributed machines, scalability, etc.  {A rigorous analysis of the cost of obtaining a random sketch in different computational environments can be found in~\cite[Section 4.4]{balabanov2019galerkin}. When a sketch has been computed, the reduced model can be approximated without operating on large vectors but only on their small sketches typically requiring  negligible computational costs. The approximation can be guaranteed to almost preserve the quality of the original reduced model with user-specified probability. The computational cost depends only logarithmically on the probability of failure, which can therefore be chosen very small, say $10^{-10}$.  In~\cite{balabanov2019galerkin} it was shown how random sketching can be employed for an efficient approximation of the Galerkin projection, the computation of the norm of the residual for error estimation, and the computation of {a} primal-dual correction. Furthermore, new efficient sketched versions of {the weak} greedy algorithm and Proper Orthogonal Decomposition were introduced for generation of reduced bases. 
		
		{The present work is a continuation of~\cite{balabanov2019galerkin}. Here we adapt the random sketching technique to minimal residual methods, propose a novel dictionary-based approximation method and additionally discuss the questions of a posteriori certification of the sketch and efficient extraction of the output quantity from a solution of the reduced model.
		}

		\subsection{Contributions and outline} \label{contr}
		
		First is proposed a method to approximate minres projection from a random sketch whose precomputation presents a much lower complexity, memory consumption and communication cost~(see~\cite[Section 4.4]{balabanov2019galerkin} for details) than precomputation of affine decomposition of the reduced normal system of equations in the classical offline stage. The sketched minres projection is obtained online by solving a small least-squares problem, that can be more online-efficient and stable than solving the normal system of equations. Precise conditions on the sketch to yield {(approximate) preservation} of quasi-optimality constants of the standard minres projection are provided. The conditions do not depend on operator's properties, which implies robustness for ill-conditioned and non-coercive problems in contrast to the sketched Galerkin methods~\cite{balabanov2019galerkin}. 
		
		In~Section\nobreakspace \ref {dbminres} we introduce a novel nonlinear method, where the solution to~\textup {(\ref {eq:initialproblem})} is approximated by minres projection onto a subspace with basis picked online from a dictionary of candidate basis vectors. {In Section\nobreakspace \ref {dictapprox} is shown why the dictionary-based approach is a more natural choice than the  $hp$-refinement method~\cite{eftang2010hp,eftang2011hp} or other partitioning methods, for problems where the solution is a superposition of several components (e.g., for PDEs with multiple transport phenomena).}  The dictionary-based approximation can be efficiently obtained online by solving a small sparse least-squares problem assembled from a random sketch of dictionary vectors, which entails practical feasibility of the method. 
		
		The potential of approximation with dictionaries for problems with a slow decay of the Kolmogorov $r$-widths of the solution manifold was revealed in~\cite{dihlmann2012online,kaulmann2013online}.  Although they improved classical approaches, the algorithms proposed in~\cite{dihlmann2012online,kaulmann2013online} still involve in general heavy computations in both offline and online stages, and can suffer from round-off errors. More specifically, the offline complexity and memory consumption associated with post-processing of the snapshots  in~\cite{dihlmann2012online,kaulmann2013online} are at least  $\mathcal{O}(nK^2)$  and $\mathcal{O}(nK)$, respectively, where $K$ is the cardinality of the dictionary and $n$ is the dimension of the original problem~\textup {(\ref {eq:initialproblem})}. With random sketching these costs may be reduced (by at least) a factor of $\mathcal{O}(K/\log{K})$. Furthermore, the online complexity and memory consumption of the approach in~\cite{kaulmann2013online} are at least $\mathcal{O}(Kr^3)$ and $\mathcal{O}(K^2)$, respectively, where $r$ is the dimension of the reduced approximation space. In its turn, our approach consumes (at least) $\mathcal{O}(r/\log{K})$ and $\mathcal{O}(K/(r \log{K}))$ times less number of flops and amount of storage, respectively.
		
		{We also provide a probabilistic approach for a posteriori certification of the quality of (random) embedding and the associated sketch.  The certification can be performed with a computational cost (much) less than the cost of obtaining the (sketched) reduced model's solution. The proposed procedure can be particularly useful for adaptive selection of the size of a random sketching matrix since the a priori bounds can be pessimistic in practice~\cite{balabanov2019galerkin}.} 	
		
		Finally, in Appendix A we propose an approach, based on random sketching, for efficient extraction of linear/quadratic quantity of interest or primal-dual correction from the reduced model's solution. {It consists in first approximating the  solution with a projection on a new (very small or sparse) reduced basis, that yields an efficiently computable approximation of the output quantity. Then this approximation is improved by a correction  computable from sketches of the reduced bases with a negligible computational cost.} 	
		This approach can be particularly important when considering a large approximation space or dictionary, with distributed/streamed basis vectors.

		{{The outline of the article is as follows. In~Section\nobreakspace \ref {prel} we introduce the problem setting and recall the main ingredients of the framework developed in~\cite{balabanov2019galerkin}. The minimal residual method considering a projection on a single low-dimensional subspace is discussed in~Section\nobreakspace \ref {minres}. We present a standard minres projection in a discrete form followed by its efficient approximation with random sketching. Section\nobreakspace \ref {dbminres} presents a novel dictionary-based minimal residual method using random sketching. A posteriori verification of the quality of a sketch and few scenarios where such a procedure can be used are provided in~Section\nobreakspace \ref {sketchcert}. The methodology is validated numerically on two nontrivial benchmark problems in~Section\nobreakspace \ref {numex}.}} {A discussion on the efficient and stable extraction of the quantity of interest from the reduced model's solution is given in Appendix A.}

		\section{Preliminaries} \label{prel}
		
		Let $\mathbb{K} = \mathbb{R}$ or  $\mathbb{C}$ and let $U:=\mathbb{K}^{n}$ and $U':=\mathbb{K}^{n}$ represent the solution space and its dual, respectively. The solution $\bu(\mu)$ is an element from $U$, $\bA(\mu)$ is a linear operator from $U$ to $U'$, the right hand side $\bb(\mu)$ and the extractor of the quantity of interest $\bl(\mu)$ are elements of $U'$. 
		
		Spaces $U$ and $U'$ are equipped with inner products $\langle \cdot , \cdot  \rangle_U := \langle \bR_U \cdot,  \cdot \rangle$ and $\langle \cdot , \cdot \rangle_{U'}:= \langle  \cdot, \bR_U^{-1}\cdot  \rangle $, where $\langle \cdot, \cdot \rangle$ is the canonical $\ell_2$-inner product on $\mathbb{K}^{n}$ and $\bR_U : U \to U'$ is some symmetric (for $\mathbb{K} = \mathbb{R}$), or Hermitian (for $\mathbb{K} = \mathbb{C}$), positive definite operator. We denote by $\| \cdot \|$ the canonical $\ell_2$-norm on $\mathbb{K}^{n}$. Finally, for a {matrix} $\bM$ we denote by $\bM^\mathrm{H}$ its (Hermitian) transpose. 
		
		\subsection{Random sketching} \label{randsk}
		A framework  for using random sketching (see~\cite{halko2011finding,woodruff2014sketching}) in the context of MOR was introduced in~\cite{balabanov2019galerkin}. The sketching technique is seen as a modification of {the} inner product in a given subspace (or a collection of subspaces). The modified inner product is an estimation of the original one and is much easier and more efficient to operate with. {Next, we briefly recall the basic preliminaries from~\cite{balabanov2019galerkin}.}

		Let $V$ be a subspace of $U$. The dual of $V$ is identified with a subspace $V':=\{ \bR_{U} \bx : \bx \in V \}$ of $U'$. For a matrix $\bTheta \in \mathbb{K}^{k\times n}$ with $k\leq n$ we define the following semi-inner products on $U$:
		\begin{equation} \label{eq:thetadef}
		\langle \cdot, \cdot \rangle^{\bTheta}_{U}:= \langle \bTheta \cdot, \bTheta \cdot \rangle, \textup{ and }
		\langle \cdot, \cdot \rangle^{\bTheta}_{U'} := \langle \bTheta \bR_U^{-1} \cdot,  \bTheta \bR_U^{-1} \cdot \rangle,
		\end{equation}
		and we let $\| \cdot \|^{\bTheta}_{U}$ and $\| \cdot \|^{\bTheta}_{U'}$ denote the associated semi-norms. 
		
		\begin{definition} \label{def:epsilon_embedding}
			A matrix $\bTheta$ is called a $U \to \ell_2$ $\varepsilon$-subspace embedding (or simply an $\varepsilon$-embedding) for $V$, if it satisfies
			\begin{equation} \label{eq:epsilon_embedding}
			\forall \bx, \by \in V, \ \left | \langle \bx, \by \rangle_U - \langle \bx, \by \rangle^{\bTheta}_{U} \right |\leq \varepsilon \| \bx \|_U \| \by \|_U.
			\end{equation}
		\end{definition}
		
		Here $\varepsilon$-embeddings shall be constructed as realizations of random matrices that are built in an oblivious way without any a priori knowledge of $V$. 
		
		\begin{definition} \label{def:oblepsilon_embedding}
			A random matrix $\bTheta$ is called a $( \varepsilon, \delta, d)$ oblivious $U \to \ell_2$ subspace embedding if it is an $\varepsilon$-embedding for an arbitrary $d$-dimensional subspace {$V \subset U$} with probability at least $1-\delta$.
		\end{definition}

		Oblivious $\ell_2 \to \ell_2$ subspace embeddings (defined by~Definition\nobreakspace \ref {def:epsilon_embedding} with $\langle \cdot, \cdot \rangle_{U} := \langle \cdot, \cdot \rangle$) include the rescaled Gaussian distribution, the rescaled Rademacher distribution, the Subsampled Randomized Hadamard Transform (SRHT), the Subsampled Randomized Fourier Transform (SRFT), CountSketch matrix, SRFT combined with sequences of random Givens rotations, and others~\cite{balabanov2019galerkin,halko2011finding,woodruff2014sketching,rokhlin2008fast}. In this work we shall rely on the rescaled Gaussian distribution and SRHT.
		
		An oblivious $U \to \ell_2$ subspace embedding for a general inner product $\langle \cdot, \cdot \rangle_{U}$ can be constructed as 
		\begin{equation}
		\bTheta = \bOmega \bQ, 
		\end{equation}
		where $\bOmega$ is a $\ell_2 \to \ell_2$  subspace embedding and $\bQ \in \mathbb{K}^{s\times n}$ is an easily computable (possibly rectangular) matrix such that $\bQ^{\mathrm{H}} \bQ= \bR_U$ (see~\cite[Remark 2.7]{balabanov2019galerkin}).
		
		It follows that an $U \to \ell_2$ $\varepsilon$-subspace embedding for $V$ can be obtained with high probability as a realization of an oblivious subspace embedding of sufficiently large size. 
		{The number of rows for $\bTheta$ may be selected with the theoretical bounds from~\cite{balabanov2019galerkin}. These bounds, however, can be pessimistic or even impractical (e.g., for adaptive algorithms or POD). In practice, one can consider $\bTheta$ of much smaller sizes and still obtain accurate subspace embeddings. When the conditions on the size of the sketch based on a priori {analysis} are too pessimistic, {one can provide a posteriori {guarantee}.} An easy and robust procedure for a posteriori verification of the quality of $\bTheta$ is provided in~Section\nobreakspace \ref {sketchcert}. The methodology can also be used for deriving a {criterion} for adaptive selection of the size of the random sketching matrix to satisfy the $\varepsilon$-embedding property with high probability.}
		
		\subsection{A sketch of a reduced model}
		Here the output of a reduced order model is efficiently estimated from its random sketch. The $\bTheta$-sketch of a reduced model associated with a subspace $U_r$ is defined as
		\begin{equation} \label{eq:sketchUr}
		\left \{ \left \{ \bTheta \bx, \bTheta \bR_U^{-1} \br(\bx; \mu) \right \}: ~~ \bx \in U_r \right \}, 
		\end{equation}
		where $\br(\bx;\mu):= \bb(\mu)- \bA(\mu)\bx$. Let $\bU_r \in \mathbb{K}^{n \times r}$ be a matrix whose columns form a basis of $U_r$. Then each element of~\textup {(\ref {eq:sketchUr})} can be characterized from the coordinates of $\bx$ associated with $\bU_r$, i.e., a vector $\ba_r \in \mathbb{K}^{r}$ such that $\bx=\bU_r \ba_r$, and the following quantities
		\begin{equation} \label{eq:sketch}
		\bU^{\bTheta}_r:=\bTheta \bU_r,~ \bV^{\bTheta}_r(\mu):=\bTheta \bR_U^{-1}\bA(\mu) \bU_r \textup{ and }~ {\bb^{\bTheta}}(\mu):=\bTheta \bR_U^{-1} \bb(\mu). 
		\end{equation}
		{ Clearly $\bV^{\bTheta}_r(\mu)$ and $\bb^{\bTheta}(\mu)$ have affine expansions containing at most as many terms as the ones of $\bA(\mu)$ and $\bb(\mu)$, respectively. {The} matrix $\bU^{\bTheta}_r$ and the affine expansions of $\bV^{\bTheta}_r(\mu)$ and $\bb^{\bTheta}(\mu)$ are referred to as the $\bTheta$-sketch of $\bU_r$ {(a representation of the  $\bTheta$-sketch of a reduced model associated with $U_r$).}  With a good choice of an oblivious embedding, a $\bTheta$-sketch of $\bU_r$ can be efficiently precomputed in any computational environment (see~\cite{balabanov2019galerkin}). Thereafter, an approximation of a reduced order model can be obtained with a negligible computational cost.  
			
			{Note that in~\cite{balabanov2019galerkin} the affine expansion of {$\bl_r(\mu)^{\mathrm{H}}:=\bU_r^{\mathrm{H}} \bl(\mu)$, where $\bl(\mu) \in U'$ is an extractor of the linear quantity of interest $l(\bu(\mu);\mu) =  \langle \bl(\mu), \bu(\mu) \rangle$}, is also considered as a part of the $\bTheta$-sketch of $\bU_r$ and is assumed to be efficiently computable. Let us address a scenario where the computation of the affine expansion of  $\bl_r(\mu)$ {or its online evaluation} is expensive. This may happen when $\bl_r(\mu)$ has many terms {in the affine expansion} or when one considers a large approximation space (or dictionary) with possibly distributed basis vectors. In such a case, the output quantity can be approximated with an approach detailed in Appendix A. This method proceeds with two steps. First, the solution $\bu_r(\mu)$ is approximated by a projection $\bw_p(\mu)$ on a new basis $\bW_p$, which is cheap to operate with (e.g, it has less columns than $\bU_r$ or the columns are sparse). The affine decomposition of an approximation $l(\bw_p(\mu); \mu)$ of $l(\bu_r(\mu); \mu)$ can now be efficiently precomputed. Then in the second step, the accuracy of $l(\bw_p(\mu); \mu)$ is improved with a random sketching correction:
				\begin{equation*}
				\begin{split}
				l(\bu_r(\mu); \mu) &\approx l(\bw_p(\mu); \mu) + \langle \bR^{-1}_{U} \bl(\mu), \bu_r(\mu) - \bw_p(\mu) \rangle^\bTheta_U,
				\end{split}
				\end{equation*}
				computable from the sketches $\bTheta \bU_r$, $\bTheta \bW_p$ and $\bTheta \bR^{-1}_{U} \bl(\mu)$  with a negligible computational cost. Note that when interpreting random sketching as a Monte Carlo method, the proposed approach can  be linked to a control variate variance reduction method where $\bw_p(\mu)$ plays the role of the control variate for the estimation of $\langle \bl(\mu), \bu_r(\mu) \rangle$. It is also important to note that the presented approach can be employed not only to extraction of a linear quantity of interest but also  extraction of a quadratic quantity of interest or computation of the primal-dual correction.}
			
			\section{Minimal residual projection} \label{minres} 
			In this section we first present the standard minimal residual projection in a form that allows an easy introduction of random sketching. Then we introduce the sketched version of the minimal residual projection and provide conditions to guarantee its {quality}.     
			
			\subsection{Standard minimal residual projection} \label{minresproj} 
			Let $U_r \subset U$ be a subspace of $U$ (typically obtained with a greedy algorithm or approximate POD). The minres approximation $\bu_r(\mu) \in U_r$ of $\bu(\mu)$ can be defined by
			\begin{equation} \label{eq:minresproj}
			\bu_r(\mu) = \arg \min_{\bw \in U_r} \|\br(\bw; \mu)\|_{U'}.
			\end{equation} 
			For linear problems it is equivalently characterized by the following (Petrov-)Galerkin orthogonality condition:
			\begin{equation} \label{eq:minresproj2} 
			\langle \br(\bu_r(\mu);\mu), \bw \rangle=0, ~\forall \bw \in V_r(\mu),  
			\end{equation}
			{where $V_r(\mu) :=\{ \bR^{-1}_U\bA(\mu)\bx : \bx \in U_r \}$.}
			
			If the operator $\bA(\mu)$ is invertible then~\textup {(\ref {eq:minresproj})} is well-posed. In order to characterize the quality of the projection $\bu_r(\mu)$ we define the following parameter-dependent constants 
			\begin{subequations} \label{eq:zetaiota_r}
				\begin{align} 
				&\zeta_{r}(\mu):=  \min_{ \bx  \in \left (\mathrm{span} \{ \bu(\mu) \}+ U_r \right ) \backslash \{ \bnull \}} \frac{\| \bA(\mu) \bx  \|_{U'}}{\| \bx \|_U}, \label{eq:alphar} \\ 
				&\iota_{r}(\mu):=  \max_{ \bx  \in \left (\mathrm{span} \{ \bu(\mu) \}+ U_r \right ) \backslash \{ \bnull \}}  \frac{ \| \bA(\mu) \bx \|_{U'}}{\| \bx \|_U}.\label{eq:betar}
				\end{align} 
			\end{subequations}
			Let $\bP_{W}:U \rightarrow W$ denote the orthogonal projection from $U$ on a subspace $W \subset U$, defined for $\bx \in U$ by 
			\begin{equation*}
			\bP_{W} \bx = \arg\min_{\bw \in W} \| \bx- \bw \|_{U}.
			\end{equation*}
			
			\begin{proposition}  \label{thm:cea}
				If~$\bu_r(\mu)$ satisfies~\textup {(\ref {eq:minresproj})} and $\zeta_{r}(\mu)>0$, then
				\begin{equation} \label{eq:quasi-opt}
				\| \bu(\mu)- \bu_r(\mu) \|_{U} \leq  \frac{\iota_{r}(\mu)}{\zeta_{r}(\mu)} \| \bu(\mu)- \bP_{U_r} \bu(\mu) \|_{U}. 
				\end{equation}
				\begin{proof}
					See Appendix B. 
				\end{proof}
			\end{proposition}
			
			The constants $\zeta_{r}(\mu)$ and $\iota_{r}(\mu)$ can be bounded by the minimal and maximal singular values of $\bA(\mu)$:
			\begin{subequations} \label{eq:beta}
				\begin{align} 
				\alpha(\mu) &:=  \underset{ \bx  \in U \backslash \{ \bnull \}} \min \frac{\| \bA(\mu) \bx  \|_{U'}}{\| \bx \|_U} \leq \zeta_r(\mu), \\
				\beta(\mu)  &:=  \underset{ \bx  \in U \backslash \{ \bnull \}} \max \frac{\| \bA(\mu) \bx  \|_{U'}}{\| \bx \|_U} \geq \iota_r(\mu).
				\end{align} 
			\end{subequations}
			Bounds of $\alpha(\mu)$ and $\beta(\mu)$ can be obtained theoretically~\cite{haasdonk2017reduced} or numerically with the successive constraint method~\cite{huynh2007successive}. 
			
			For each $\mu$, the vector $\ba_r(\mu) \in \mathbb{K}^{r}$ such that $\bu_r(\mu) = \bU_r \ba_r(\mu)$ satisfies~\textup {(\ref {eq:minresproj2})} can be obtained by solving the following reduced {(normal)} system of equations:
			\begin{equation} \label{eq:reduced_system}
			\bA_r(\mu) \ba_r(\mu) = \bb_r(\mu), 
			\end{equation} 
			where $\bA_r(\mu)= \bU_r^{\mathrm{H}} \bA(\mu)^{\mathrm{H}} \bR_U^{-1} \bA(\mu)\bU_r \in \mathbb{K}^{r\times r}$ and $\bb_r(\mu)= \bU_r^{\mathrm{H}} \bA(\mu)^{\mathrm{H}} \bR_U^{-1} \bb(\mu) \in \mathbb{K}^r$. 
			The numerical stability of~\textup {(\ref {eq:reduced_system})} can be ensured through orthogonalization of $\bU_r$ similarly as for the classical Galerkin projection. Such orthogonalization yields the following bound for the condition number of $\bA_r(\mu)$:
			\begin{equation} \label{eq:condAr}
			\kappa (\bA_r(\mu)):= {\|\bA_r(\mu)\| \| \bA_r(\mu)^{-1} \|}  \leq \left ( \frac{\iota_{r}(\mu)}{\zeta_{r}(\mu)} \right )^2 \leq \left ( \frac{\beta(\mu)}{\alpha(\mu)}  \right )^2.
			\end{equation} 
			This bound can be insufficient for problems with matrix $\bA(\mu)$ having a high or even moderate condition number. 
			
			{The random sketching technique can be used to improve the efficiency and numerical stability of the minimal residual projection, as shown below. }
			
			\subsection{Sketched minimal residual projection} \label{skminres}

			Let $\bTheta \in \mathbb{K}^{k \times n}$ be a certain $U \to \ell_2$ subspace embedding. The sketched minres projection can be defined by~\textup {(\ref {eq:minresproj})} with the dual norm $\| \cdot \|_{U'}$ replaced by its estimation $\| \cdot \|_{U'}^\bTheta$, which results in an approximation
			
			\begin{equation} \label{eq:skminresproj}
			\bu_r(\mu) = \arg \min_{\bw \in U_r} \|\br(\bw; \mu)\|^\bTheta_{U'}.
			\end{equation}
			
			The quasi-optimality of such a projection can be controlled in exactly the same manner as the quasi-optimality of the original minres projection. By defining the constants  
			\begin{subequations}
				\begin{align} 
				&\zeta^\bTheta_{r}(\mu):=  \underset{ \bx  \in \left (\mathrm{span} \{ \bu(\mu) \}+ U_r \right ) \backslash \{ \bnull \}} \min \frac{\| \bA(\mu) \bx  \|^\bTheta_{U'}}{\| \bx \|_U}, \label{eq:skalphar} \\ 
				&\iota^\bTheta_{r}(\mu):=  \underset{ \bx  \in \left (\mathrm{span} \{ \bu(\mu) \}+ U_r \right ) \backslash \{ \bnull \}} \max \frac{ \| \bA(\mu) \bx \|^\bTheta_{U'}}{\| \bx \|_U},\label{eq:skbetar}
				\end{align} 
			\end{subequations}
			we obtain the following result.
			\begin{proposition}  \label{thm:skcea}
				If~$\bu_r(\mu)$ satisfies~\textup {(\ref {eq:skminresproj})} and $\zeta^\bTheta_{r}(\mu)>0$, then
				\begin{equation} \label{eq:skquasi-opt}
				\|\bu(\mu)- \bu_r(\mu) \|_{U} \leq  \frac{\iota^\bTheta_{r}(\mu)}{\zeta^\bTheta_{r}(\mu)} \| \bu(\mu)- \bP_{U_r} \bu(\mu) \|_{U}. 
				\end{equation}
				\begin{proof}
					See Appendix B.
				\end{proof}
			\end{proposition}
			
			{It follows that if $\zeta^\bTheta_{r}(\mu)$ and $\iota^\bTheta_{r}(\mu)$ are almost equal to $\zeta_{r}(\mu)$ and  $\iota_{r}(\mu)$, respectively,} then the quasi-optimality of the original minres projection~\textup {(\ref {eq:minresproj})} shall be almost preserved by its sketched version~\textup {(\ref {eq:skminresproj})}. These properties of $\iota^\bTheta_{r}(\mu)$ and $\zeta^\bTheta_{r}(\mu)$ can be guaranteed under some conditions on $\bTheta$ (see~Proposition\nobreakspace \ref {thm:skminresopt}).
			
			\begin{proposition} \label{thm:skminresopt}
				Define the subspace
				\begin{equation} \label{eq:R_r}
				R_r(U_r; \mu) :=   \mathrm{span} \{ \bR_U^{-1} \br(\bx; \mu) : \bx \in U_r \}.
				\end{equation}
				If $\bTheta$ is a $U \to \ell_2$  $\varepsilon$-subspace embedding for $R_r(U_r; \mu)$, then 
				\begin{equation}\label{eq:skalphabetabounds}
				\sqrt{1-\varepsilon}~\zeta_{r}(\mu) \leq {\zeta^\bTheta_{r}(\mu)} \leq \sqrt{1+\varepsilon}~\zeta_{r}(\mu),\textup{ and } \sqrt{1-\varepsilon}~{\iota_{r}(\mu)} \leq  {\iota^\bTheta_{r}(\mu)} \leq \sqrt{1+\varepsilon}~{\iota_{r}(\mu)}.
				\end{equation}
				\begin{proof}
					See Appendix B. 
				\end{proof}
			\end{proposition}
			An embedding $\bTheta$ satisfying an $U \to \ell_2$ $\varepsilon$-subspace embedding property for the subspace $R_r(U_r; \mu)$ defined in~\textup {(\ref {eq:R_r})}, for all $\mu \in \mathcal{P}$ simultaneously, with high probability, may be generated from an oblivious embedding of sufficiently large size. Note that $\dim({R_r(U_r; \mu)}) \leq r+1$. The number of rows $k$ of the oblivious embedding may be selected a priori using the bounds provided in~\cite{balabanov2019galerkin}, along with a union bound for the probability of success or the fact that $ \bigcup_{\mu  \in \mathcal{P}}  R_r(U_r; \mu)$  is contained in a low-dimensional space. Alternatively, a better value for $k$ can be chosen with a posteriori procedure explained in~Section\nobreakspace \ref {sketchcert}. {Note that if~\textup {(\ref {eq:skalphabetabounds})} is satisfied then the quasi-optimality constants of the minres projection are guaranteed to be preserved up to a small factor depending only on the value of $\varepsilon$.  
				Since $\bTheta$ is here constructed in an oblivious way, the accuracy of random sketching for minres projection can be controlled regardless of the properties of $\bA(\mu)$ (e.g., coercivity, condition number, etc.).} Recall that in~\cite{balabanov2019galerkin} it was revealed that the preservation of the quasi-optimality constants of the classical Galerkin projection by its sketched version is sensitive to {the} operator's properties. More specifically, random sketching can worsen quasi-optimality constants dramatically for non-coercive or ill-conditioned problems. {Consequently, the sketched minres projection should be preferred to the sketched Galerkin projection for such problems.}
			
			\begin{remark}
				{Random sketching is not the only way to construct $\bTheta$ which satisfies the condition in~Proposition\nobreakspace \ref {thm:skminresopt} for all $\mu \in \mathcal{P}$.  In particular, the columns for  a sketching matrix could be chosen as a basis of a low-dimensional space (called empirical test space in~\cite{taddei2019}) that approximates the manifold~$R^*_r (U_r) = \{ \bx \in R_r(U_r, \mu): \|\bx\|_{U}=1,~\mu \in \mathcal{P} \}$ as in~\cite{taddei2019,carlberg2013}. Such a basis could be generated with POD or greedy algorithms. }
				
				{There are several advantages of random sketching over the aforementioned approach, that may also hold for nonlinear problems. {First, the (offline) construction of a low-dimensional approximation of $R^*_r (U_r)$ can be very expensive and become a bottleneck of an algorithm.} Random sketching, on the other hand, requires much lower computational cost.
					The second advantage is the oblivious construction of $\bTheta$ without the knowledge of $U_r$, which can be particularly important when $U_r$ is constructed adaptively (e.g., with a greedy algorithm). Finally, with random sketching it can be sufficient to use $\bTheta$ of small size regardless whether $R^*_r (U_r)$ admits a low-dimensional approximation or not. More specifically, for finite $\mathcal{P}$,  the condition in~Proposition\nobreakspace \ref {thm:skminresopt} can be satisfied (for not too small $\varepsilon$, say $\varepsilon = 0.1$) with high probability by using an oblivious embedding with $k = \mathcal{O}(r+\log{(\#\mathcal{P})})$ rows close to the minimal value $k=r$~\cite{balabanov2019galerkin}. Moreover, a similar online complexity can be attained also for infinite $\mathcal{P}$ by using an additional oblivious embedding as is discussed further in the section.}
			\end{remark}

			The vector of coordinates $\ba_r(\mu) \in \mathbb{K}^{r}$ in the basis $\bU_r$ of the sketched projection $\bu_r(\mu)$ defined by~\textup {(\ref {eq:skminresproj})} may be obtained in a classical way, i.e., by considering a parameter-dependent reduced (normal) system of equations similar to~\textup {(\ref {eq:reduced_system})}. As for the classical approach, this may lead to numerical instabilities during either the online evaluation of the reduced system from the affine expansions or its solution. A remedy is to directly consider      
			\begin{equation} \label{eq:skreduced_eq}
			\ba_r(\mu) = \arg \min_{\bx \in \mathbb{K}^{r} } \|\bA(\mu)\bU_r\bx - \bb(\mu)\|^\bTheta_{U'}= \arg \min_{\bx \in \mathbb{K}^{r} } \| \bV_r^\bTheta(\mu) \bx - \bb^\bTheta(\mu)\|.
			\end{equation} 
			Since the sketched matrix $\bV_r^\bTheta(\mu)$ and vector $\bb^\bTheta(\mu)$ are of rather small sizes, the minimization problem~\textup {(\ref {eq:skreduced_eq})} may be efficiently {formed} (from the precomputed affine expansions) and then solved (e.g., using QR factorization or SVD) in the online stage. 
			\begin{proposition} \label{thm:skminresstab}
				If $\bTheta$ is an $\varepsilon$-embedding for $U_r$, and $\bU_r$ is orthogonal with respect  $\langle \cdot, \cdot \rangle^\bTheta_U$ then the condition number of $\bV^\bTheta_r(\mu)$ is bounded by $\sqrt{\frac{1+\varepsilon}{1-\varepsilon}} \frac{\iota^\bTheta_{r}}{\zeta^\bTheta_{r}}$.
				\begin{proof}
					See Appendix B.
				\end{proof}
			\end{proposition}
			It follows from~Proposition\nobreakspace \ref {thm:skminresstab} (along with~Proposition\nobreakspace \ref {thm:skminresopt}) that considering~\textup {(\ref {eq:skreduced_eq})} can provide better numerical stability than solving reduced systems of equations with standard methods.  Furthermore, since affine expansions of $\bV^\bTheta_r(\mu)$ and $\bb^\bTheta(\mu)$ have less terms than affine expansions of $\bA_r(\mu)$ and $\bb_r(\mu)$ in~\textup {(\ref {eq:reduced_system})}, their {online assembling} should also be much more {stable.}
			
			The online efficiency can be further improved with a procedure similar to the one depicted in~\cite[Section 4.5]{balabanov2019galerkin}. {This procedure can be important for problems requiring rather high size of $\bTheta$, which can happen for algorithms where $U_r$ is chosen adaptively and/or for infinite parameter sets.
				Consider the following sketching matrix
				$$\bPhi = \bGamma \bTheta,$$ 
				where $\bGamma \in \mathbb{K}^{k' \times k}$, $k'< k$, is a small $( \varepsilon', \delta', r+1)$ oblivious  $\ell_2 \to \ell_2$ subspace embedding. {The matrix $\bGamma$ can be taken as a Gaussian matrix  with $k' = \mathcal{O}({\varepsilon'}^{-2}(r + \log(1/\delta'))$ rows (in practice, with a small constant)~\cite{balabanov2019galerkin}.}  It follows that, if $\bTheta$ is a $\varepsilon$-embedding for $R_r(U_r;\mu)$ from~Proposition\nobreakspace \ref {thm:skminresopt}, then {for a single $\mu$}, $\bPhi$ is a $\mathcal{O}(\varepsilon)$-embedding for $R_r(U_r;\mu)$ with probability at least $1-\delta'$. Consequently, {for a single $\mu$},  the minimal residual projection can be accurately estimated by 
				\begin{equation} \label{eq:skminresprojeff}
				\bu_r(\mu) = \arg \min_{\bw \in U_r} \|\br(\bw; \mu)\|^\bPhi_{U'},
				\end{equation}
				with probability at least $1-\delta'$.  The probability of success for several $\mu$ can then be guaranteed with a union bound. }
			
			{
				{It follows that} to improve the online efficiency, we can use a fixed $\bTheta$ which is guaranteed to be a $\varepsilon$-embedding for $R_r(U_r;\mu)$ for all $\mu \in \mathcal{P}$ simultaneously, but consider different realizations of a smaller matrix $\bGamma$ for each particular test set $\mathcal{P}_{\mathrm{test}}$ composed of several parameter values. In this way, in the offline stage a $\bTheta$-sketch of $\bU_r$ can be precomputed and maintained for the online computations. Thereafter, for the given test set $\mathcal{P}_{\mathrm{test}}$ (with the corresponding new realization of $\bGamma$) the affine expansions of small matrices $\bV^\bPhi(\mu):= \bGamma \bV^\bTheta(\mu)$ and $\bb^\bPhi(\mu):= \bGamma \bb^\bTheta(\mu)$ can be efficiently precomputed from the $\bTheta$-sketch in the ``intermediate" online stage. And finally, for each $\mu \in \mathcal{P}_{\mathrm{test}}$, the vector of coordinates of $\bu_r(\mu)$ can be obtained by evaluating $\bV^\bPhi(\mu)$ and $\bb^\bPhi(\mu)$ from just precomputed affine expansions, and solving
				\begin{equation} \label{eq:skreduced_eq_eff}
				\ba_r(\mu) = \arg \min_{\bx \in \mathbb{K}^{r} } \| \bV_r^\bPhi(\mu) \bx - \bb^\bPhi(\mu)\|
				\end{equation} 
				with {a} standard method such as QR factorization,  SVD or (less stable) normal equation.}

			\section{Dictionary-based minimal residual method} \label{dbminres}
			Classical RB method becomes ineffective for parameter-dependent problems for which the solution manifold $\mathcal{M}:= \{ \bu(\mu): \mu \in \mathcal{P} \}$ cannot be well approximated by a single low-dimensional subspace, i.e., its {(linear)} Kolmogorov $r$-width {$d_r(\mathcal{M})$} does not decay rapidly. One can extend the classical RB method by considering a reduced subspace $U_r(\mu)$ depending on a parameter $\mu$.  One way to obtain $U_r(\mu)$ is to use a {$hp$-refinement method} as in~\cite{eftang2010hp,eftang2011hp}, which consists in partitioning the parameter set $\mathcal{P}$ into subsets $\{ \mathcal{P}_i \}^M_{i=1}$ and in associating to each subset $\mathcal{P}_i$ a subspace $U^i_r \subset U$ of dimension at most $r$, therefore resulting in $U_r(\mu)=U^{i}_r$ if $\mu \in \mathcal{P}_i$, $1 \leq i \leq M$. More formally, the $hp$-refinement method aims to approximate $\mathcal{M}$ with a library $\mathcal{L}_{r}:= \{U^{i}_r: 1 \leq i \leq M \}$ of low-dimensional subspaces. For efficiency, the number of subspaces in $\mathcal{L}_{r}$ has to be moderate {(no more than $\mathcal{O}(r^\nu)$ for some small $\nu$, say $\nu=2$ or $3$, which should be dictated by the particular computational architecture).} A nonlinear Kolmogorov $r$-width of $\mathcal{M}$ with a library of $M$ subspaces can be defined as in~\cite{temlyakov1998nonlinear} by
			\begin{equation} \label{eq:nonlinearwidth}
			d_r(\mathcal{M}; M) = \inf_{ \# \mathcal{L}_{r} = M } \sup_{\bu \in \mathcal{M}} \min_{W_r \in \mathcal{L}_{r}} \| \bu - \bP_{W_r} \bu\|_U,
			\end{equation}
			where the infimum is taken over all libraries of $M$ subspaces. {The width~\textup {(\ref {eq:nonlinearwidth})} represents the smallest error of approximation of $\mathcal{M}$ attainable with a library of $M$ $r$-dimensional subspaces. In particular, the approximation $\bP_{U_r(\mu)} \bu(\mu)$ over a parameter-dependent subspace $U_r(\mu)$  associated with a partitioning of $\mathcal{P}$ into $M$ subdomains satisfies}
			\begin{equation}  \label{eq:rwidth}
			d_r(\mathcal{M}; M) \leq {\sup_{\mu \in \mathcal{P}}}{\| \bu(\mu) - \bP_{U_r(\mu)} \bu(\mu) \|_U}.
			\end{equation} 
			Therefore, for the $hp$-refinement method to be effective, the solution manifold is required to be well approximable in terms of the measure $d_r(\mathcal{M}; M)$. 
			
			The $hp$-refinement method may present serious drawbacks: it can be highly sensitive to the parametrization, it can require a large number of subdomains in $\mathcal{P}$ (especially for high-dimensional parameter domains) and it can require computing too many solution samples. These drawbacks can be partially reduced by various modifications of the $hp$-refinement method~\cite{maday2013locally,amsallem2016pebl}, but not circumvented. 
			
			We here propose a dictionary-based method, which can be seen as an alternative to a partitioning of $\mathcal{P}$ for {defining} $U_r(\mu)$, and argue why this method is more natural and can be applied to a larger class of problems.

			\subsection{Dictionary-based approximation} \label{dictapprox}
			
			For each value $\mu$ of the parameter, the basis vectors for $U_r(\mu)$ are selected {online} from a certain dictionary $\mathcal{D}_K$ of $K$ candidate vectors in $U$, $K \geq r$. For {efficiency} of the algorithms in the particular computational environment, the value for $K$ has to be chosen as $\mathcal{O}(r^\nu)$ with a small $\nu$ similarly as the number of subdomains $M$ for the $hp$-refinement method. 
			Let $\mathcal{L}_{r}(\mathcal{D}_K)$ denote the library of all subspaces spanned by $r$ vectors from $\mathcal{D}_K$. A dictionary-based $r$-width is defined as
			\begin{equation} \label{eq:drwidth}
			\sigma_r(\mathcal{M}; K) = \inf_{\# \mathcal{D}_K = K} \sup_{\bu \in \mathcal{M}} \min_{W_r \in \mathcal{L}_{r}(\mathcal{D}_K)} \| \bu - \bP_{W_r} \bu\|_U,
			\end{equation} 
			where the infimum is taken over all subsets $\mathcal{D}_K$ of $U$ with cardinality $\#\mathcal{D}_K=K$. {Clearly, the approximation space $U_r(\mu)$ associated with a dictionary with $K$ vectors satisfies
				$$ \sigma_r(\mathcal{M}; K) \leq {\sup_{\mu \in \mathcal{P}}}{\| \bu(\mu) - \bP_{U_r(\mu)} \bu(\mu) \|_U}. $$}								
			A dictionary $\mathcal{D}_K$ can be efficiently constructed offline {from snapshots} with an adaptive greedy procedure (see~Section\nobreakspace \ref {Dgen}). 
			
			{In general, the performance of the method can be characterized through the approximability of the solution manifold $\mathcal{M}$ in terms of the  $r$-width, and quasi-optimality of the considered $U_r(\mu)$ compared to the best approximation. The dictionary-based approximation can be beneficial over the refinement methods in either of these aspects, which is explained below.}   
			
			It can be easily shown that 
			$$\sigma_r(\mathcal{M}; K) \leq d_r(\mathcal{M}; M), \textup{ for } K \geq rM,$$
			{holds for all $\mathcal{M}$.}
			Therefore{,} if a solution manifold can be well approximated with a partitioning of the parameter domain into $M$ subdomains each associated with a subspace of dimension $r$, then it should also be well approximated with a dictionary of size $K=rM$, which {implies a similar computational cost}.  {The converse statement, however, is not true. There exist manifolds $\mathcal{M}$ satisfying\footnote{{For instance, take  $\mathcal{M} = \mathcal{L}_r(\{\bv_1, \hdots, \bv_{K} \}$) with linearly independent vectors $\bv_i \in U$, $1 \leq i \leq K$. Then we clearly have $\sigma_r(\mathcal{M}; K) = d_r(\mathcal{M}; \tbinom{K}{r}) = 0$ and $d_r(\mathcal{M}; \tbinom{K}{r}-1) > 0$. }}}
			$$ {d_r(\mathcal{M}; \tbinom{K}{r})  \leq \sigma_r(\mathcal{M}; K) < {d}_r(\mathcal{M}; \tbinom{K}{r}-1).} $$
			Consequently, to obtain a decay of $d_r(\mathcal{M}; M)$ with $r$ similar to the decay of $\sigma_r(\mathcal{M}; r^\nu)$, {it may} be required to use $M$ which depends exponentially on $r$.

			The great potential of the dictionary-based approximation {can be justified} by important properties of the dictionary-based $r$-width given in~Proposition\nobreakspace \ref {prop:widthsuper} and\nobreakspace Corollary\nobreakspace \ref {prop:corwidthsuper}.
			
			\begin{proposition} \label{prop:widthsuper}
				Let $\mathcal{M}$ be obtained by the superposition of parameter-dependent vectors:	
				\begin{equation} \label{eq:widthsuper_sman}
				{\mathcal{M}} = \{ \sum^l_{i=1} \bu^{(i)}(\mu): \mu \in \mathcal{P}\}, 
				\end{equation} 
				where $\bu^{(i)}(\mu) \in U,~~ i =1, \hdots, l $.
				Then, we have 
				\begin{equation} \label{eq:widthsuper}
				\sigma_{{r}}({\mathcal{M}}; {K}) \leq \sum^l_{i=1} \sigma_{r_i}(\mathcal{M}^{(i)}; K_i), 
				\end{equation}
				with ${r} = \sum^l_{i=1} r_i,~K= \sum^l_{i=1} K_i$ and
				\begin{equation} \label{eq:widthsuper_iman}
				\mathcal{M}^{(i)} = \{ \bu^{(i)}(\mu): \mu \in \mathcal{P}\}.
				\end{equation} 
				\begin{proof}	
					See Appendix B. 
				\end{proof}
			\end{proposition}
			\begin{corollary}[Approximability of a superposition of solutions] \label{prop:corwidthsuper}
				Consider several solution manifolds $\mathcal{M}^{(i)}$ defined by~\textup {(\ref {eq:widthsuper_iman})}, $1 \leq i \leq l$, and the resulting manifold $\mathcal{M}$ defined by~\textup {(\ref {eq:widthsuper_sman})}. Let $c$, $C$, $\alpha$, $\beta$ and $\gamma$ be some {positive} constants. The following properties hold. 
				\begin{enumerate}[(i)]	
					\item If $\sigma_{r}(\mathcal{M}^{(i)}; c r^{\nu}) \leq C r^{-\alpha}$, then $\sigma_{r}(\mathcal{M}; c l^{1-\nu} r^{\nu}) \leq C l^{1+\alpha} r^{-\alpha}$,
					\item If $\sigma_{r}(\mathcal{M}^{(i)}; c r^{\nu}) \leq C e^{-\gamma r^{\beta}}$, then $\sigma_{r}(\mathcal{M};  c l^{1-\nu} r^{\nu}) \leq  C l e^{-\gamma l^{-\beta} r^{\beta}}$.
				\end{enumerate}
			\end{corollary}
			
			From~Proposition\nobreakspace \ref {prop:widthsuper} and\nobreakspace Corollary\nobreakspace \ref {prop:corwidthsuper} it follows that the approximability of the solution manifold in terms of the dictionary-based $r$-width is preserved under the superposition operation. In other words, if the dictionary-based $r$-widths of manifolds $\mathcal{M}^{(i)}$ {have} a certain decay with $r$ (e.g., exponential or algebraic), by using dictionaries containing $K=\mathcal{O}(r^\nu)$ vectors, then the type of decay is preserved by their superposition (with the same rate for the algebraic decay). This property can be crucial for problems where the solution is a superposition of several contributions (possibly unknown), which is a quite typical situation. {For instance, we have such a situation for PDEs with multiple transport phenomena.} {A similar property as~\textup {(\ref {eq:widthsuper})} also holds for the classical linear Kolmogorov $r$-width $d_{{r}}({\mathcal{M}})$.} Namely, we have
			\begin{equation} 
			d_{{r}}({\mathcal{M}}) \leq \sum^l_{i=1} d_{r_i}(\mathcal{M}^{(i)}), 
			\end{equation}
			with ${r} = \sum^l_{i=1} r_i$. This relation follows immediately from~Proposition\nobreakspace \ref {prop:widthsuper} and the fact that   $d_{{r}}({\mathcal{M}})=\sigma_r({\mathcal{M}}; {1})$. For the nonlinear Kolmogorov $r$-width~\textup {(\ref {eq:rwidth})}, however, the relation
			\begin{equation} 
			d_{{r}}({\mathcal{M}},M) \leq \sum^l_{i=1} d_{r_i}(\mathcal{M}^{(i)},M^{(i)}), 
			\end{equation}
			where ${r} = \sum^l_{i=1} r_i$, {holds under the condition that} $M \geq \prod_{i=1}^l {M}^{(i)}$. In general, the preservation of the type of decay with $r$ of $d_{{r}}({\mathcal{M}},M)$, by using libraries with $M=\mathcal{O}(r^\nu)$ terms, may not be guaranteed. {It} can require libraries with much larger numbers of $r$-dimensional subspaces than $\mathcal{O}(r^\nu)$, namely $M=\mathcal{O}(r^{l \nu})$ subspaces. 
			
			Another advantage of the dictionary-based method is its weak sensitivity to the parametrization of the manifold $\mathcal{M}$, in contrast to the $hp$-refinement method, for which a bad choice of parametrization can result in {approximations with too many local reduced subspaces.}  Indeed, the solution map $\mu \to \bu(\mu)$ is often expected to have certain properties (e.g., symmetries {or anisotropies}) that yield the existence of a better parametrization of $\mathcal{M}$ than the one proposed by the user. 
			Finding a good parametrization of the solution manifold may require a deep intrusive analysis of the problem, and is therefore usually an {intractable} task. On the other hand, our dictionary-based methodology provides a reduced approximation subspace for each vector from $\mathcal{M}$ regardless of the chosen parametrization. 
			
			\subsection{Sparse minimal residual approximation} \label{sminres}
			
			Here we assume to be given a dictionary $\mathcal{D}_K$ of $K$ vectors in $U$. Ideally, for each $\mu$, $\bu(\mu)$ should be approximated by orthogonal projection onto a subspace $W_r(\mu)$ that minimizes 
			\begin{equation} \label{eq:Vmproj}
			\|\bu(\mu) - \bP_{W_r(\mu)} \bu(\mu)\|_U
			\end{equation}
			over the library $\mathcal{L}_r(\mathcal{D}_K)$. The selection of the optimal subspace requires operating with the exact solution $\bu(\mu)$ which is  prohibited. Therefore, the reduced approximation space $U_r(\mu) \in \mathcal{L}_r(\mathcal{D}_K)$ and the associated approximate solution $\bu_r(\mu) \in U_r(\mu)$ are {obtained by residual norm minimization. The minimization of the residual norm, i.e., solving
				\begin{equation} \label{eq:Vmminresprojopt}
				\min_{W_r \in \mathcal{L}_r(\mathcal{D}_K)} \min_{\bw \in W_r } \|\br(\bw; \mu)\|_{U'},
				\end{equation}
				is a combinatorial problem that can be intractable in practice. From a practical perspective we will assume that only a suboptimal solution can be obtained.\footnote{The common approaches for solving sparse least-squares problems include greedy algorithms and LASSO. An analysis of these methods can be found in~\cite{devore2009nonlinear}.} Therefore we relax~\textup {(\ref {eq:Vmminresprojopt})} to the following problem: find $\bu_r(\mu)$ such that
				\begin{equation} \label{eq:Vmminresproj}
				\|\br(\bu_r(\mu); \mu)\|_{U'}   \leq D \min_{W_r \in \mathcal{L}_r(\mathcal{D}_K)} \min_{\bw \in W_r } \|\br(\bw; \mu)\|_{U'} + \tau \|\bb(\mu)\|_{U'},
				\end{equation}
				holds for a small constant $D \geq 1$ and tolerance  $\tau \geq 0$.} The solution $\bu_r(\mu)$ from~\textup {(\ref {eq:Vmminresproj})} shall be referred to as sparse minres approximation (relatively to {the} dictionary $\mathcal{D}_K$).   The quasi-optimality of this approximation {in the norm $\| \cdot \|_U$} can be characterized with the following parameter-dependent constants:    
			\begin{subequations} \label{eq:zetaiotarK}
				\begin{align} 
				&\zeta_{r,K}(\mu):=  \min_{W_r \in \mathcal{L}_r(\mathcal{D}_K)}  \min_{ \bx  \in \left (\mathrm{span} \{ \bu(\mu) \}+ W_r \right ) \backslash \{ \bnull \}} \frac{\| \bA(\mu) \bx  \|_{U'}}{\| \bx \|_U}, \label{eq:Vmalphar} \\ 
				&\iota_{r,K}(\mu):=  \max_{W_r \in \mathcal{L}_r(\mathcal{D}_K)}  \max_{ \bx  \in \left (\mathrm{span} \{ \bu(\mu) \}+ W_r \right ) \backslash \{ \bnull \}} \frac{\| \bA(\mu) \bx  \|_{U'}}{\| \bx \|_U}. \label{eq:Vmbetar}
				\end{align} 
			\end{subequations}
			{In general,} one can bound $\zeta_{r,K}(\mu)$ and $\iota_{r,K}(\mu)$ by the minimal and the maximal singular values $\alpha(\mu)$ and $\beta(\mu)$ of $\bA(\mu)$. Observe also that for $K=r$ (i.e., when the library $\mathcal{L}_r(\mathcal{D}_K) = \{ U_r\}$ has a single subspace) we have $\zeta_{r,K}(\mu) = \zeta_{r}(\mu)$ and $\iota_{r,K}(\mu) = \iota_{r}(\mu)$. 
			
			\begin{proposition}  \label{thm:Vmcea}
				Let $\bu_r(\mu)$ {satisfy}~\textup {(\ref {eq:Vmminresproj})} and $\zeta_{r,K}(\mu)>0$. Then
				\begin{equation} \label{eq:sminres_quasi-opt}
				{\| \bu(\mu)- \bu_r(\mu) \|_{U} \leq  \frac{\iota_{r,K}(\mu)}{\zeta_{r,K}(\mu)}(D \min_{W_r \in \mathcal{L}_r(\mathcal{D}_K)} \| \bu(\mu) - \bP_{W_r} \bu(\mu)\|_U +\tau \|\bu(\mu)\|_{U} ). }
				\end{equation}
				\begin{proof}
					See Appendix B. 
				\end{proof}
			\end{proposition}
			Let $\bU_K \in \mathbb{K}^{n\times K}$ be a matrix whose columns are the vectors in the dictionary $\mathcal{D}_K$ and $\ba_{r,K} (\mu) \in \mathbb{K}^{K}$, with $\| \ba_{r,K}(\mu) \|_0 \leq r$,  be the $r$-sparse vector of coordinates of $\bu_r(\mu)$ in the dictionary, i.e., $\bu_r(\mu)= \bU_K\ba_{r,K}(\mu)$.  The vector of coordinates associated with {$\bu_r(\mu)$ satisfying}~\textup {(\ref {eq:Vmminresproj})} is {an} approximate solution to the following parameter-dependent sparse least-squares problem:
			\begin{equation} \label{eq:sparseminresproj}
			\min_{\bz \in \mathbb{K}^{K} } \|\bA(\mu) \bU_K \bz - \bb(\mu)\|_{U'}, \textup{ subject to } \| \bz \|_0 \leq r.
			\end{equation}
			For each $\mu \in \mathcal{P}$ an approximate solution to problem~\textup {(\ref {eq:sparseminresproj})} can be obtained with a standard greedy algorithm depicted in~Algorithm\nobreakspace \ref {alg:omp}. It selects the nonzero entries of $\ba_{r,K}(\mu)$ one by one to minimize the residual. The algorithm corresponds to either the orthogonal greedy {(also called Orthogonal Matching Pursuit in signal processing community~\cite{tropp2007signal})} or stepwise projection algorithm (see~\cite{devore2009nonlinear}) depending on whether the (optional) Step 8 (which is  the orthogonalization of $\{ \bv_j(\mu) \}^K_{j=1}$ with respect to $V_i(\mu)$) is considered. It should be noted that performing Step 8 can be of great importance due to possible high {mutual coherence} of the (mapped) dictionary $\{ \bv_j(\mu) \}^K_{j=1}$. Algorithm\nobreakspace \ref {alg:omp} is provided in a conceptual form. A more sophisticated procedure can be derived to improve the online efficiency (e.g., considering precomputed affine expansions of $\bA_K(\mu) := \bU_K^{\mathrm{H}} \bA(\mu)^{\mathrm{H}} \bR_U^{-1} \bA(\mu)\bU_K \in \mathbb{K}^{K\times K}$ and $\bb_K(\mu)= \bU_K^{\mathrm{H}} \bA(\mu)^{\mathrm{H}} \bR_U^{-1} \bb(\mu) \in \mathbb{K}^K$, updating the residual using a Gram-Schmidt procedure, etc). Algorithm\nobreakspace \ref {alg:omp}, even when efficiently implemented, can still require heavy computations in both the offline and online stages, and be numerically unstable. One of the contributions of this paper is a drastic improvement of its efficiency and stability by random sketching, thus making the use of dictionary-based model reduction feasible in practice.
			
			{\begin{algorithm}[h] \caption{Orthogonal greedy algorithm} \label{alg:omp}
					\begin{algorithmic}
						\Statex{\textbf{Given:} $\mu$, $\bU_K=[\bw_j]^K_{j=1}$, $\bA(\mu)$, $\bb(\mu)$, $\tau$, $r$.}
						\Statex{\textbf{Output}:  index set $\Lambda_r(\mu)$, the coordinates $\ba_{r}(\mu)$ of $\bu_r(\mu)$ on the basis $\{ \bw_j \}_{j \in \Lambda_r(\mu)} $.}
						\STATE{1. Set $i:=0$, $U_{0}(\mu) = \{ \bnull \}$, $\bu_{0}(\mu) = \bnull$, $\Lambda_0(\mu) = \emptyset$, $\widetilde{\Delta}_0(\mu) = \infty$.}
						\STATE{2. Set $[\bv_1(\mu), ..., \bv_K(\mu)]:= \bA(\mu) \bU_K$ and normalize the vectors $\bv_j(\mu)$, $1 \leq j \leq K$.}
						\WHILE{ $\widetilde{\Delta}_i(\mu) \geq \tau$ and $i < r$} 
						\STATE{3. Set $i:=i+1$.}
						\STATE{4. Find the index $p_i \in \{1, \hdots, K\}$ which maximizes $| \langle \bv_{p_i}(\mu), \br(\bu_{i-1}(\mu); \mu)  \rangle_{U'}|$. }
						\STATE{5. Set $\Lambda_i(\mu) := \Lambda_{i-1}(\mu) \cup \{ p_i\}$.}
						\STATE{6. Solve~\textup {(\ref {eq:reduced_system})} with a reduced matrix $\bU_i(\mu) = [\bw_j ]_{j \in \Lambda_i(\mu)}$ 
							and obtain coordinates $\ba_{i}(\mu)$.}		
						\STATE{7. Compute error bound $\widetilde{\Delta}_i(\mu)$ of $\bu_{i}(\mu) = \bU_i(\mu) \ba_{i}(\mu)$.}
						\STATE{8. (Optional) Set $\bv_j(\mu) := \bv_j(\mu) - \bP_{V_i(\mu)} \bv_j(\mu)$, where $\bP_{V_i(\mu)}$ is the orthogonal projector\\~~~~~~ on $V_i(\mu):= \mathrm{span}(\{\bv_p(\mu) \}_{p \in \Lambda_i(\mu)})$, and normalize $\bv_j(\mu)$, $j \in \{1,2,\hdots,K\}\backslash\Lambda_i(\mu)$.}		
						\ENDWHILE
					\end{algorithmic}
			\end{algorithm}}
			
			\subsection{Sketched sparse minimal residual approximation} \label{ssminres}
			Let $\bTheta \in \mathbb{K}^{k \times n}$ be a certain $U \to \ell_2$ subspace embedding. {A} sparse minres approximation defined by~\textup {(\ref {eq:Vmminresproj})}, associated with dictionary $\mathcal{D}_K$, can be estimated by {solving the following problem: find $\bu_r(\mu) \in U_r(\mu) \in \mathcal{L}_r(\mathcal{D}_K)$, such that
				\begin{equation} \label{eq:skVmminresproj}
				\|\br(\bu_r(\mu); \mu)\|^\bTheta_{U'}   \leq D \min_{W_r \in \mathcal{L}_r(\mathcal{D}_K)} \min_{\bw \in W_r } \|\br(\bw; \mu)\|^\bTheta_{U'} + \tau \|\bb(\mu)\|^\bTheta_{U'}.
				\end{equation}
			} In order to characterize the quasi-optimality of the sketched sparse minres approximation defined by~\textup {(\ref {eq:skVmminresproj})} we introduce the following parameter-dependent values
			\begin{subequations}
				\begin{align} 
				&\zeta^\bTheta_{r,K}(\mu):=  \min_{W_r \in \mathcal{L}_r(\mathcal{D}_K)}  \min_{ \bx  \in \left (\mathrm{span} \{ \bu(\mu) \}+ W_r \right ) \backslash \{ \bnull \}} \frac{\| \bA(\mu) \bx  \|^\bTheta_{U'}}{\| \bx \|_U}, \label{eq:sksparsealphar} \\ 
				&\iota^\bTheta_{r,K}(\mu):=  \max_{W_r \in \mathcal{L}_r(\mathcal{D}_K)}  \max_{ \bx  \in \left (\mathrm{span} \{ \bu(\mu) \}+ W_r \right ) \backslash \{ \bnull \}} \frac{\| \bA(\mu) \bx  \|^\bTheta_{U'}}{\| \bx \|_U}. \label{eq:sksparsebetar}
				\end{align} 
			\end{subequations}
			
			Observe that choosing $K=r$ yields $\zeta^\bTheta_{r,K}(\mu) = \zeta^\bTheta_{r}(\mu)$ and $\iota^\bTheta_{r,K}(\mu) = \iota^\bTheta_{r}(\mu)$. 
			
			\begin{proposition} \label{thm:sksparsecea}
				If~$\bu_r(\mu)$ satisfies~\textup {(\ref {eq:skVmminresproj})} and $\zeta^\bTheta_{r,K}(\mu)>0$, then {
					\begin{equation} \label{eq:sksparsequasi-opt}
					\| \bu(\mu)- \bu_r(\mu) \|_{U} \leq  \frac{\iota^\bTheta_{r,K}(\mu)}{\zeta^\bTheta_{r,K}(\mu)} ( D \min_{W_r \in \mathcal{L}_r(\mathcal{D}_K)}\| \bu(\mu) - \bP_{W_r} \bu(\mu)\|_U + \tau \|\bu(\mu)\|_{U}). 
					\end{equation}}
				\begin{proof}
					See Appendix B.
				\end{proof}
			\end{proposition}
			It follows from~Proposition\nobreakspace \ref {thm:sksparsecea} that the quasi-optimality of the sketched sparse minres approximation can be controlled by bounding the constants $\zeta^\bTheta_{r,K}(\mu)$ and $\iota^\bTheta_{r,K}(\mu)$. 
			
			\begin{proposition} \label{ssminresopt}	
				If $\bTheta$ is a $U \to \ell_2$  $\varepsilon$-embedding for every subspace $R_r(W_r; \mu)$, defined by~\textup {(\ref {eq:R_r})}, with $W_r \in \mathcal{L}_r(\mathcal{D}_K)$, then 
				\begin{subequations}
					\begin{align} 
					&\sqrt{1-\varepsilon}{\zeta_{r,K}(\mu)} \leq {\zeta^\bTheta_{r,K}(\mu)} \leq \sqrt{1+\varepsilon} {\zeta_{r,K}(\mu)},
					\intertext{and } 
					&\sqrt{1-\varepsilon} {{\iota_{r,K}(\mu)}} \leq  {\iota^\bTheta_{r,K}(\mu)} \leq \sqrt{1+\varepsilon}{{\iota_{r,K}(\mu)}}.
					\end{align} 
				\end{subequations}
				
				\begin{proof}
					See Appendix B.
				\end{proof}
				
			\end{proposition}
			By~Definition\nobreakspace \ref {def:oblepsilon_embedding} and the union bound for the probability of success, if $\bTheta$ is a $(\varepsilon, \binom{K}{r}^{-1} \delta, r+1 )$ oblivious $U \to \ell_2$ subspace embedding, then $\bTheta$ satisfies the assumption of~Proposition\nobreakspace \ref {ssminresopt}  with probability {of} at least  $1-\delta$. The sufficient number of rows for $\bTheta$ may be chosen a priori with the bounds provided in~\cite{balabanov2019galerkin} or adaptively with a procedure from~Section\nobreakspace \ref {sketchcert}. For the Gaussian embeddings the a priori bounds are logarithmic in $K$ and $n$, and proportional to $r$. For SRHT they are also logarithmic in $K$ and $n$, but proportional to $r^2$ (although in practice SRHT performs equally well as the Gaussian distribution). Moreover, if $\mathcal{P}$ is a finite set, an oblivious embedding $\bTheta$ which satisfies the hypothesis of~Proposition\nobreakspace \ref {ssminresopt} for all $\mu \in \mathcal{P}$, simultaneously, may be chosen using the above considerations and a union bound for the probability of success. Alternatively, for an infinite set $\mathcal{P}$, $\bTheta$ can be chosen as an $\varepsilon$-embedding for a collection of low-dimensional subspaces $R^*_r(W_r)$ (which can be obtained from the affine expansions of $\bA(\mu)$ and $\bb(\mu)$) each containing $\bigcup_{\mu  \in \mathcal{P}} R_r(\mu; W_r)$ and associated with a subspace $W_r$ of $\mathcal{L}_r(\mathcal{D}_K)$. Such an embedding can be again generated in an oblivious way by considering~Definition\nobreakspace \ref {def:oblepsilon_embedding} and a union bound for the probability of success. 
			
			From an algebraic point of view, the optimization problem~\textup {(\ref {eq:skVmminresproj})} can be formulated as the following sparse least-squares problem:
			\begin{equation} \label{eq:algsksparseminresproj}
			\min_{\substack{\bz \in \mathbb{K}^{K} \\ \| \bz \|_0 \leq r }} \|\bA(\mu) \bU_K \bz - \bb(\mu)\|^\bTheta_{U'} =  \min_{\substack{\bz \in \mathbb{K}^{K} \\ \| \bz \|_0 \leq r}} \|\bV_K^\bTheta(\mu) \bz - \bb^\bTheta(\mu)\|,
			\end{equation}
			where $\bV_K^\bTheta(\mu)$ and $\bb^\bTheta(\mu)$ are the components~\textup {(\ref {eq:sketch})} of the $\bTheta$-sketch of $\bU_K$ (a matrix whose columns are the vectors in $\mathcal{D}_K$). {An approximate} solution $\ba_{r,K}(\mu)$ of~\textup {(\ref {eq:algsksparseminresproj})} is the $r$-sparse vector of the coordinates of $\bu_r(\mu)$. We observe that~\textup {(\ref {eq:algsksparseminresproj})} is simply an approximation of a small vector $\bb^\bTheta(\mu)$ with a dictionary composed from column vectors of $\bV_K^\bTheta(\mu)$. Therefore, unlike the original sparse least-squares problem~\textup {(\ref {eq:sparseminresproj})}, the solution to its sketched version~\textup {(\ref {eq:algsksparseminresproj})} can be efficiently approximated with standard tools in the online stage. For instance, we can use~Algorithm\nobreakspace \ref {alg:omp} replacing $\langle \cdot, \cdot \rangle_{U'}$ with $\langle \cdot, \cdot  \rangle^\bTheta_{U'}$. Clearly, in~Algorithm\nobreakspace \ref {alg:omp} the inner products $\langle \cdot, \cdot \rangle^\bTheta_{U'}$ should be efficiently evaluated from $\bV_K^\bTheta(\mu)$ and $\bb^\bTheta(\mu)$. For this a $\bTheta$-sketch of $\bU_K$ can be precomputed in the offline stage and then used for online evaluation of $\bV_K^\bTheta(\mu)$ and $\bb^\bTheta(\mu)$ for each value of the parameter. Another way to obtain an approximate solution to the sketched sparse least-squares problem~\textup {(\ref {eq:algsksparseminresproj})} is to use LASSO or similar methods. 
			
			{Let us now characterize the algebraic stability (i.e., sensitivity to round-off errors) of the (approximate) solution of~\textup {(\ref {eq:algsksparseminresproj})}. The solution of~\textup {(\ref {eq:algsksparseminresproj})} is essentially obtained from the following least-squares problem    
				\begin{equation} \label{eq:stabalgsksparseminresproj}
				\min_{{\bx \in \mathbb{K}^{r}}} \|\bV_r^\bTheta(\mu) \bx - \bb^\bTheta(\mu)\|,
				\end{equation}
				where $\bV_r^\bTheta(\mu)$ is a matrix whose column vectors are (adaptively) selected from the columns of $\bV_K^\bTheta(\mu)$. The algebraic stability of this problem can be measured by the condition number of $\bV_r^\bTheta(\mu)$.
				The minimal and the maximal singular values of $\bV_r^\bTheta(\mu)$ can be bounded using the parameter-dependent coefficients $\iota^\bTheta_{r,K}(\mu)$, $\zeta^\bTheta_{r,K}(\mu)$ and the so-called restricted isometry property (RIP) constants associated with the dictionary $\mathcal{D}_K$, which are defined by
				\begin{equation} \label{eq:ric}
				\Sigma^\mathrm{min}_{r,K}:=  \min_{\substack{\bz \in \mathbb{K}^{K} \\ \| \bz \|_0 \leq r }} \frac{\| \bU_K \bz  \|_{U}}{\| \bz \|}, ~~~~ \Sigma^\mathrm{max}_{r,K}:=  \max_{\substack{\bz \in \mathbb{K}^{K} \\ \| \bz \|_0 \leq r }} \frac{\| \bU_K \bz  \|_{U}}{\| \bz \|}. 
				\end{equation} 
				\begin{proposition} \label{prop:stabalgsksparseminresproj}
					The minimal singular value of $\bV_r^\bTheta(\mu)$ in~\textup {(\ref {eq:stabalgsksparseminresproj})} is bounded below by $\zeta^\bTheta_{r,K}(\mu) \Sigma^\mathrm{min}_{r,K}$, while the maximal singular value of $\bV_r^\bTheta(\mu)$  is bounded above by $\iota^\bTheta_{r,K}(\mu) \Sigma^\mathrm{max}_{r,K}$.
					\begin{proof}
						See Appendix B. 
					\end{proof}
				\end{proposition}
				The RIP constants quantify the linear dependency of the dictionary vectors. For instance, it is easy to see that for a dictionary composed of orthogonal unit vectors we have $\Sigma^\mathrm{min}_{r,K}=\Sigma^\mathrm{max}_{r,K}=1$. From~Proposition\nobreakspace \ref {prop:stabalgsksparseminresproj}, one can deduce the maximal level of degeneracy of $\mathcal{D}_K$ for which the sparse optimization problem~\textup {(\ref {eq:algsksparseminresproj})} remains sufficiently stable.}
			
			\begin{remark}
				{In general, our approach is more stable than the algorithms from~\cite{dihlmann2012online,kaulmann2013online}. These algorithms basically proceed with the solution of the reduced system of equations $\bA_r(\mu)\ba_r(\mu) = \bb_r(\mu)$, where $\bA_r(\mu)=\bU_r(\mu)^\mathrm{H} \bA(\mu) \bU_r(\mu)$, with $\bU_r(\mu)$ being a matrix whose column vectors are selected from the column vectors of $\bU_K$. In this case, the bounds for the minimal and the maximal singular values of $\bA_r(\mu)$ are proportional to the squares of the minimal and the maximal singular values of $\bU_r(\mu)$, which implies a quadratic dependency on the RIP constants $\Sigma^\mathrm{min}_{r,K}$ and $\Sigma^\mathrm{max}_{r,K}$. On the other hand, with (sketched) minres methods the dependency of the singular values of the reduced matrix $\bV^\bTheta(\mu)$ on $\Sigma^\mathrm{min}_{r,K}$ and $\Sigma^\mathrm{max}_{r,K}$ is only linear (see~Proposition\nobreakspace \ref {prop:stabalgsksparseminresproj}). Consequently, our methodology provides an improvement of not only efficiency but also numerical stability for problems with high linear dependency of dictionary vectors.  
				}
			\end{remark}
			
			{Similarly to the sketched minres projection, a better online efficiency can be obtained  by introducing
				$$\bPhi = \bGamma \bTheta,$$ 
				where $\bGamma \in \mathbb{K}^{k' \times k}$, $k'< k$, is a small $( \varepsilon', \binom{K}{r}^{-1} \delta', r+1)$ oblivious  $\ell_2 \to \ell_2$ subspace embedding, and by computing $\bu_r(\mu)$ such that
				\begin{equation} \label{eq:sksparseminresproj}
				\|\br(\bu_r(\mu); \mu)\|^\bPhi_{U'}   \leq D \min_{W_r \in \mathcal{L}_r(\mathcal{D}_K)} \min_{\bw \in W_r } \|\br(\bw; \mu)\|^\bPhi_{U'} + {\tau \| \bb(\mu) \|^\bPhi_{U'}} ,
				\end{equation}
				It follows that, for a single $\mu$, the accuracy {(and the stability)} of a solution $\bu_r(\mu)$ of~\textup {(\ref {eq:sksparseminresproj})} is {almost the same} as a solution of~\textup {(\ref {eq:skVmminresproj})} with probability at least $1-\delta'$.} In an algebraic setting{,}~\textup {(\ref {eq:sksparseminresproj})} can be expressed as 
			
			\begin{equation} \label{eq:algsksparseminresprojeff}
			\min_{\substack{\bz \in \mathbb{K}^{K} \\ \| \bz \|_0 \leq r }} \|\bV_K^\bPhi(\mu) \bz - \bb^\bPhi(\mu)\|,
			\end{equation}
			whose {(approximate)} solution $\ba_r(\mu)$ is a $r$-sparse vector of coordinates of $\bu_r(\mu)$. {Such a} solution can be computed with~Algorithm\nobreakspace \ref {alg:omp} by replacing $\langle \cdot, \cdot \rangle_{U'}$ with $\langle \cdot, \cdot  \rangle^\bPhi_{U'}$.
			An efficient procedure for evaluating the coordinates of a sketched dictionary-based approximation on a test set $\mathcal{P}_{\mathrm{test}}$ from the $\bTheta$-sketch of $\bU_K$ is provided in~Algorithm\nobreakspace \ref {alg:skomp}. Algorithm\nobreakspace \ref {alg:skomp} uses a residual-based sketched error estimator from~\cite{balabanov2019galerkin} defined by
			\begin{equation} \label{eq:reserrorind}
			\widetilde{\Delta}_i(\mu) = \Delta^\bPhi(\bu_i(\mu); \mu) = \frac{\|\br(\bu_i(\mu); \mu) \|^\bPhi_{U'}}{\eta(\mu)},
			\end{equation}
			where $\eta(\mu)$ is a {(online)} computable lower bound/{estimator} of the minimal singular value of matrix $\bA(\mu)$ {seen as operator from $U$ to $U'$.} \footnote{{In many applications it can be sufficient to take $\eta(\mu)$ as constant, e.g., equal to the (approximation of the) smallest of singular values of $\bA(\mu)$ on a training set. Sharper bounds $\eta(\mu)$ can be obtained from theoretical approaches or the successive constraint method~\cite{haasdonk2017reduced,huynh2007successive,						Rozza2007,HUYNH20101963,CHEN2009,CHEN20081295}.}  } 				
			Let us underline the importance of performing Step 8 (orthogonalization of the dictionary vectors with respect to the previously selected basis vectors), for problems with ``degenerate'' dictionaries (with high mutual coherence).  It should be noted that at Steps 7 and 8 we use a Gram-Schmidt procedure for orthogonalization because of its simplicity and efficiency, whereas a modified Gram-Schmidt algorithm could provide better accuracy.  It is also important to note that~Algorithm\nobreakspace \ref {alg:skomp} satisfies a basic consistency property in the sense that it exactly recovers the vectors from the dictionary with high probability. 
			
			If $\mathcal{P}$ is a finite set, then the theoretical bounds for Gaussian matrices and the empirical experience for SRHT state that choosing $k=\mathcal{O}(r\log{K} +{\log{(1/\delta)}}+\log{(\#\mathcal{P}))}$ and $k'=\mathcal{O}(r\log{K} +{\log{(1/\delta)}}+\log{(\#\mathcal{P}_\mathrm{test})})$ in~Algorithm\nobreakspace \ref {alg:skomp} yield a quasi-optimal solution to~\textup {(\ref {eq:Vmproj})} for all $\mu \in \mathcal{P}_\mathrm{test}$ with probability at least $1-\delta$. Let us neglect the logarithmic summands.  Assuming $\bA(\mu)$ and $\bb(\mu)$ admit affine representations with $m_A$ and $m_b$ terms, it follows that the online complexity and memory consumption of~Algorithm\nobreakspace \ref {alg:skomp} is only $\mathcal{O}((m_A K + m_b) r \log{K} + r^2 K \log{K}) \# \mathcal{P}_\mathrm{test}$ and $\mathcal{O}((m_A K +m_b) r \log{K})$, respectively. 
			The quasi-optimality for infinite $\mathcal{P}$ can be ensured with high probability by increasing $k$ to  $\mathcal{O}(r^*\log{K} +\log{\delta})$, where $r^*$ is the maximal dimension of subspaces $R^*_r(W_r)$ containing $\bigcup_{\mu  \in \mathcal{P}}  R_r(\mu; W_r)$ with $W_r \in \mathcal{L}_r(\mathcal{D}_K)$. This shall increase the memory consumption by a factor of $r^*/r$ but should have a negligible effect (especially for large $\mathcal{P}_\mathrm{test}$) on the complexity, which is mainly characterized by the size of $\bPhi$.  Note that for {parameter-separable} problems we have $r^* \leq m_A r+m_b$.

			\begin{algorithm} \caption{Efficient/stable sketched orthogonal greedy algorithm} \label{alg:skomp}
				\begin{algorithmic}
					\STATE{\textbf{Given:} $\mathcal{P}_{\mathrm{test}}$, $\bTheta$-sketch of $\bU_K = [\bw_j]^K_{j=1}$, $\tau$, $r$.}
					\STATE{\textbf{Output}: index set $\Lambda_r(\mu)$, the coordinates $\ba_{r}(\mu)$ of $\bu_r(\mu)$ on basis $\{ \bw_j \}_{j \in \Lambda_r(\mu)}$,\\  
						~~~~~~~~~~~~~and the error indicator $ \Delta^\bPhi(\bu_r(\mu);\mu)$ for each $\mu \in \mathcal{P}_{\mathrm{test}} $.} 
					\STATE{1. Generate $\bGamma$ and evaluate the affine factors of $\bV^{\bPhi}_K(\mu):= \bGamma \bV^{\bTheta}_K(\mu)$ and $\bb^\bPhi(\mu):=\bGamma \bb^{\bTheta}(\mu)$.}
					
					\FOR{$\mu \in \mathcal{P}_{\mathrm{test}}$}
					\STATE{2. Evaluate $[\bv^\bPhi_1(\mu), \dots, \bv^\bPhi_K(\mu)] :=\bV^{\bPhi}_K(\mu)$ and $\bb^{\bPhi}(\mu)$ from the affine expansions \\  
						~~~~~~~and normalize $\bv^\bPhi_j(\mu), 1 \leq j \leq K$.}
					
					\STATE{3. Set $i=0$, obtain $\eta(\mu)$ in~\textup {(\ref {eq:reserrorind})}, set $\Lambda_0(\mu) = \emptyset$, $\br^\bPhi_{0}(\mu) := \bb^\bPhi(\mu)$ \\
						~~~~~~~and $\Delta^\bPhi(\mu):= \| \bb^\bPhi(\mu) \| / \eta(\mu)$.}
					\WHILE{ $\Delta^\bPhi(\bu_i(\mu); \mu) \geq \tau$ and $i \leq r$} 
					\STATE{4. Set $i:=i+1$.}
					\STATE{5. Find the index $p_i$ which maximizes $|\bv^\bPhi_{p_i}(\mu)^\mathrm{H} \br^\bPhi_{i-1}(\mu)|$. 
						Set $\Lambda_i(\mu) := \Lambda_{i-1}(\mu) \cup \{ p_i\}$.}	
					\STATE{6. Set $\bv^\bPhi_{p_i}(\mu): = \bv^\bPhi_{p_i}(\mu) - \sum^{i-1}_{j=1} \bv^\bPhi_{p_j}(\mu) [\bv^\bPhi_{p_j}(\mu)^\mathrm{H} \bv^\bPhi_{p_i}(\mu)] $ and normalize it.}		
					\STATE{7. Compute $\br^\bPhi_i(\mu):= \br^\bPhi_{i-1}(\mu) - \bv^\bPhi_{p_i}(\mu) [\bv^\bPhi_{p_i}(\mu)^\mathrm{H} \br^\bPhi_{i-1}(\mu)] $, 
						$\Delta^\bPhi(\bu_i(\mu); \mu) = \| \br^\bPhi_i(\mu) \| / \eta(\mu)$.}
					\STATE{8. (Optional) Set $\bv^\bPhi_{j}(\mu) = \bv^\bPhi_{j}(\mu) - \bv^\bPhi_{p_i}(\mu) [\bv^\bPhi_{p_i}(\mu)^\mathrm{H} \bv^\bPhi_{j}(\mu)]$ and normalize it,\\ ~~~~~~~~ {$j \in \{1,2,\hdots,K\}\backslash\Lambda_i(\mu)$.} }		
					\ENDWHILE
					\STATE{9. Solve~\textup {(\ref {eq:skreduced_eq_eff})} choosing $r:=i$ and the columns $p_1, p_2, \dots, p_i$ of $\bV^{\bPhi}_K(\mu)$ as
						the columns\\~~~~~~~for $\bV^{\bPhi}_r(\mu)${,} and obtain solution $\ba_r(\mu)$.}
					\ENDFOR
				\end{algorithmic}
			\end{algorithm}

			\subsection{Dictionary generation} \label{Dgen} 
			
			The simplest way is to choose the dictionary as a set of solution samples (snapshots) associated with a training set $\mathcal{P}_{\mathrm{train}}$, i.e., 
			\begin{equation} 
			\mathcal{D}_K = \{ \bu(\mu): \mu \in  \mathcal{P}_{\mathrm{train}} \}.
			\end{equation}
			Let us recall that we are interested in computing a $\bTheta$-sketch of $\bU_K$ (matrix whose columns form $\mathcal{D}_K$) rather than the full matrix. In certain computational environments, a $\bTheta$-sketch of $\bU_K$ can be computed very efficiently. For instance, each snapshot may be computed and sketched on a separate distributed machine. Thereafter small sketches can be efficiently transfered to the master workstation for constructing the reduced order model.
			
			A better dictionary may be computed with the greedy procedure presented in~Algorithm\nobreakspace \ref {alg:sk_greedy_online}, recursively enriching the dictionary  with a snapshot  at the parameter value associated with the maximal error at the previous iteration. The value for $r$ (the dimension of the parameter-dependent reduced subspace $U_r(\mu)$) should be chosen according to the particular computational architecture. Since the provisional online solver (identified with~Algorithm\nobreakspace \ref {alg:skomp}) guarantees exact recovery of snapshots belonging to $\mathcal{D}_i$, Algorithm\nobreakspace \ref {alg:sk_greedy_online} is consistent.  It has to be noted that the first $r$ iterations of the proposed greedy algorithm for the dictionary generation coincide with the first $r$ iterations of the standard greedy algorithm for the reduced basis generation.
			
			\begin{algorithm} \caption{Greedy algorithm for dictionary generation} \label{alg:sk_greedy_online}
				\begin{algorithmic}
					\STATE{\textbf{Given:} $\mathcal{P}_{\mathrm{train}}$, $\bA(\mu)$, $\bb(\mu)$, $\bl(\mu)$, $\bTheta$, $\tau$, $r$.}
					\STATE{\textbf{Output}: $\bTheta$-sketch of $\bU_K$.}
					\STATE{1. Set $i=0$, $\mathcal{D}_0 = \emptyset$, obtain $\eta(\mu)$ in~\textup {(\ref {eq:reserrorind})}, set $\Delta^\bPhi(\mu) = \| \bb(\mu) \|^\Phi_{U'}/\eta(\mu)$ 
						and pick $\mu^1 \in \mathcal{P}_{\mathrm{train}}$.}
					\WHILE{$\max_{\mu \in \mathcal{P}{\mathrm{train}}}\Delta^\bPhi(\bu_r(\mu);\mu)> \tau$} 
					\STATE{2. Set $i=i+1$.}
					\STATE{3. Evaluate $\bu(\mu^{i})$ and set $\mathcal{D}_{i}:= \mathcal{D}_{i-1} \cup \{ \bu(\mu^{i}) \}$}
					\STATE{4. Update the $\bTheta$-sketch of $\bU_i$ (matrix composed from the vectors in $\mathcal{D}_i$).}
					\STATE{5. Use~Algorithm\nobreakspace \ref {alg:skomp} (if $i < r$, choosing $r:=i$) with $\mathcal{P}_{\mathrm{test}}$ replaced by $\mathcal{P}_{\mathrm{train}}$ to solve~\textup {(\ref {eq:sksparseminresproj})} \\ 
						~~~~~~~for all $\mu \in \mathcal{P}_{\mathrm{train}}$.}
					\STATE{6. Find $\mu^{i+1}:= \arg \max_{\mu \in \mathcal{P}{\mathrm{train}}}\Delta^\bPhi(\bu_r(\mu);\mu)$.}
					\ENDWHILE
				\end{algorithmic}
				
			\end{algorithm}
			
			By~Proposition\nobreakspace \ref {ssminresopt}, a good quality of {a} $\bTheta$-sketch for the sketched sparse minres approximation associated with dictionary $\mathcal{D}_K$ on $\mathcal{P}_\mathrm{train}$ can be guaranteed if $\bTheta$ is an $\varepsilon$-embedding for every subspace $R_r(W_r; \mu)$, defined by~\textup {(\ref {eq:R_r})}, with $W_r \in \mathcal{L}_r(\mathcal{D}_K)$ and $\mu \in \mathcal{P}_\mathrm{train}$. This condition can be enforced a priori for all possible outcomes of~Algorithm\nobreakspace \ref {alg:sk_greedy_online} by choosing $\bTheta$ such that it is an $\varepsilon$-embedding for every subspace $R_r(W_r; \mu)$ with $W_r \in \mathcal{L}_r(\{ \bu(\mu): \mu \in \mathcal{P}_\mathrm{train}\})$ and $\mu \in \mathcal{P}_\mathrm{train}$. An embedding $\bTheta$ satisfying this property with probability at least $1-\delta$ can be obtained as a realization of a $(\varepsilon, (\#{P}_\mathrm{train})^{-1}\binom{\#\mathcal{P}_\mathrm{train}}{r}^{-1} \delta, r+1 )$ oblivious $U \to \ell_2$ subspace embedding. {The computational cost of~Algorithm\nobreakspace \ref {alg:sk_greedy_online} is dominated by the calculation of the snapshots $\bu(\mu^i)$ and their $\bTheta$-sketches. As was argued in~\cite{balabanov2019galerkin}, the computation of the snapshots can have only a minor impact on the overall cost of an algorithm. For the classical sequential or limited-memory computational architectures, each snapshot should require a log-linear complexity and memory consumption, while for parallel and distributed computing the routines for computing the snapshots should be well-parallelizable and require low communication between cores. Moreover, for the computation of the snapshots one may use a highly-optimized commercial solver or a powerful server. The $\bTheta$-sketch of the snapshots may be computed extremely efficiently in basically any computational architecture~\cite[Section 4.4]{balabanov2019galerkin}. With SRHT, sketching of $K$ snapshots requires only $\mathcal{O}(n(K m_A + m_b) \log k)$ flops, and the maintenance of the sketch requires $\mathcal{O}((K m_A + m_b)k)$ bytes of memory. By using similar arguments as in~Section\nobreakspace \ref {ssminres} it can be shown that $k = \mathcal{O}(r\log{K})$ (or $k = \mathcal{O}(r^*\log{K})$) is enough to yield with high probability an accurate approximation of the dictionary-based reduced model. With this value of $k$, the required number of flops for the computation and the amount of memory for the storage of a $\bTheta$-sketch becomes $\mathcal{O}(n(K m_A + m_b) (\log r +\log \log{K}))$, and $\mathcal{O}((K m_A + m_b)r\log{K})$, respectively.}

			\section{A posteriori certification of the sketch and solution} \label{sketchcert}
			
			Here we provide a simple, yet efficient procedure for a posteriori verification of the quality of a sketching matrix and describe {a} few scenarios where such a procedure can be employed. The proposed a posteriori certification of the sketched reduced model and its solution is probabilistic. It does not require operating with high-dimensional vectors but only with their small sketches. The quality of a certificate shall be characterized by two user specified parameters: $0<\delta^*<1$ for the probability of success and $0<\varepsilon^*<1$ for the tightness of the computed error bounds.  
			
			\subsection{Verification of an $\varepsilon$-embedding for a given subspace}
			Let $\bTheta$ be a $U \to \ell_2$ subspace embedding and $V \subset U$ be a subspace of $U$ (chosen depending on the reduced model, e.g., $V:= {R}_r(U_r; \mu)$ in~\textup {(\ref {eq:R_r})}). {Recall that the quality of $\bTheta$ can be measured by the accuracy of $\langle \cdot, \cdot \rangle^{\bTheta}_{U}$ as an approximation of $\langle \cdot, \cdot \rangle_{U}$ for vectors in $V$}.
			
			We propose to verify the accuracy of $\langle \cdot, \cdot \rangle^{\bTheta}_{U}$ simply by comparing it to an inner product $\langle \cdot, \cdot \rangle^{\bTheta^*}_{U}$ associated with a new random embedding $\bTheta^* \in \mathbb{K}^{k^*\times n}$, {where $\bTheta^*$ is chosen such that for any given vectors $\bx,\by \in V$ the concentration inequality 
				\begin{equation} \label{eq:Vconcentration}
				| \langle \bx, \by \rangle_U -  \langle \bx, \by \rangle^{\bTheta^*}_U  |\leq \varepsilon^* \| \bx \|_U \| \by \|_U  
				\end{equation}
				holds with probability at least $1-\delta^*$.} One way to ensure~\textup {(\ref {eq:Vconcentration})} is to choose $\bTheta^*$ as an $( \varepsilon^*, \delta^*, 1)$ oblivious $U \to \ell_2$ subspace embedding. A condition on the number of rows for the oblivious embedding can be either obtained theoretically (see~\cite{balabanov2019galerkin}) or chosen from the practical experience, which should be the case for embeddings constructed with SRHT matrices {(recall that they have worse theoretical guarantees than the Gaussian matrices but perform equally well in practice).} Alternatively, $\bTheta^*$ can be built by random sampling of the rows of a larger $\varepsilon$-embedding for $V$. This approach can be far more efficient than generating $\bTheta^*$ as an oblivious embedding (see~Remark\nobreakspace \ref {rmk:parentembedding}) or even essential for some computational architectures. Another requirement for $\bTheta^*$  is that it is generated independently from $\bTheta$ {and $V$}. Therefore, in the algorithms we suggest to consider $\bTheta^*$ only for the certification of the solution and nothing else.  
			
			\begin{remark} \label{rmk:parentembedding}	
				In some scenarios it can be beneficial to construct $\bTheta$ and $\bTheta^*$ by sampling their rows from a fixed realization of a larger oblivious embedding $\hat{\bTheta}$, which is guaranteed a priori to be an $\varepsilon$-embedding for $V$ with high probability. More precisely, $\bTheta$ and $\bTheta^*$ can be defined as
				\begin{equation}
				\bTheta := \bGamma \hat{\bTheta},~~ \bTheta^* := \bGamma^* \hat{\bTheta},
				\end{equation}
				where $\bGamma$ and $\bGamma^*$ are random independent sampling (or Gaussian, or SRHT) matrices. In this way, a $\hat{\bTheta}$-sketch of a reduced order model can be first precomputed and then used for efficient evaluation/update of the sketches associated with $\bTheta$ and $\bTheta^*$. This approach can be essential for the adaptive selection of the optimal size for $\bTheta$ in a limited-memory environment where only one pass (or a few passes) over the reduced basis vectors is allowed and therefore there is no chance to recompute a sketch associated with an oblivious embedding at each iteration. It may also reduce the complexity of an algorithm (especially when $\bTheta$ is constructed with SRHT matrices) by not requiring to recompute high-dimensional matrix-vector products multiple times. 
			\end{remark}

			Let $\bV$ denote a matrix whose columns form a basis of $V$. Define the sketches $\bV^\bTheta:= \bTheta \bV $ and $\bV^{\bTheta^*}:= \bTheta^* \bV $. Note that $\bV^\bTheta$ and $\bV^{\bTheta^*}$ contain as columns low-dimensional vectors and therefore are cheap to {maintain} and to operate with (unlike the matrix $\bV$).  
			
			We start with the certification of the inner product between two fixed vectors from $V$ (see~Proposition\nobreakspace \ref {thm:certvectors}).
			
			\begin{proposition} \label{thm:certvectors}
				For any two vectors $\bx, \by \in V$ {possibly depending on $\bTheta$ but independent of $\bTheta^*$}, we have that
				\begin{equation} \label{eq:certvectors}
				\begin{split}
				| \langle \bx, \by \rangle^{\bTheta^*}_{U} - \langle \bx, \by \rangle^{\bTheta}_{U} | - \frac{\varepsilon^*}{1-\varepsilon^*} \| \bx \|^{\bTheta^*}_{U}  \| \by\|^{\bTheta^*}_{U} 
				&\leq {| \langle \bx, \by \rangle_{U} - \langle \bx, \by \rangle^{\bTheta}_{U}|} \\ 
				& \leq | \langle \bx, \by \rangle^{\bTheta^*}_{U} - \langle \bx, \by \rangle^{\bTheta}_{U} | + \frac{\varepsilon^*}{1-\varepsilon^*} \| \bx \|^{\bTheta^*}_{U}  \| \by\|^{\bTheta^*}_{U}
				\end{split}
				\end{equation}
				holds with probability at least $1-4\delta^*$. 
				\begin{proof}
					See Appendix B.
				\end{proof}
			\end{proposition}
			The error bounds in~Proposition\nobreakspace \ref {thm:certvectors} can be computed from the sketches of $\bx$ and $\by$, which may be efficiently evaluated from $\bV^\bTheta$ and $\bV^{\bTheta^*}$ and the coordinates of $\bx$ and $\by$ associated with $\bV$, with no operations on high-dimensional vectors. 
			A certification for several pairs of vectors should be obtained using a union bound for the probability of success. By replacing $\bx$ by $\bR_U^{-1} \bx'$ and $\by$ by $\bR_U^{-1} \by'$ in~Proposition\nobreakspace \ref {thm:certvectors} and using definition~\textup {(\ref {eq:thetadef})} one can derive a certification of the dual inner product  $\langle \cdot, \cdot \rangle^{\bTheta}_{U'}$ for vectors $\bx'$ and $\by'$ in $V':= \{ \bR_U \bx : \bx \in V\}$. 
			
			In general, the quality of an approximation with a $\bTheta$-sketch of a reduced model should be characterized by the accuracy of $\langle \cdot, \cdot \rangle^{\bTheta}_{U}$ for the whole subspace $V$. Let $\omega$ be the minimal value for $\varepsilon$ such that $\bTheta$ satisfies an $\varepsilon$-embedding property for $V$. Now, we address the problem of computing an a posteriori upper bound $\bar{\omega}$ for $\omega$ from the sketches $\bV^\bTheta$ and $\bV^{\bTheta^*}$ and we provide conditions to ensure quasi-optimality of $\bar{\omega}$.
			
			\begin{proposition} \label{thm:omegaUB}
				For a fixed realization of $\bTheta^*$, let us define 
				\begin{equation} \label{eq:omegaUB}
				\bar{\omega} := \max \left \{ 1- (1-\varepsilon^*) \min_{\bx \in V / \{ \bnull \}} \left (\frac{\| \bx\|^{\bTheta}_U}{\| \bx\|^{\bTheta^*}_U} \right )^2, (1+\varepsilon^*) \max_{\bx \in V / \{ \bnull \}} \left (\frac{\| \bx\|^{\bTheta}_U}{\| \bx\|^{\bTheta^*}_U}\right )^2 -1  \right \}.
				\end{equation}
				If $\bar{\omega}<1$, then $\bTheta$ is guaranteed to be a $U \to \ell_2$ $\bar{\omega}$-subspace embedding for $V$ with probability at least $1-\delta^*$.
				\begin{proof}
					See Appendix B.
				\end{proof}
			\end{proposition}
			It follows that if $\bar{\omega}<1$ then it is an upper bound for $\omega$ with high probability. Assume that $\bV^\bTheta$ and $\bV^{\bTheta^*}$ have full ranks. Let $\bT^*$ be the matrix such that $\bV^{\bTheta^*} \bT^*$ is orthogonal (with respect to $\ell_2$-inner product). Such a matrix can be computed with {a} QR factorization. Then $\bar{\omega}$ defined~in~\textup {(\ref {eq:omegaUB})} can be obtained from the following relation
			\begin{equation} \label{eq:compute_omega}
			\bar{\omega} = \max \left \{ 1-(1-\varepsilon^*) \sigma^2_{\min}, (1+\varepsilon^*) \sigma^2_{\max} -1 \right \},
			\end{equation}
			where $\sigma_{\min}$ and $\sigma_{\max}$ are the minimal and the maximal singular 
			values of the small matrix $\bV^\bTheta \bT^*$. 
			
			We have that $\bar{\omega} \geq \varepsilon^*.$ The value for $\varepsilon^*$ may be selected an order of magnitude less than $\omega$ with no considerable impact on the computational cost, therefore in practice the effect of $\varepsilon^*$ on $\bar{\omega}$  can be considered {to be} negligible. Proposition\nobreakspace \ref {thm:omegaUB} implies that $\bar{\omega}$ is an upper bound of $\omega$ with high probability. A guarantee of effectivity of $\bar{\omega}$ (i.e., its closeness to $\omega$), however, has not been yet provided. To do so we shall need a stronger assumption on $\bTheta^*$ than~\textup {(\ref {eq:Vconcentration})}. 
			\begin{proposition}	\label{thm:omegaUBoptimality}
				If the realization of $\bTheta^*$ is a $U \to \ell_2$ $\omega^*$-subspace embedding for $V$, then $\bar{\omega}$ (defined by~\textup {(\ref {eq:omegaUB})}) satisfies	 
				\begin{equation} \label{eq:omegaUBoptimality}
				\bar{\omega} \leq  \frac{1+\varepsilon^*}{1-\omega^*} (1+\omega) -1.
				\end{equation}
				\begin{proof}
					See Appendix B.
				\end{proof}
			\end{proposition}
			
			If $\bTheta^*$ is a $(\omega^*, \gamma^*, \mathrm{dim}({V}))$ oblivious $U \to \ell_2$ subspace embedding, then the condition on $\bTheta^*$ in~Proposition\nobreakspace \ref {thm:omegaUBoptimality} is satisfied with probability at least $1-\gamma^*$ (for some user-specified value $\gamma^*$). 
			Therefore, a matrix $\bTheta^*$ of moderate size should yield a very good upper bound $\bar{\omega}$ of $\omega$. Moreover, if $\bTheta$ and $\bTheta^*$ are drawn from the same distribution, then $\bTheta^*$ can be expected to be an $\omega^*$-embedding for $V$ with $\omega^* = \mathcal{O}( \omega)$ with high probability. Combining this consideration with~Proposition\nobreakspace \ref {thm:omegaUBoptimality} we deduce that a sharp upper bound should be obtained for {some} $k^* \leq k$. Therefore, in the algorithms one may readily consider $k^*:=k$. If pertinent, a better value for $k^*$ can be selected adaptively, at each iteration increasing $k^*$ by a constant factor until the desired tolerance or a stagnation of $\bar{\omega}$ is reached. 
			
			\subsection{Certification of a sketch of a reduced model and its solution} 
			
			The results of~Propositions\nobreakspace \ref {thm:certvectors} and\nobreakspace  \ref {thm:omegaUB} can be employed for certification of a sketch of a reduced model and its solution. They can also be used for adaptive selection of the number of rows of a random sketching matrix to yield an accurate approximation of the reduced model. Thereafter we discuss several practical applications of the methodology described above.  
			
			\subsubsection*{Approximate solution}
			
			Let $\bu_r(\mu) \in U$ be an approximation of $\bu(\mu)$. The accuracy of $\bu_r(\mu)$ can be measured with the residual error $\| \br(\bu_r(\mu); \mu) \|_{U'}$, which can be efficiently estimated by  
			$$ \| \br(\bu_r(\mu); \mu) \|_{U'} \approx \| \br(\bu_r(\mu); \mu) \|^\bTheta_{U'}.$$
			The certification of such estimation can be derived from~Proposition\nobreakspace \ref {thm:certvectors}
			choosing $\bx = \by := \bR_U^{-1} \br(\bu_r(\mu); \mu)$ and using definition~\textup {(\ref {eq:thetadef})} of $\| \cdot \|^\bTheta_{U'}$.
			
			For applications, which involve computation of snapshots over the training set (e.g., approximate POD or greedy algorithm with the exact error indicator), one should be able to efficiently precompute the sketches of $\bu(\mu)$. Then the error $\| \bu(\mu) - \bu_r(\mu) \|_U$ can be efficiently estimated by 
			$$\| \bu(\mu) - \bu_r(\mu) \|_U \approx \| \bu(\mu) - \bu_r(\mu) \|^\bTheta_U. $$
			Such an estimation can be certified with~Proposition\nobreakspace \ref {thm:certvectors} choosing $\bx = \by := \bu(\mu) - \bu_r(\mu)$.

			\subsubsection*{Minimal residual projection}
			By~Proposition\nobreakspace \ref {thm:skminresopt}, the quality of the $\bTheta$-sketch of a subspace $U_r$ for approximating the minres projection for a given parameter value can be characterized by the lowest value $\omega$ for $\varepsilon$ such that $\bTheta$ satisfies the $\varepsilon$-embedding property for subspace $V := R_r(U_r; \mu)$, defined in~\textup {(\ref {eq:R_r})}. The upper bound $\bar{\omega}$ of such $\omega$ can be efficiently computed using~\textup {(\ref {eq:omegaUB})}. The verification of {a} $\bTheta$-sketch for all parameter values in $\mathcal{P}$, simultaneously, can be performed by considering a subspace $V$ in~\textup {(\ref {eq:omegaUB})}, which contains $\bigcup_{\mu \in \mathcal{P}} R_r(U_r; \mu)$.
			
			\subsubsection*{Dictionary-based approximation}
			For each parameter value, the quality of $\bTheta$ for dictionary-based approximation defined in~\textup {(\ref {eq:sksparseminresproj})} can be characterized by the quality of the sketched minres projection associated with a subspace $U_r(\mu)$, which can be verified by computing $\bar{\omega}$ in~\textup {(\ref {eq:omegaUB})} associated with $V := R_r(U_r(\mu); \mu)$.

			\subsubsection*{Output quantity}
			{In Appendix A we provide a way for estimating the output quantity $s_r(\mu) =  \langle \bl(\mu), \bu_r(\mu) \rangle$, with $\bl(\mu) \in U'$. More specifically $s_r(\mu)$ can be efficiently estimated by $s^\star_r(\mu) =\langle \bl(\mu), \bw_p(\mu) \rangle +   \langle \bR^{-1}_U \bl(\mu), \bu_r(\mu) - \bw_p (\mu) \rangle^\bTheta_U$, where $\bw_p(\mu)$ is a projection of $\bu_r(\mu)$ on a new reduced basis.}  We have
			$$ |s_r(\mu) - s^\star_r(\mu)| = | \langle \bR^{-1}_U \bl(\mu), \bu_r(\mu) - \bw_p (\mu) \rangle_U -  \langle \bR^{-1}_U \bl(\mu), \bu_r(\mu) - \bw_p (\mu) \rangle^\bTheta_U|,$$
			therefore the quality of $s^\star_r(\mu)$ may be certified by~Proposition\nobreakspace \ref {thm:certvectors} with $\bx = \bR_U^{-1} \bl(\mu) $, $\by = \bu_r(\mu) - \bw_p (\mu)$.

			\subsubsection*{Adaptive selection of the size for a random sketching matrix}
			
			When no a priori bound for the size of $\bTheta$ sufficient to yield an accurate sketch of a reduced model is available, or when the bounds are pessimistic, the sketching matrix should be selected adaptively. At each iteration, if the certificate indicates a poor quality of a $\bTheta$-sketch for approximating the solution (or the error) on $\mathcal{P}_\mathrm{train} \subseteq \mathcal{P}$, one can improve the accuracy of the sketch by adding extra rows to $\bTheta$. In the analysis, the embedding $\bTheta^*$ used for certification was assumed to be independent of $\bTheta$, consequently a new realization of $\bTheta^*$ should be sampled after each decision to improve $\bTheta$ has been made. To save computational costs, the previous realizations of $\bTheta^*$ and the associated $\bTheta^*$-sketches can be readily {recycled} as parts of the updates for $\bTheta$ and the $\bTheta$-sketch. 
			
			We finish with a practical application of~Propositions\nobreakspace \ref {thm:omegaUB} and\nobreakspace  \ref {thm:omegaUBoptimality}. Consider a situation where one is given a class of random embeddings (e.g., oblivious subspace embeddings mentioned in~Section\nobreakspace \ref {randsk} or the embeddings constructed with random sampling of rows of an $\varepsilon$-embedding as in~Remark\nobreakspace \ref {rmk:parentembedding}) and one is interested in generating an $\varepsilon$-embedding $\bTheta$ (or rather computing the associated sketch), {with a user-specified accuracy $\varepsilon \leq \tau$,} for $V$ (e.g., a subspace containing $\cup_{\mu \in \mathcal{P}} R_r(U_r; \mu)$) with nearly optimal number of rows. Moreover, we {consider the cases when} no bound {for} the size of matrices to yield an $\varepsilon$-embedding is available or when the bound is pessimistic. It is only known that matrices with more than $k_0$ rows satisfy~\textup {(\ref {eq:Vconcentration})}. This condition could be derived theoretically (as for Gaussian matrices) or {deduced} from practical experience (for SRHT). {The} matrix $\bTheta$ can be readily generated adaptively using $\bar{\omega}$ defined by~\textup {(\ref {eq:omegaUB})} as an error indicator (see~Algorithm\nobreakspace \ref {alg:adaptive_embedding}). {It directly follows by a union bound argument that $\bTheta$ generated in~Algorithm\nobreakspace \ref {alg:adaptive_embedding} is an $\varepsilon$-embedding for $V$, with $\varepsilon \leq \tau$, with probability at  least $1-t \delta^*$, where $t$ is the number of iterations taken by the algorithm.} 
			
			\begin{algorithm}[h] \caption{Adaptive selection of the number of rows for $\bTheta$} \label{alg:adaptive_embedding}
				\begin{algorithmic}
					\STATE{\textbf{Given:} $k_0$, $\bV$, $\tau > \varepsilon^*$.}
					\STATE{\textbf{Output}:  $\bV^\bTheta$, $\bV^{\bTheta^*}$}
					\STATE{1. Set $k=k_0$ and $\bar{\omega} =\infty$.}
					\WHILE{$\bar{\omega} > \tau$} 
					\STATE{2. Generate $\bTheta$ and $\bTheta^*$ with $k$ rows and evaluate $\bV^\bTheta:= \bTheta \bV$ and $\bV^{\bTheta^*}:=\bTheta^* \bV$.}
					\STATE{3. Use~\textup {(\ref {eq:compute_omega})} to compute $\bar{\omega}$.} 
					\STATE{4. Increase $k$ by a constant factor.}
					\ENDWHILE
				\end{algorithmic}
			\end{algorithm}
			
			To improve the efficiency, at each iteration of~Algorithm\nobreakspace \ref {alg:adaptive_embedding} we could select the number of rows for $\bTheta^*$ adaptively instead of choosing it equal to $k$. In addition, the embeddings from previous iterations can be considered as parts of $\bTheta$ at further iterations. 
			
			\section{Numerical experiments} \label{numex}
			This section is devoted to experimental validation of the methodology as well as realization of its great practical potential. The numerical tests were carried out on two benchmark problems that are difficult to tackle with the standard projection-based MOR methods due to a high computational cost and issues with numerical stability of the computation (or minimization) of the residual norm, or bad approximability of the solution manifold with a low-dimensional space. 
			
			In all the experiments we used oblivious $U \to \ell_2$ embeddings of the form  $$\bTheta:=  \bOmega \bQ,$$ where $\bQ$  was taken as the (sparse) transposed Cholesky factor of $\bR_U$ and $\bOmega$ as a SRHT matrix. The random embedding $\bGamma$ used for the online efficiency was also taken as  SRHT.    
			Moreover, for simplicity in all the experiments the coefficient $\eta(\mu)$ for the error estimation was chosen as~$1$. 
			
			The experiments were executed on an Intel\textsuperscript{\textregistered} Core{\texttrademark} i7-7700HQ 2.8GHz CPU, with 16.0GB RAM memory using Matlab\textsuperscript{\textregistered} {R2017b}.
			
			\subsection{Acoustic invisibility cloak} \label{cloak}
			The first numerical example is inspired by the development of invisibility cloaking~\cite{cheng2008multilayer,chen2010acoustic}. It consists {of} an acoustic wave scattering in $2$D with a perfect scatterer covered in an invisibility cloak composed of layers of homogeneous isotropic materials. The geometry of the problem is depicted in~Figure\nobreakspace \ref {fig:Ex1_intial_problem_a}. The cloak consists of $32$ layers of equal thickness $1.5625\mathrm{~cm}$ each constructed with $4$ sublayers of equal thickness of alternating materials: mercury (a heavy liquid) followed by a light liquid. The properties (density and bulk modulus) of the light liquids are chosen to minimize the visibility of the scatterer for the frequency band $[7.5, 8.5]\mathrm{~kHz}$. The associated boundary value problem with first order absorbing boundary conditions is the following     
			\begin{equation} \label{eq:BVP1}
			\left \{
			\begin{array}{rll}
			\nabla \cdot ( \rho^{-1} \nabla u) + \rho^{-1} \kappa^2 u &= 0,~~  & \textup{in } \Omega \\
			(j \kappa -\frac{1}{2 R_{\Omega}}) u + \frac{\partial u}{\partial {n}}  &=(j \kappa -\frac{1}{2 R_{\Omega}}) u^{in} + \frac{\partial u^{in}}{\partial {n}},~~ & \textup{on } \Gamma \\
			\frac{\partial u}{\partial {n}} &= 0,~~ & \textup{on } \Gamma_{s},
			\end{array}
			\right.
			\end{equation} 
			where $j=\sqrt{-1}$, $u = u^{in}+u^{sc}$ is the total pressure, with $u^{in} = \exp(-j \kappa (y-4))~\mathrm{Pa} \cdot \mathrm{m}$ being the pressure of the incident plane wave and $u^{sc}$ being the pressure of the scattered wave,  $\rho$ is the material's density, $\kappa = \frac{2 \pi f}{c}$ is the wave number,  $c = \sqrt{\frac{b}{\rho}}$ is the speed of sound and $b$ is the bulk modulus. The background material is chosen as water having density $\rho = \rho_0 = 997~\mathrm{kg/m^3} $ and bulk modulus $b = b_0 = 2.23~\mathrm{GPa}$.  
			For the frequency $f = 8~\mathrm{kHz}$ the associated wave number of the background is $\kappa=\kappa_0=33.6~ \mathrm{m}^{-1}$.  The $i$-th layer of the cloak (enumerated starting from the outermost layer) is composed of $4$ alternating layers of mercury with density $\rho = \rho_m =13500 ~\mathrm{kg/m^3}$ and bulk modulus $b = b_m = 28~\mathrm{GPa}$ and a light liquid with density $\rho=\rho_i$ and bulk modulus $b =b_i$ given in~Table\nobreakspace \ref {tab:materials}. The light liquids from~Table\nobreakspace \ref {tab:materials} can in practice be obtained, for instance, with the pentamode mechanical metamaterials~\cite{bertoldi2017flexible,kadic2014pentamode}. 
			\begin{table}[tbhp]
				\caption{The properties, density in~$\mathrm{kg/m^3}$ and bulk modulus in~${\mathrm{GPa}}$, of the light liquids in the cloak.}
				\label{tab:materials}
				\centering	
				\scalebox{0.9}{
					\begin{tabular}{||l|l|l||l|l|l||l|l|l||l|l|l||} \hline
						$i$	& $\rho_i$  & $b_i$ & $i$	& $\rho_i$ & $b_i$ & $i$	& $\rho_i$ & $b_i$ & $i$	& $\rho_i$ & $b_i$ \\ \hline	    
						1& 231   & 0.483  &	9&   179   & 0.73  & 17	    &  56.1 & 0.65       & 25	& 9   &1.34      \\ \hline
						2& 121   & 0.328  &10&   166   & 0.78    & 18	& 59.6  & 0.687       & 26	&9    & 2.49      \\ \hline
						3& 162   & 0.454  &	11&    150   & 0.745   & 19	& 40.8  &  0.597       & 27	& 9   & 2.5       \\ \hline
						4& 253  &  0.736  &	12&  140   & 0.802   & 20	    & 32.1  &0.682       & 28	& 9   &2.5      \\ \hline
						5& 259   & 0.767  &	13&  135   & 0.786   & 21   	& 22.5  & 0.521       & 29	& 9   & 0.58      \\ \hline
						6& 189   & 0.707  &	14&  111   &  0.798   & 22	& 15.3  &  0.6       & 30	&  9.5  & 1.91      \\ \hline
						7& 246   & 0.796  &	15&   107   & 0.8   & 23	    & 10    &  0.552       & 31	&  9.31 &  0.709       \\ \hline
						8& 178   & 0.739  &	16&  78   & 0.656   & 24	    &  9  &  1.076       & 32	&  9   &2.44     \\ \hline		
					\end{tabular}
				}
			\end{table}

			The last $10$ layers contain liquids with small densities that can be subject to imperfections during the manufacturing process. Moreover, the external conditions (such as temperature and pressure) may also affect the material's properties. We then consider a characterization of the impact of small perturbations of the density and the bulk modulus of the light liquids in the last $10$  layers on the quality of the cloak in the frequency regime $[7.8, 8.2]\mathrm{~kHz}$. Assuming that the density and the bulk modulus may vary by $2.5\%$, the corresponding parameter set is  $$\mathcal{P} = \underset{23\leq i \leq 32}{\times}[0.975\rho_i, 1.025 \rho_i]\underset{23\leq i \leq 32}{\times}[0.975 b_i, 1.025 b_i]~\times~[7.8~\mathrm{kHz}, 8.2~\mathrm{kHz}].$$
			Note that in this case $\mathcal{P} \subset \mathbb{R}^{21}$. 
			
			\begin{figure}[h!]
				\centering
				\begin{minipage}[c]{.4\textwidth}
					\centering
					\includegraphics[width=.96\textwidth]{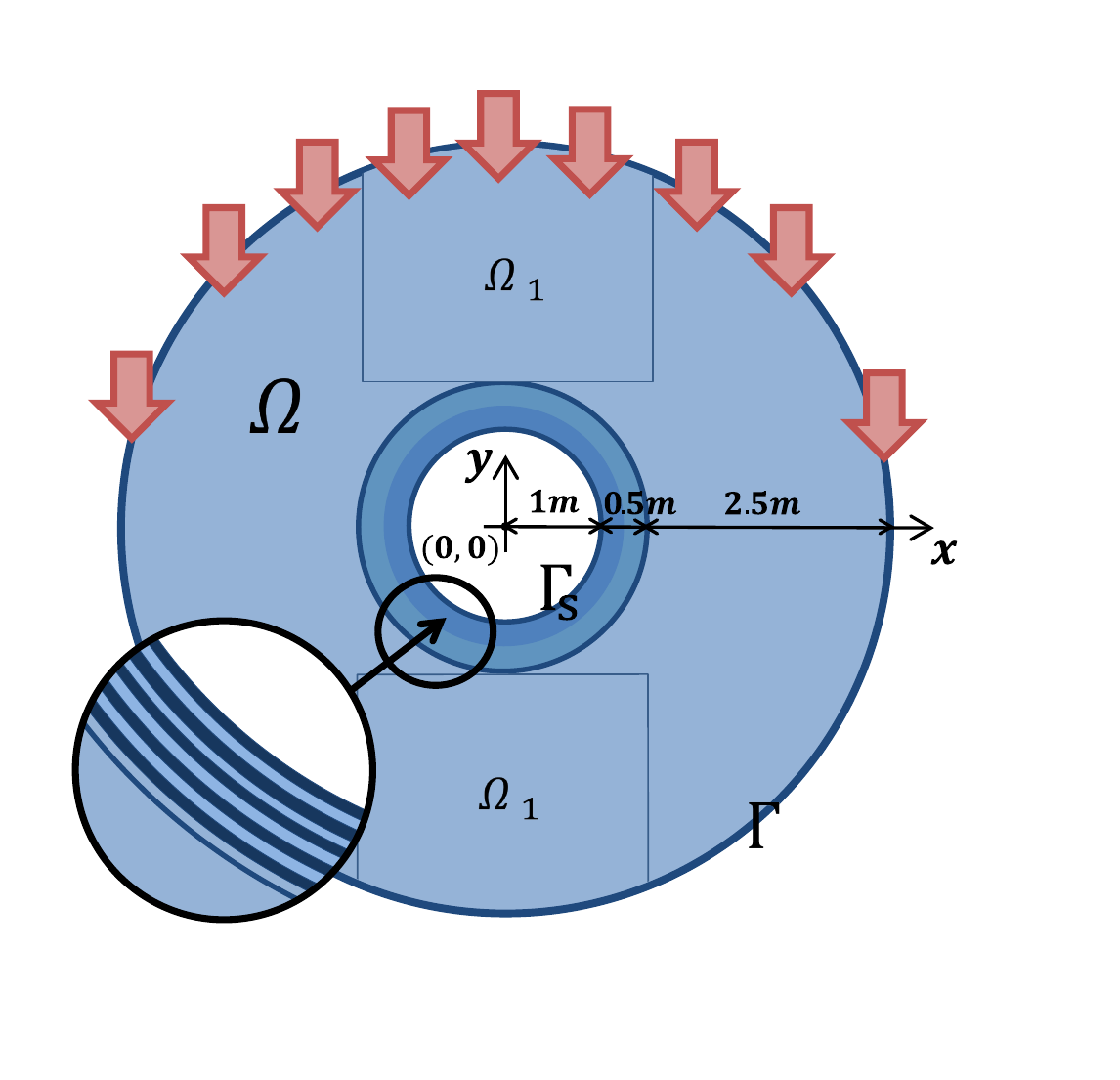}
				\end{minipage}%
				\begin{minipage}[c]{.4\textwidth}
					\centering
					\includegraphics[width=.8\textwidth]{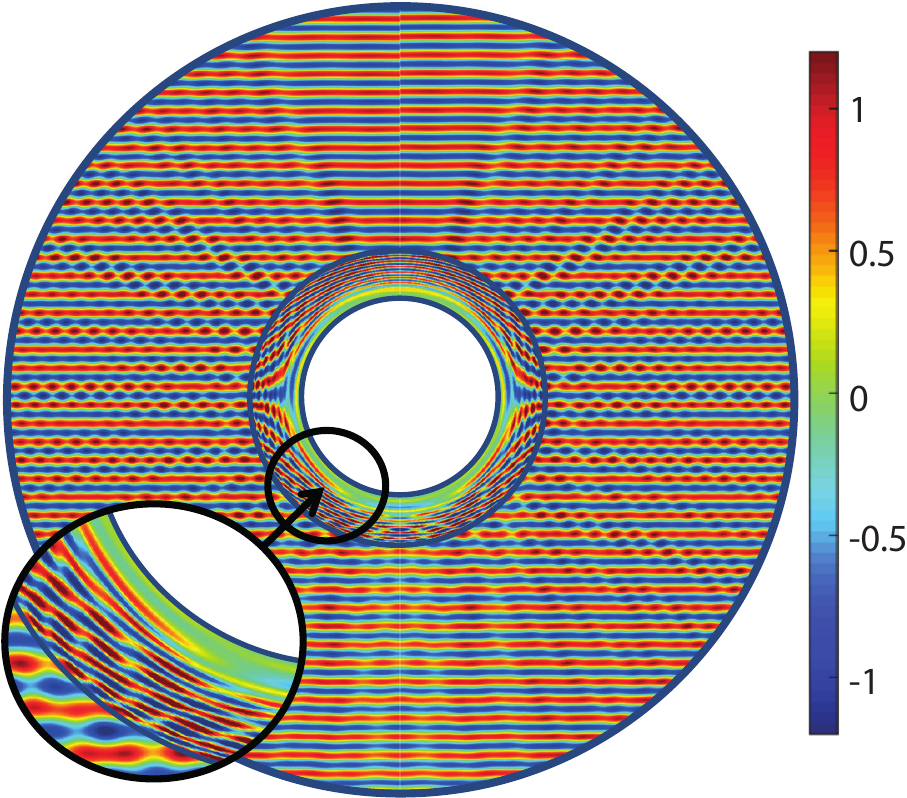}
				\end{minipage} \\ 
				\begin{minipage}[b]{.4\textwidth}
					\centering
					\subcaption{Geometry}
					\label{fig:Ex1_intial_problem_a}
				\end{minipage}%
				\begin{minipage}[b]{.4\textwidth}
					\subcaption{Unperturbed cloak}
					\label{fig:Ex1_intial_problem_b}
				\end{minipage} 
				\begin{minipage}[c]{.4\textwidth}
					\centering
					\includegraphics[width=.8\textwidth]{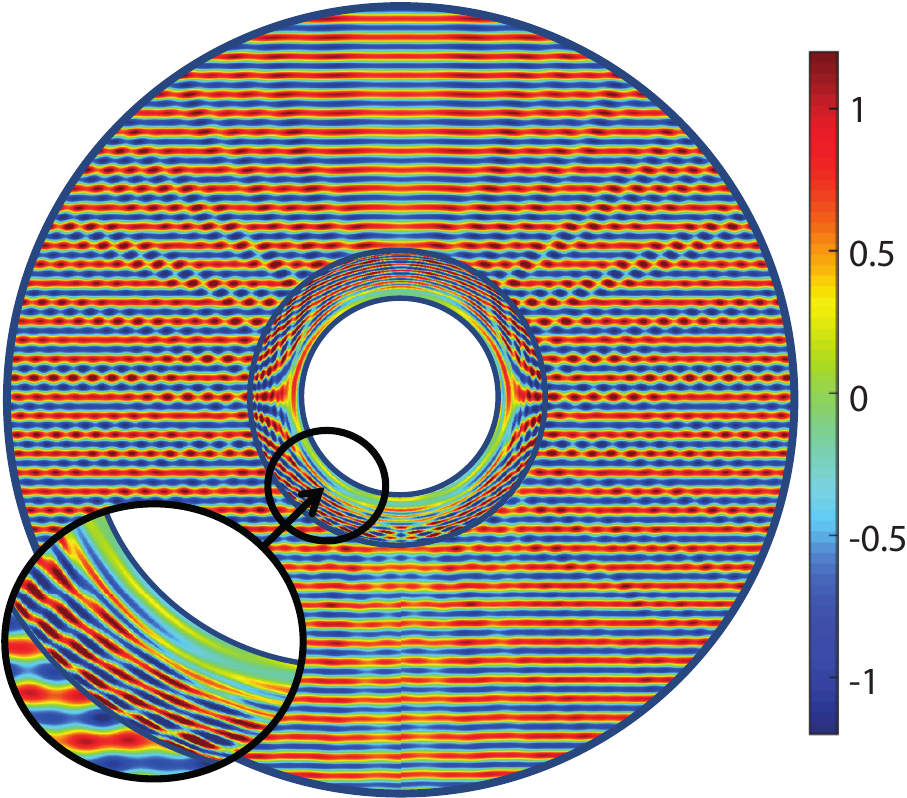}
				\end{minipage}%
				\begin{minipage}[c]{.4\textwidth}
					\centering
					\includegraphics[width=.8\textwidth]{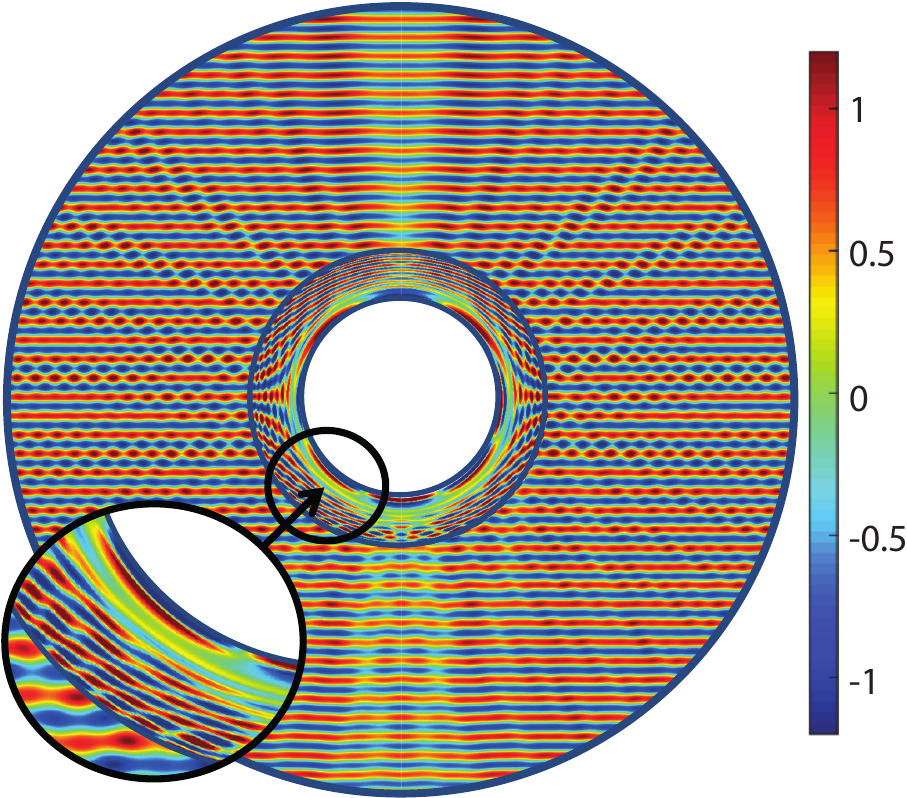}
				\end{minipage} \\ 
				\begin{minipage}[b]{.4\textwidth}
					\centering
					\subcaption{Random snapshot}
					\label{fig:Ex1_intial_problem_c}
				\end{minipage}%
				\begin{minipage}[b]{.4\textwidth}
					\subcaption{Random snapshot}
					\label{fig:Ex1_intial_problem_d}
				\end{minipage}
				\caption{(a) Geometry of the invisibility cloak benchmark. (b) The real component of $u$ in $\mathrm{Pa} \cdot \mathrm{m}$ for the parameter value $\mu \in \mathcal{P}$ corresponding to~Table\nobreakspace \ref {tab:materials} and frequency $f = 8~\mathrm{kHz}$. (c)-(d) The real component of $u$ in $\mathrm{Pa}\cdot \mathrm{m}$ for two random samples from $\mathcal{P}$ with $f = 7.8~\mathrm{kHz}$.}
				\label{fig:Ex1_intial_problem}
			\end{figure}

			The quantity of interest is chosen to be the following
			$$s(\mu) = l(u(\mu); \mu) {=} \| u(\mu) - u^{in}(\mu)\|^2_{L^2(\Omega_1)}/b_0 = \| u^{sc}(\mu)\|^2_{L^2(\Omega_1)}/b_0,$$ 
			which represents the (rescaled, time-averaged) acoustic energy of the scattered wave concealed in the region $\Omega_1$ (see~Figure\nobreakspace \ref {fig:Ex1_intial_problem_a}). For the considered parameter set $s(\mu)$ is ranging from~$0.0225 A_s$ to $0.095 A_s$, where $A_s= \| u^{in}\|^2_{L^2(\Omega_1)}/b_0=7.2 \mathrm{J} \cdot \mathrm{Pa} {\cdot \mathrm{m}}/b_0$ at frequency $8~\mathrm{kHz}$.

			The problem is symmetric with respect to {the} $x=0$ axis, therefore only half of the domain has to be considered {for discretization}. For the discretization, we used piecewise quadratic approximation on a mesh of triangular (finite) elements.  The mesh was chosen such that there were at least $20$ degrees of freedom per wavelength, which is a standard choice for Helmholtz problems with a moderate wave number.  It yielded approximately $400000$ complex degrees of freedom for the discretization. Figures\nobreakspace  \ref {fig:Ex1_intial_problem_b} to\nobreakspace  \ref {fig:Ex1_intial_problem_d}  depict the solutions $u(\mu)$ for different parameter values with quantities of interest $s(\mu)=0.032 A_s, 0.045 A_s$ and $0.065 A_s$, respectively.

			It is revealed that for this problem, considering the classical $H^1$ inner product for the solution space leads to dramatic instabilities of the projection-based MOR methods. 
			To improve the stability, the inner product is chosen corresponding to the specific structure of the operator in~\textup {(\ref {eq:BVP1})}. The solution space $U$ is here equipped with the following inner product 
			\begin{equation} \label{eq:ex1innerU}
			\langle \bv, \bw \rangle_{U} :=\langle \rho_s^{-1} \kappa_s^2 v,w  \rangle_{L^2}+ \langle \rho_s^{-1} {\nabla}v, {\nabla}w  \rangle_{L^2}, ~~ \bv, \bw \in U,
			\end{equation}
			where $v$ and $w$ are the functions identified with $\bv$ and $\bw$, respectively, and $\rho_s$ and $\kappa_s$ are the density and the wave number associated with the unperturbed cloak (i.e., with properties from~Table\nobreakspace \ref {tab:materials}) at frequency $8~\mathrm{kHz}$.

			The operator for this benchmark is directly given in an affine form with $m_A=23$ terms. Furthermore, for online efficiency we used {empirical interpolation method~\cite{maday2007general,barrault2004empirical}} to obtain an approximate affine representation of $u^{in}(\mu)$ (or rather a vector $\bu^{in}(\mu)$ from $U$ representing an approximation of $u^{in}(\mu)$) and the right hand side vector with $50$ affine terms (with error close to machine precision).
			The approximation space $U_r$ of dimension $r=150$ was constructed with a greedy algorithm (based on sketched minres projection)  performed on a training set of $50000$ uniform random samples in $\mathcal{P}$. The test set $\mathcal{P}_{test} \subset \mathcal{P}$ was taken as $1000$ uniform random samples in $\mathcal{P}$.
			
			\emph{Minimal residual projection.}
			Let us first address the validation of the sketched minres projection from~Section\nobreakspace \ref {skminres}. For this we computed sketched (and standard) minres projections $\bu_r(\mu)$ of $\bu(\mu)$ onto $U_r$ for each $\mu \in \mathcal{P}_{\mathrm{test}}$ with sketching matrix $\bTheta$ of varying sizes. The error of approximation is here characterized by  $\Delta_\mathcal{P}:=\max_{\mu \in \mathcal{P}_{\mathrm{test}}} \| \br(\bu_r(\mu); \mu) \|_{U'} /\|\bb(\mu_s)\|_{U'}$ and $e_\mathcal{P}:=\max_{\mu \in \mathcal{P}_{\mathrm{test}}} \|\bu(\mu) - \bu_r(\mu)\|_{U} / \|\bu^{in}(\mu_s)\|_{U}$, where $\bu^{in}(\mu_s)$ is the vector representing the incident wave and $\bb(\mu_s)$ is the right hand side vector associated with the unperturbed cloak and the frequency $f = 8\mathrm{~kHz}$ (see~Figures\nobreakspace \ref {fig:Ex1_1a} and\nobreakspace  \ref {fig:Ex1_1c}). Furthermore, in~Figure\nobreakspace \ref {fig:Ex1_1e} we provide the characterization of the maximal error in the quantity of interest $e^{s}_\mathcal{P}:=\max_{\mu \in \mathcal{P}_{\mathrm{test}}} | s(\mu)-s_r(\mu) | /A_s$. For each size of $\bTheta$, $20$ realizations of the sketching matrix were considered to analyze the statistical properties of $e_\mathcal{P}$, $\Delta_\mathcal{P}$ and $e^{s}_\mathcal{P}$. 
			
			For comparison, along with the minimal residual projections we also computed the sketched (and classical) Galerkin projection introduced in~\cite{balabanov2019galerkin}. Figures\nobreakspace \ref {fig:Ex1_1b},  \ref {fig:Ex1_1d} and\nobreakspace  \ref {fig:Ex1_1f} depict the errors $\Delta_\mathcal{P}$, $e_\mathcal{P}$, $e^\mathrm{s}_\mathcal{P}$ of a sketched (and classical) Galerkin projection using $\bTheta$ of different sizes. Again, for each $k$ we used $20$ realizations of $\bTheta$ to characterize the statistical properties of the error. We see that the classical Galerkin projection is more accurate in the exact norm $\| \cdot \|_U$ and the quantity of interest than the standard minres projection. On the other hand, it is revealed that the minres projection is far better suited to random sketching. 
			
			{ 
				\begin{figure}[htp!]
					\centering
					\begin{subfigure}[b]{.4\textwidth}
						\centering
						\includegraphics[width=\textwidth]{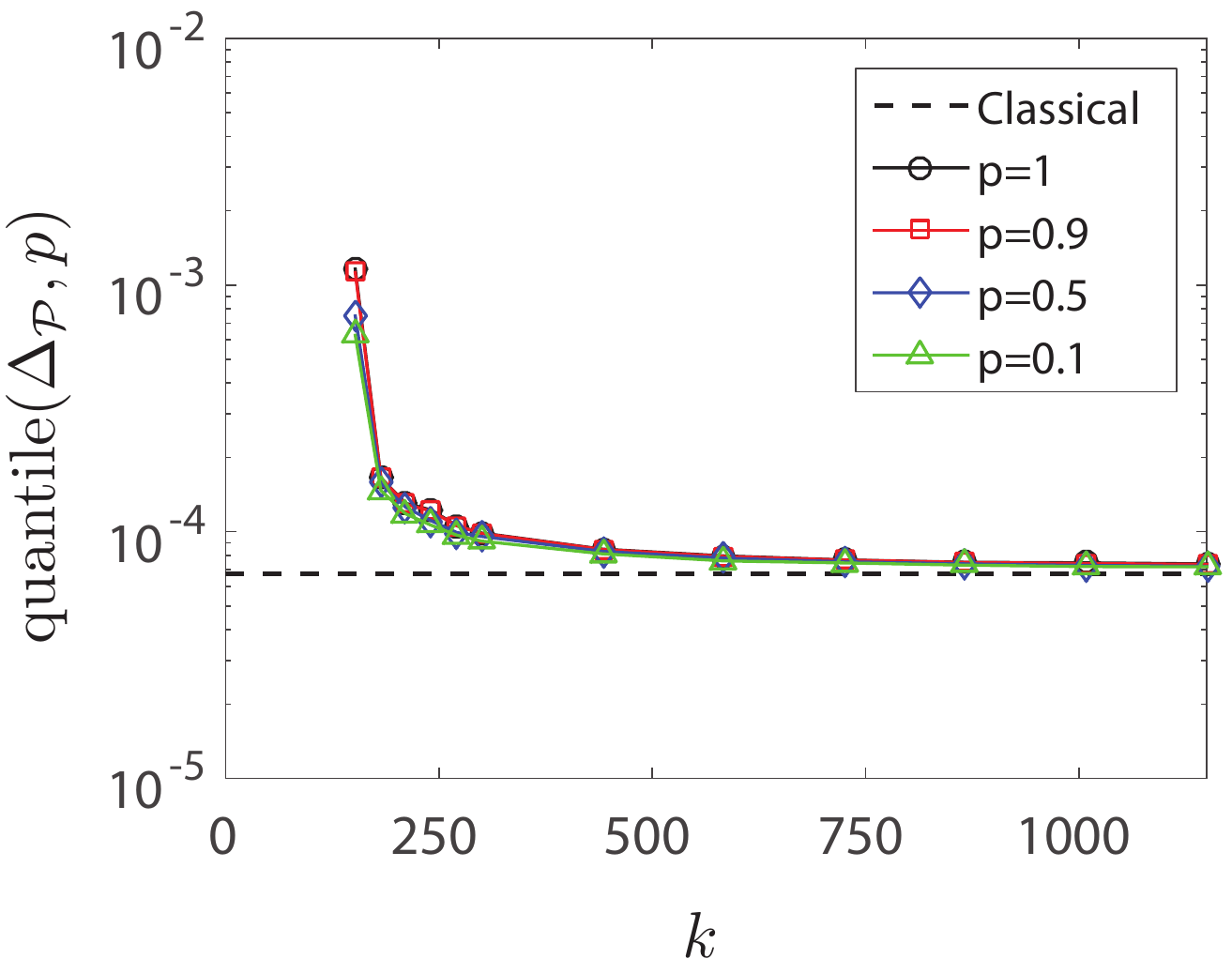}
						\caption{minres, $\Delta_{\mathcal{P}}$}
						\label{fig:Ex1_1a}
					\end{subfigure} \hspace{.01\textwidth}
					\begin{subfigure}[b]{.4\textwidth}
						\centering
						\includegraphics[width=\textwidth]{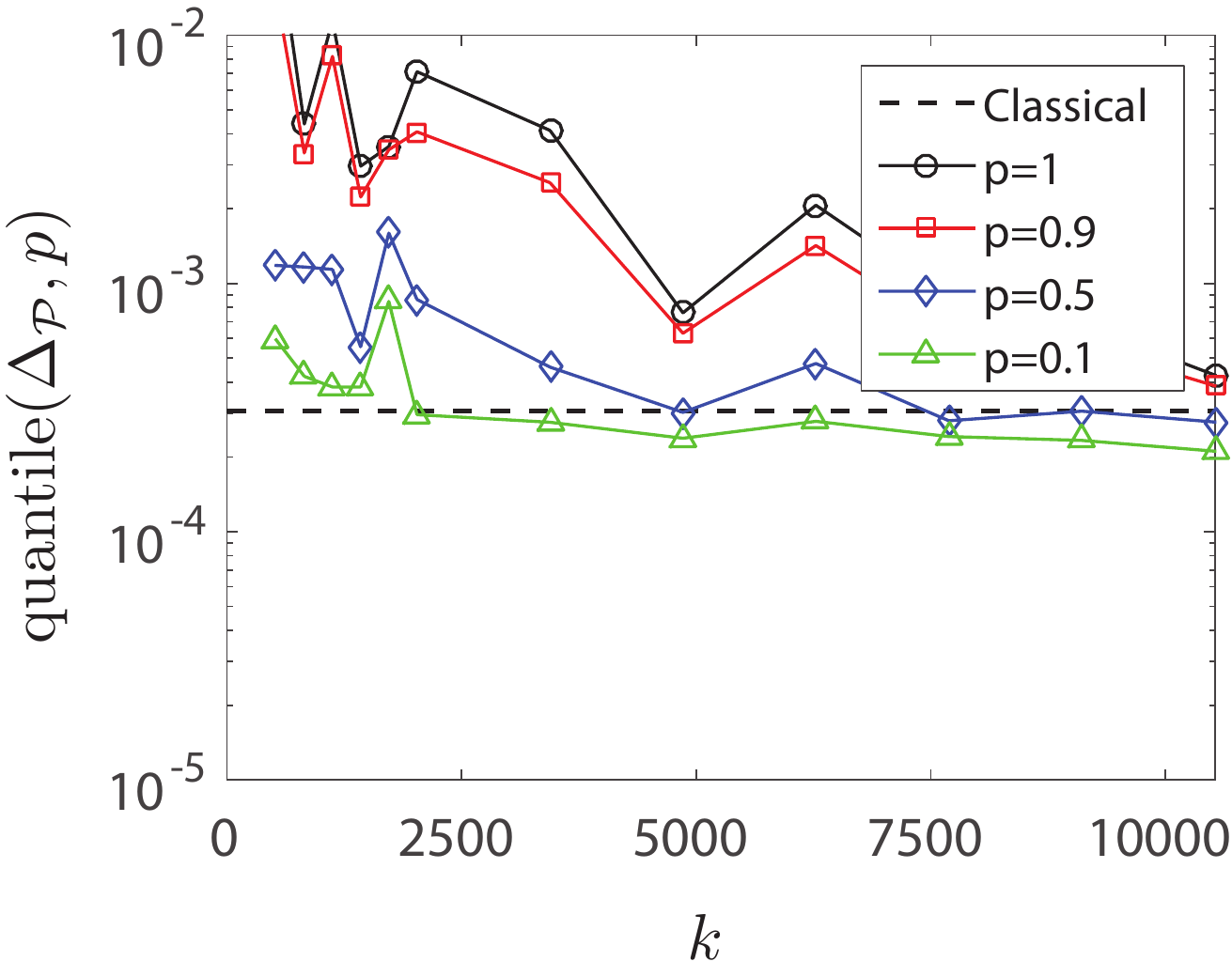}
						\caption{Galerkin, $\Delta_{\mathcal{P}}$}
						\label{fig:Ex1_1b}
					\end{subfigure}	
					\begin{subfigure}[b]{.4\textwidth}
						\centering
						\includegraphics[width=\textwidth]{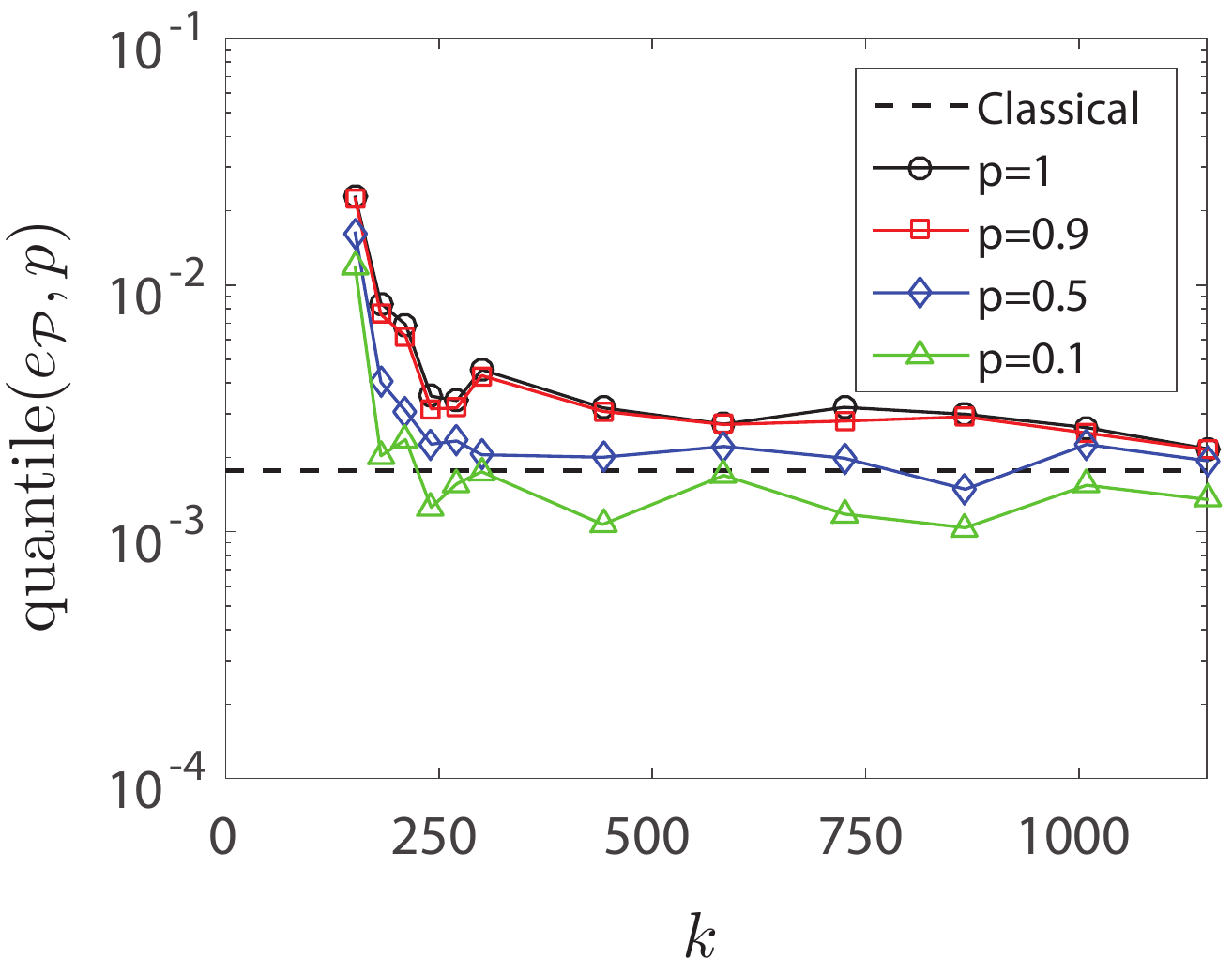}
						\caption{minres, $e_{\mathcal{P}}$}
						\label{fig:Ex1_1c}
					\end{subfigure} \hspace{.01\textwidth}
					\begin{subfigure}[b]{.4\textwidth}
						\centering
						\includegraphics[width=\textwidth]{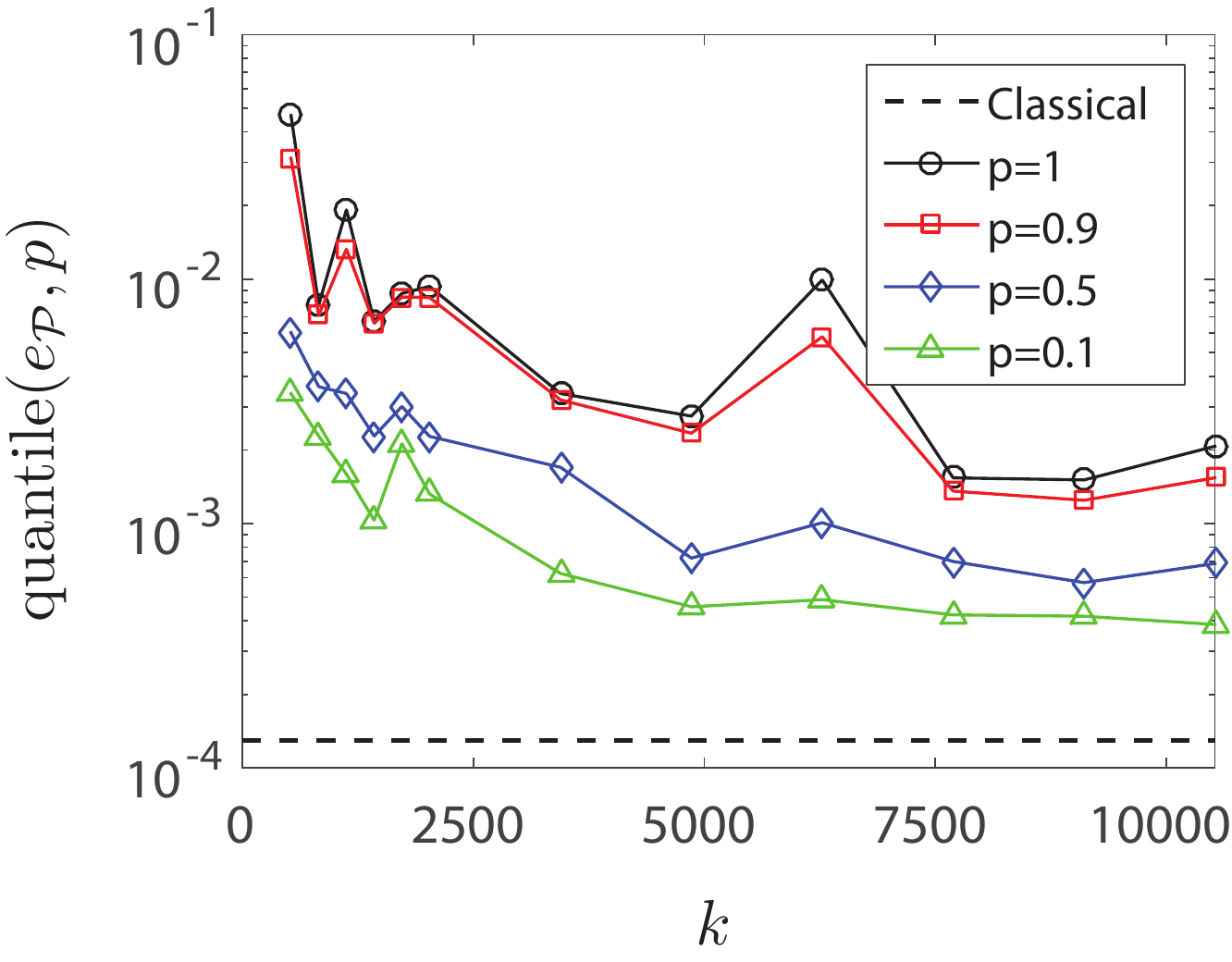}
						\caption{Galerkin, $e_{\mathcal{P}}$}
						\label{fig:Ex1_1d}
					\end{subfigure}
					\begin{subfigure}[b]{.4\textwidth}
						\centering
						\includegraphics[width=\textwidth]{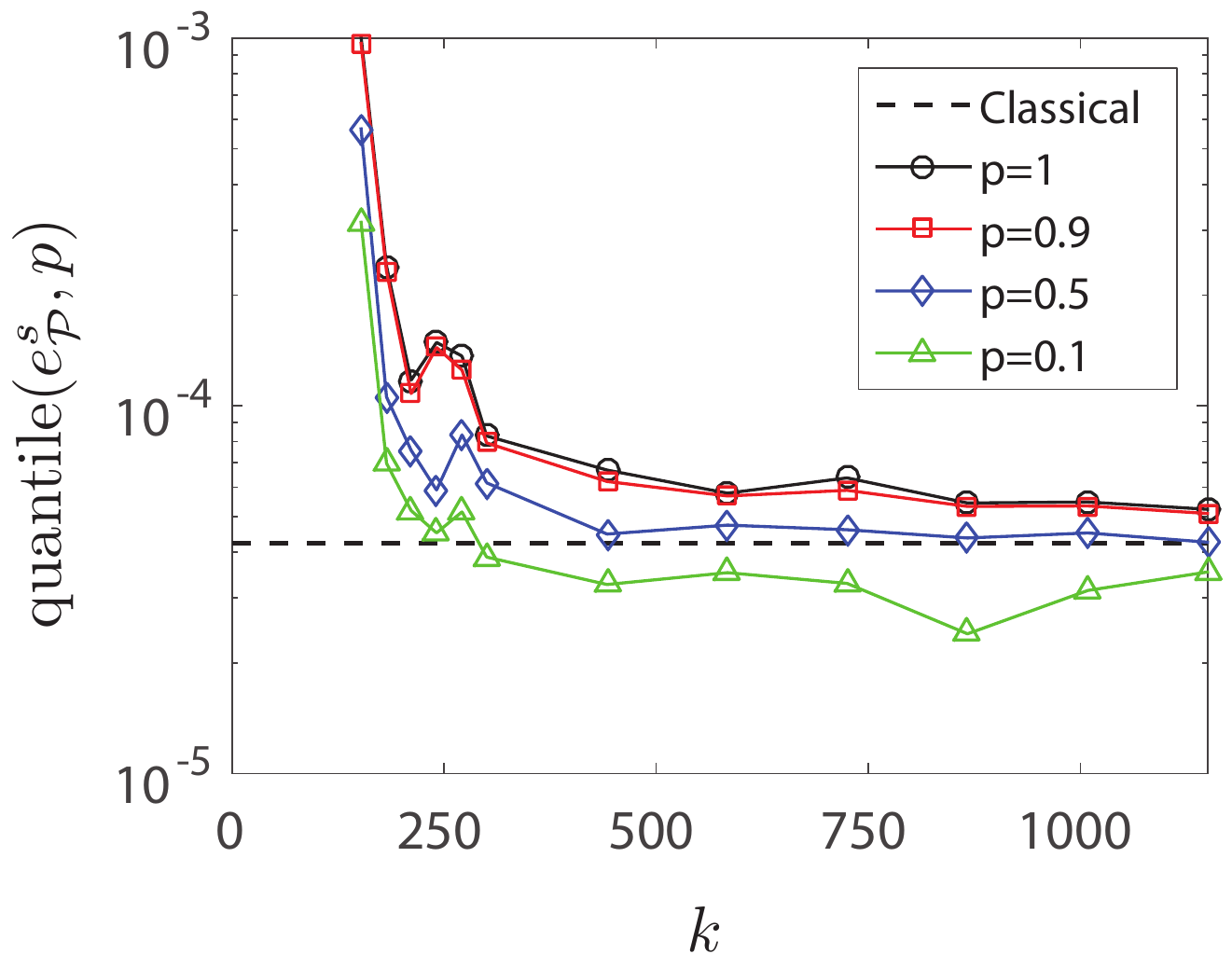}
						\caption{minres, $e^s_{\mathcal{P}}$}
						\label{fig:Ex1_1e}
					\end{subfigure} \hspace{.01\textwidth}
					\begin{subfigure}[b]{.4\textwidth}
						\centering
						\includegraphics[width=\textwidth]{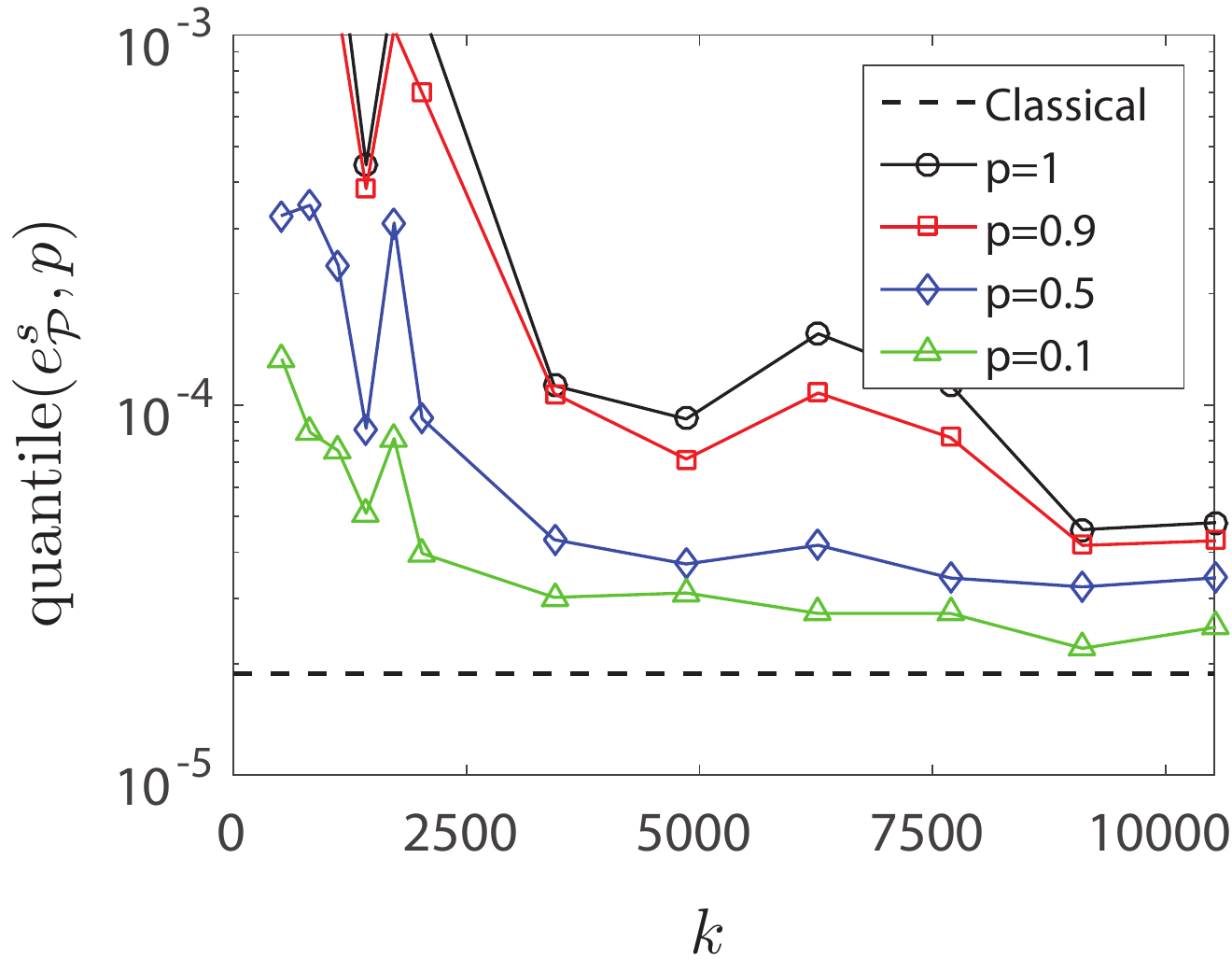}
						\caption{Galerkin, $e^s_{\mathcal{P}}$}
						\label{fig:Ex1_1f}
					\end{subfigure}
					\vspace*{-0.5em}
					\caption{ \small The errors of the classical minres and Galerkin projections and quantiles of probabilities $p=1, 0.9, 0.5$ and $0.1$ over 20 realizations of the errors of the sketched minres and Galerkin projections, versus the number of rows of $\bTheta$.  (a) Residual error $\Delta_\mathcal{P}$ of standard and sketched minres projection. (b) Residual error $\Delta_\mathcal{P}$ of standard and sketched Galerkin projection. (c) Exact error $e_\mathcal{P}$ (in $\| \cdot \|_U$) of standard and sketched minres projection. (d) Exact error $e_\mathcal{P}$ (in $\| \cdot \|_U$) of standard and sketched Galerkin projection. (e) Quantity {of interest} error $e^s_\mathcal{P}$ of standard and sketched minres projection. (f) Quantity {of interest} error $e^s_\mathcal{P}$ of standard and sketched Galerkin projection.}
					\label{fig:Ex1_1}
				\end{figure}
			}

			From~Figure\nobreakspace \ref {fig:Ex1_1} one can clearly report the (essential) preservation of the quality of the classical minres projection for $k\geq500$. Note that for the minres projection a small deviation of $e_\mathcal{P}$ and $e^{s}_\mathcal{P}$ is observed. These errors are higher or lower than the standard values with (almost) equal probability (for $k \geq 500$).  In contrast to the minres projection, the quality of the Galerkin projection is not preserved even for large $k$ up to $10000$. This can be explained by the fact that the approximation of the Galerkin projection with random sketching  is highly sensitive to the properties of the operator, which here is non-coercive and has a high condition number (for some parameter values), while the (essential) preservation of the accuracy of the standard minres projection by its sketched version is guaranteed regardless of the operator's properties. One can clearly see that the sketched minres projection using $\bTheta$ with just $k=500$ rows yields better approximation (in terms of the maximal observed error) of the solution than the sketched Galerkin projection with $k=5000$, even though the standard minres projection is less accurate than the Galerkin one.


			
			As already discussed, random sketching improves not only efficiency but also has an advantage of making the reduced model less sensitive to round-off errors thanks {to possibility of direct online solution of the (sketched) least-squares problem without appealing to the normal equation.} Figure\nobreakspace \ref {fig:Ex1_condnum} depicts the maximal condition number $\kappa_\mathcal{P}$ over $\mathcal{P}_{\mathrm{test}}$ of the reduced matrix $\bV^\bTheta_r(\mu):=\bTheta \bR^{-1}_U \bA(\mu) \bU_r$ associated with the sketched minres projection using reduced basis matrix $\bU_r$ with (approximately) unit-orthogonal columns with respect to $\langle \cdot, \cdot \rangle_U$, for varying sizes of $\bTheta$. We also provide the maximal condition number of the reduced (normal) system of equations associated with the classical minres projection. It is observed that indeed random sketching yields an improvement of numerical stability by a square root.  
			
			\begin{figure}[h!]
				\centering
				\includegraphics[width=0.5\textwidth]{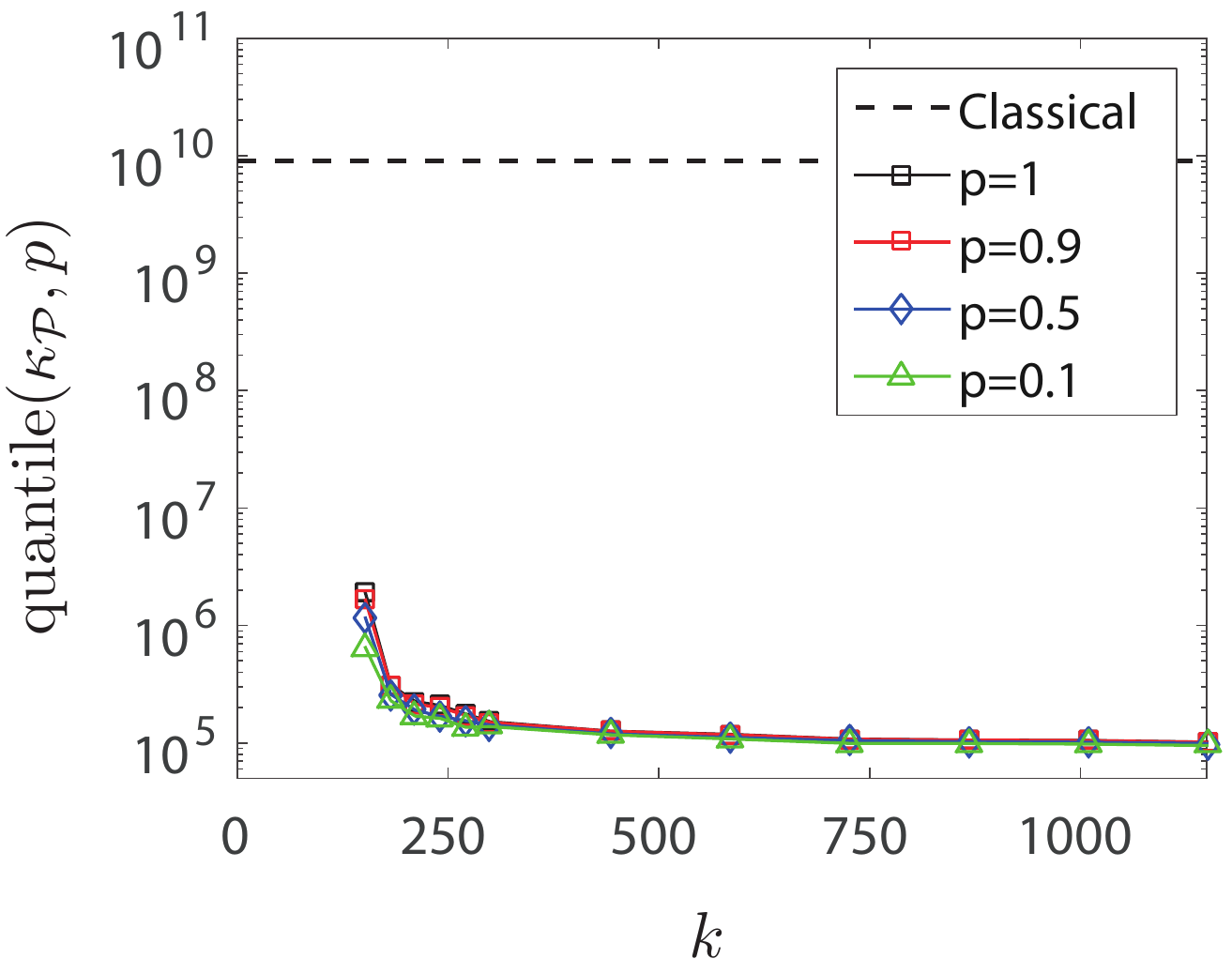}
				\caption{The maximal condition number over $\mathcal{P}_{\mathrm{test}}$ of the reduced (normal) system associated with the classical minres projection and  quantiles of probabilities $p=1, 0.9, 0.5$ and $0.1$ over $20$ realizations of the maximal condition number of the sketched reduced matrix $\bV^\bTheta_r(\mu)$, versus the number of rows $k$ of $\bTheta$.} 
				\label{fig:Ex1_condnum}
			\end{figure}

			\emph{Certification of the sketch.}
			Next the experimental validation of the procedure for a posteriori certification of the $\bTheta$-sketch or the sketched solution (see~Section\nobreakspace \ref {sketchcert}) is addressed. For this, we generated several $\bTheta$ of different sizes $k$ and for each of them computed the sketched minres projections $\bu_r(\mu) \in U_r$ for all $\mu \in \mathcal{P}_\mathrm{test}$. Thereafter~Propositions\nobreakspace \ref {thm:certvectors} and\nobreakspace  \ref {thm:omegaUB}, with $V(\mu):= R_r(U_r;\mu)$ defined by~\textup {(\ref {eq:R_r})}, were considered for certification of the residual error estimates $\| \br(\bu_r(\mu);\mu) \|^\bTheta_{U'}$ or the quasi-optimality of $\bu_r(\mu)$ in the residual error. Oblivious embeddings of varying sizes were tested for $\bTheta^*$. For simplicity it was assumed that all considered $\bTheta^*$ satisfy~\textup {(\ref {eq:Vconcentration})} with $\varepsilon^*=0.05$ and (small) probability of failure $\delta^*$.

			By~Proposition\nobreakspace \ref {thm:certvectors} the certification of the sketched residual error estimator $\| \br(\bu_r(\mu);\mu) \|^\bTheta_{U'}$ can be performed by comparing it to $\| \br(\bu_r(\mu);\mu) \|^{\bTheta^*}_{U'}$. More specifically, by~\textup {(\ref {eq:certvectors})} we have that with probability at least {$1-4\delta^*$},
			\begin{equation}\label{eq:rescert}
			\begin{split}
			&  |\| \br(\bu_r(\mu);\mu) \|^2_{U'} - (\| \br(\bu_r(\mu);\mu) \|^\bTheta_{U'})^2|^{1/2} \\
			&\leq \left( |  (\| \br(\bu_r(\mu);\mu) \|^{\bTheta^*}_{U'})^2 -  (\| \br(\bu_r(\mu);\mu) \|^\bTheta_{U'})^2| 
			+ \frac{\varepsilon^*}{1-\varepsilon^*} (\| \br(\bu_r(\mu);\mu) \|^{\bTheta^*}_{U'})^2 \right)^{1/2}.
			\end{split}
			\end{equation}
			Figure\nobreakspace \ref {fig:Ex1a_2} depicts $d_\mathcal{P} := \max_{\mu \in \mathcal{P}_{\mathrm{test}}} d(\bu_r(\mu); \mu)/ \| \bb(\mu_s) \|_{U'}$, where $d(\bu_r(\mu); \mu)$ is the exact discrepancy $  |\| \br(\bu_r(\mu);\mu) \|^2_{U'} - (\| \br(\bu_r(\mu);\mu) \|^\bTheta_{U'})^2|^{1/2}$ or its (probabilistic) upper bound in~\textup {(\ref {eq:rescert})}. For each $\bTheta$ and  $k^*$, $20$ realizations of $d_\mathcal{P}$ were computed for statistical analysis.  
			We see that (sufficiently) tight upper bounds for $ |\| \br(\bu_r(\mu);\mu) \|^2_{U'} - (\| \br(\bu_r(\mu);\mu) \|^\bTheta_{U'})^2|^{1/2}$ were obtained already when $k^*\geq 100$, which is in particular several times smaller than the size of $\bTheta$ required for quasi-optimality of $\bu_r(\mu)$. This implies that the certification of the effectivity of the error estimator $\| \br(\bu_r(\mu);\mu) \|^\bTheta_{U'}$ by $\| \br(\bu_r(\mu);\mu) \|^{\bTheta^*}_{U'}$ should require negligible computational costs compared to the cost of obtaining the solution (or estimating the error in adaptive algorithms such as greedy algorithms).   
			
			By~Proposition\nobreakspace \ref {thm:skminresopt}, the quasi-optimality of $\bu_r(\mu)$ can be guaranteed if $\bTheta$ is an $\varepsilon$-embedding for $V(\mu)$. The $\varepsilon$-embedding property of each $\bTheta$ was verified with~Proposition\nobreakspace \ref {thm:omegaUB}. In~Figure\nobreakspace \ref {fig:Ex1_3} we provide $\omega_{\mathcal{P}}:= \max_{\mu \in \mathcal{P}_{\mathrm{test}}}{\tilde{\omega}(\mu)}$ where $\tilde{\omega}(\mu)=\omega(\mu)$, which is the minimal value for $\varepsilon$ such that $\bTheta$ is an $\varepsilon$-embedding for $V(\mu)$, or $\tilde{\omega}(\mu)=\bar{\omega}(\mu)$, which is the upper bound of $\omega(\mu)$ computed with~\textup {(\ref {eq:compute_omega})} using $\bTheta^*$ of varying sizes. For illustration purposes we here allow the value $\varepsilon$ in~Definition\nobreakspace \ref {def:epsilon_embedding} to be larger than $1$.
			The statistical properties of ${\omega}_{\mathcal{P}}$ were obtained with $20$ realizations for each $\bTheta$ and value of $k^*$.
			Figure\nobreakspace \ref {fig:Ex1_3a} depicts the statistical characterization of ${\omega}_{\mathcal{P}}$ for $\bTheta$ of size $k=5000$. The maximal value of ${\omega}_{\mathcal{P}}$ observed for each $k^*$ and $\bTheta$ is presented in~Figure\nobreakspace \ref {fig:Ex1_3b}. It is observed that with a posteriori estimates from~Proposition\nobreakspace \ref {thm:omegaUB} using $\bTheta^*$ of size $k^*=6000$, we here can guarantee with high probability that $\bTheta$ with $k=5000$ satisfies an $\varepsilon$-embedding property for $\varepsilon \approx 0.6$. The theoretical bounds from~\cite{balabanov2019galerkin} for $\bTheta$ to be an $\varepsilon$-embedding for $V(\mu)$ with $\varepsilon=0.6$ yield much larger sizes, namely, for the probability of failure $\delta \leq 10^{-6}$, they require more than $k=45700$ rows for Gaussian matrices and $k=102900$ rows for SRHT. This proves~Proposition\nobreakspace \ref {thm:omegaUB} to be very useful for the adaptive selection of sizes of random matrices or for the certification {of the sketched inner product $\langle \cdot, \cdot \rangle^\bTheta_U$ for all vectors in $V$.} Note that the adaptive selection of the size of $\bTheta$ can also be performed without requiring $\bTheta$ to be an $\varepsilon$-embedding for $V(\mu)$ with $\varepsilon<1$, based on the observation that oblivious embeddings yield preservation of the quality of the minres projection when they are $\varepsilon$-embeddings for $V(\mu)$ with small $\varepsilon$, which is possibly larger than $1$ {(see Remark\nobreakspace \ref {rmk:altquasiopt}).}
			
			\begin{figure}[h!]
				\centering
				\begin{subfigure}[b]{.4\textwidth}
					\centering
					\includegraphics[width=\textwidth]{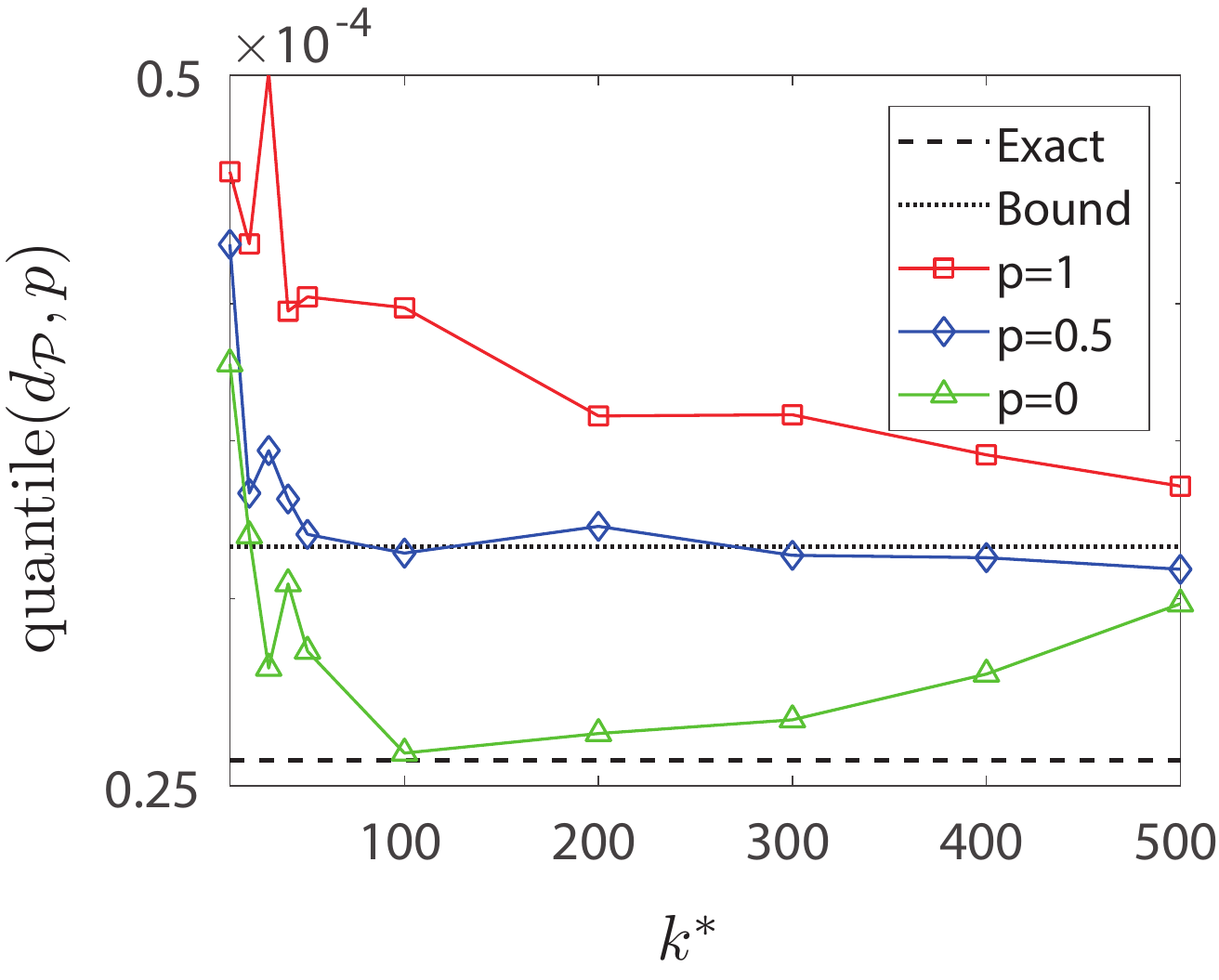}
					\caption{}
					\label{fig:Ex1a_2a}
				\end{subfigure} \hspace{.01\textwidth}
				\begin{subfigure}[b]{.4\textwidth}
					\centering
					\includegraphics[width=\textwidth]{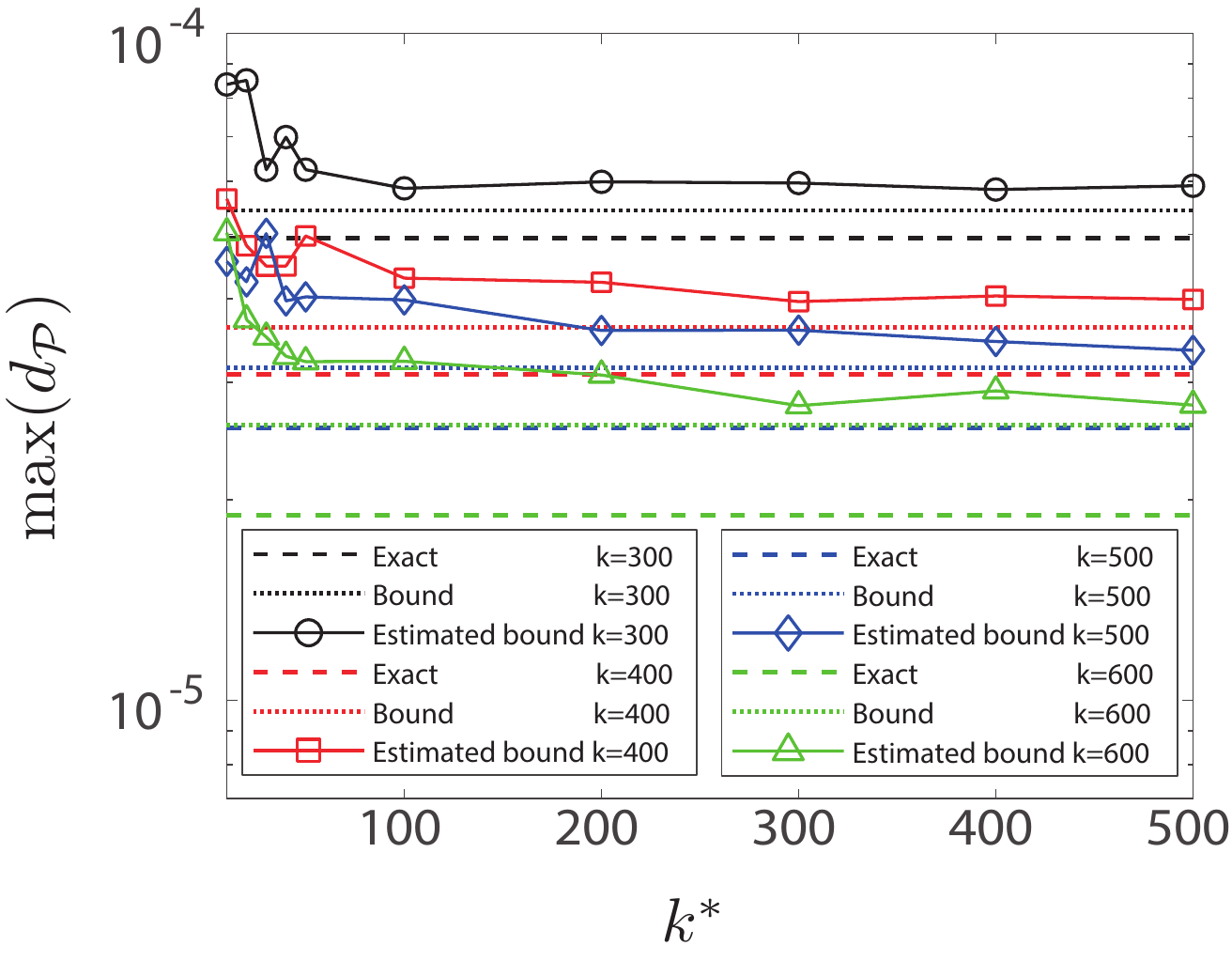}
					\caption{}
					\label{fig:Ex1a_2b}
				\end{subfigure}	
				\caption{The discrepancy $  |\| \br(\bu_r(\mu);\mu) \|^2_{U'} - (\| \br(\bu_r(\mu);\mu) \|^\bTheta_{U'})^2|^{1/2}$ between the residual error and the sketched error estimator with $\bTheta$, and the upper bound of this value computed with~\textup {(\ref {eq:rescert})}.
					(a) The exact discrepancy for $\bTheta$ with {$k=500$} rows, the upper bound~\textup {(\ref {eq:rescert})} of this discrepancy taking $\| \cdot \|^{\bTheta^*}_{U'} = \| \cdot \|_{U'}$, and quantiles of probabilities $p=1,0.5$ and $0$ (i.e., the observed maximum, median and minimum) over $20$ realizations of the (probabilistic) upper bound~\textup {(\ref {eq:rescert})} versus the size of $\bTheta^*$.
					(b) The exact discrepancy, the upper bound~\textup {(\ref {eq:rescert})} taking $\| \cdot \|^{\bTheta^*}_{U'} = \| \cdot \|_{U'}$, and the maximum of $20$ realizations of the (probabilistic) upper bound~\textup {(\ref {eq:rescert})} versus the number of rows $k^*$ of $\bTheta^*$ for varying sizes {$k$} of $\bTheta$.}
				\label{fig:Ex1a_2}
			\end{figure}
			
			\begin{figure}[h!]
				\centering
				\begin{subfigure}[b]{.4\textwidth}
					\centering
					\includegraphics[width=\textwidth]{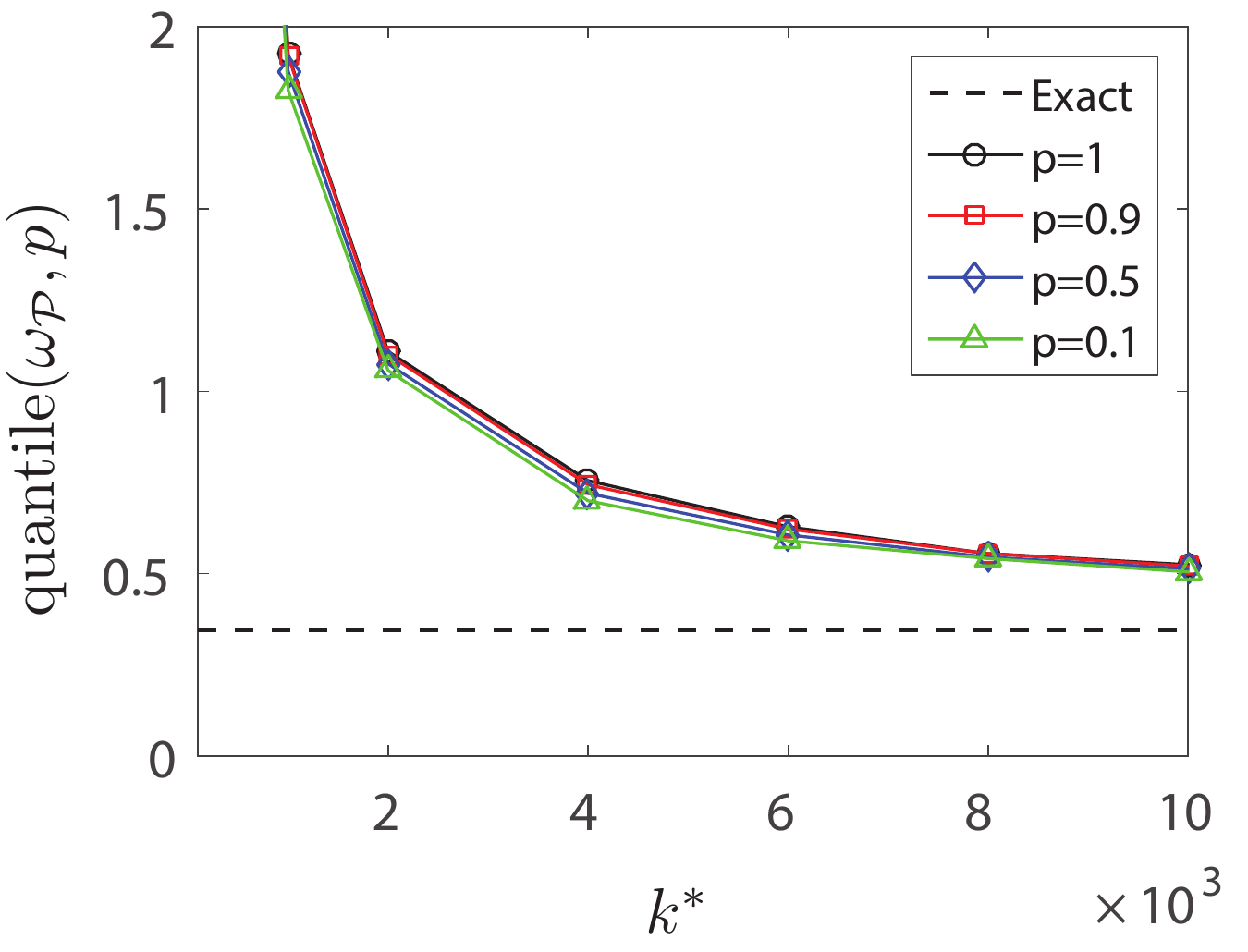}
					\caption{}
					\label{fig:Ex1_3a}
				\end{subfigure} \hspace{.01\textwidth}
				\begin{subfigure}[b]{.4\textwidth}
					\centering
					\includegraphics[width=\textwidth]{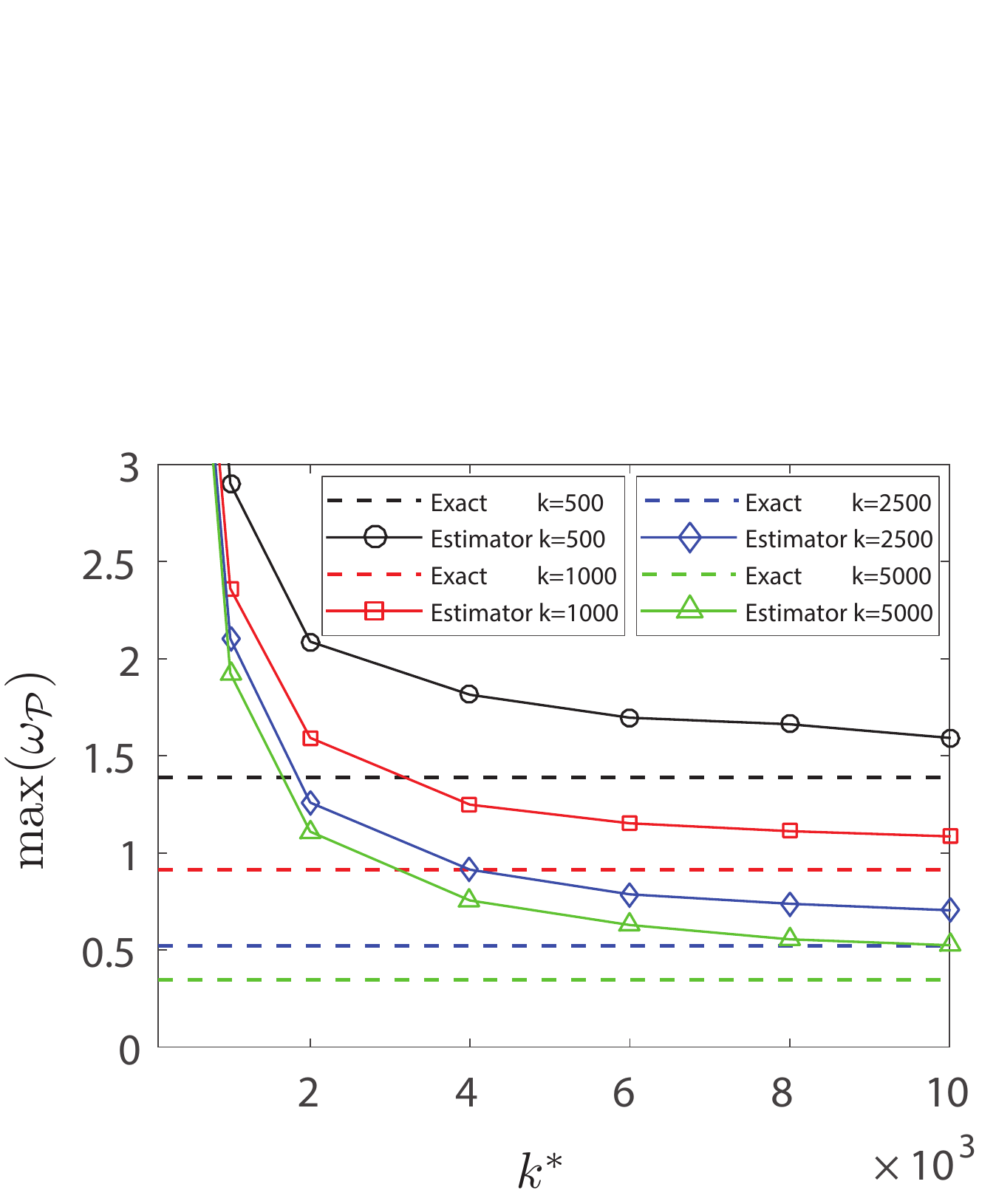}
					\caption{}
					\label{fig:Ex1_3b}
				\end{subfigure}	
				\caption{The minimal value for $\varepsilon$ such that $\bTheta$ is an $\varepsilon$-embedding for $V_r(\mu)$ for all $\mu \in \mathcal{P}_\mathrm{test}$, and a posteriori random estimator of this value obtained with the procedure from~Section\nobreakspace \ref {sketchcert} using $\bTheta^*$ with $k^*$ rows. (a) The minimal value for $\varepsilon$ for $\bTheta$ with $k=5000$ rows and quantiles of probabilities $p=1,0.9,0.5$ and $0.1$ over $20$ samples of the estimator, versus the size of $\bTheta^*$. (b) The minimal value for $\varepsilon$  and the maximum of $20$ samples of the estimator, versus the number of rows of $\bTheta^*$ for varying sizes of $\bTheta$. }
				\label{fig:Ex1_3}
			\end{figure}

			\begin{remark} \label{rmk:altquasiopt}
				Throughout the paper the quality of $\bTheta$ (e.g., for approximation of minres projection in~Section\nobreakspace \ref {skminres}) was characterized by  $\varepsilon$-embedding property. However, for this numerical benchmark the sufficient size for $\bTheta$ to be an $\varepsilon$-embedding for $V(\mu)$ is in several times larger than the one yielding an accurate approximation of the minres projection. In particular, $\bTheta$ with $k=500$ rows provides {with high probability} an approximation with residual error very close to the minimal one, but it does not satisfy an $\varepsilon$-embedding property (with $\varepsilon <1$), which is required for guaranteeing the quasi-optimality of $\bu_r(\mu)$ with~Proposition\nobreakspace \ref {thm:skminresopt}. A more {reliable} way for certification of the quality of $\bTheta$ for approximation of the minres projection onto $U_r$ can be derived by taking into account that $\bTheta$ was generated from a distribution of oblivious embeddings. In such a case it is enough to only certify that $\| \cdot \|_{U}^\bTheta$ provides an approximate upper bound of $\| \cdot \|_{U}$ for all vectors in $V(\mu)$ without the need to guarantee that $\| \cdot \|_{U}^\bTheta$ is an approximate lower bound (that {is} in practice  the main bottleneck). This approach is outlined below.  
				
				We first observe that $\bTheta$ was generated from a distribution of random matrices such that for all $\bx \in V(\mu)$, we have $$ \mathbb{P} \left ( \left | \| \bx \|^2_U -  (\| \bx \|^{\bTheta}_U)^2 \right |\leq \varepsilon_0 \| \bx \|^2_U \right ) \geq 1-\delta_0. $$
				The values $\varepsilon_0$ and $\delta_0$ can be obtained from the theoretical bounds from~\cite{balabanov2019galerkin} or practical experience.  Then one can show that for the sketched minres projection $\bu_r(\mu)$ associated with $\bTheta$, the {inequality}
				\begin{equation} \label{eq:altquasiopt}
				\| \br(\bu_r(\mu); \mu) \|_{U'} \leq \sqrt{\frac{1+\varepsilon_0}{1-\omega(\mu)}} \min_{\bx \in U_r}\| \br(\bx ; \mu) \|_{U'},  
				\end{equation}
				holds with probability at least $1-\delta_0$, where $\omega(\mu)<1$ is the minimal value for $\varepsilon$ such that for all $\bx \in V(\mu)$
				$$(1- \varepsilon) \| \bx \|^2_U \leq (\| \bx \|^\bTheta_U)^2. $$ 
				The quasi-optimality of $\bu_r(\mu)$ in the norm $\| \cdot \|_U$ rather than the residual norm can be readily derived from relation~\textup {(\ref {eq:altquasiopt})} by using the equivalence between the residual norm and the error in $\| \cdot \|_U$.

				In this way a characterization of the quasi-optimality of the sketched minres projection with $\bTheta$ can be obtained from the a posteriori upper bound of $\omega(\mu)$ in~\textup {(\ref {eq:altquasiopt})}.  Note that since $\bTheta$ is an oblivious subspace embedding, the parameters $\varepsilon_0$ and $\delta_0$ do not depend on the dimension of $V(\mu)$, which implies that the considered value for $\varepsilon_0$ should {be an order of magnitude less than $\omega(\mu)$.} Therefore, it can be a good way to choose $\varepsilon_0$ as $\omega(\mu)$ (or rather its upper bound) {multiplied by a small factor, say $0.1$.}
				
				The (probabilistic) upper bound $\bar{\omega}(\mu)$ for $\omega(\mu)$ can be obtained a posteriori by following a similar procedure as the one from~Proposition\nobreakspace \ref {thm:omegaUB} described for {verification of the $\varepsilon$-embedding property}.  More precisely, we can use similar arguments as in~Proposition\nobreakspace \ref {thm:omegaUB}  to show that 
				$$\bar{\omega}(\mu): =  1- (1-\varepsilon^*) \min_{\bx \in V / \{ \bnull \}} \left (\frac{\| \bx\|^{\bTheta}_U}{\| \bx\|^{\bTheta^*}_U} \right )^2$$
				is an upper bound for $\omega(\mu)$ with probability at least $1-\delta^*$.
				
				Let us now provide experimental validation of the proposed approach. For this we considered same sketching matrices $\bTheta$ as in the previous experiment for validation of the $\varepsilon$-embedding property.	For each $\bTheta$  we computed $\omega_{\mathcal{P}}:= \max_{\mu \in \mathcal{P}_{\mathrm{test}}}{\tilde{\omega}(\mu)}$, where $\tilde{\omega}(\mu) = {\omega}(\mu)$  or its upper bound $\bar{\omega}(\mu)$ using $\bTheta^*$ of different sizes (see~Figure\nobreakspace \ref {fig:Ex1_32}). Again $20$ realizations of ${\omega}_{\mathcal{P}}$ were considered for the statistical characterization of ${\omega}_{\mathcal{P}}$ for each $\bTheta$ and size of $\bTheta^*$. One can clearly see that the present approach provides better estimation of the quasi-optimality constants than the one with the $\varepsilon$-embedding property. In particular, the quasi-optimality guarantee  for $\bTheta$ with $k=500$ rows is experimentally verified. Furthermore, we see that in all the experiments the a posteriori estimates are lower than $1$ even for $\bTheta^*$ of small sizes, yet they are larger than the exact values, which implies efficiency and robustnesses of the method. From~Figure\nobreakspace \ref {fig:Ex1_32}, a good accuracy of a posteriori estimates is with high probability attained for $k^* \geq k/2$.

				\begin{figure}[h!]
					\centering
					\begin{subfigure}[b]{.4\textwidth}
						\centering
						\includegraphics[width=\textwidth]{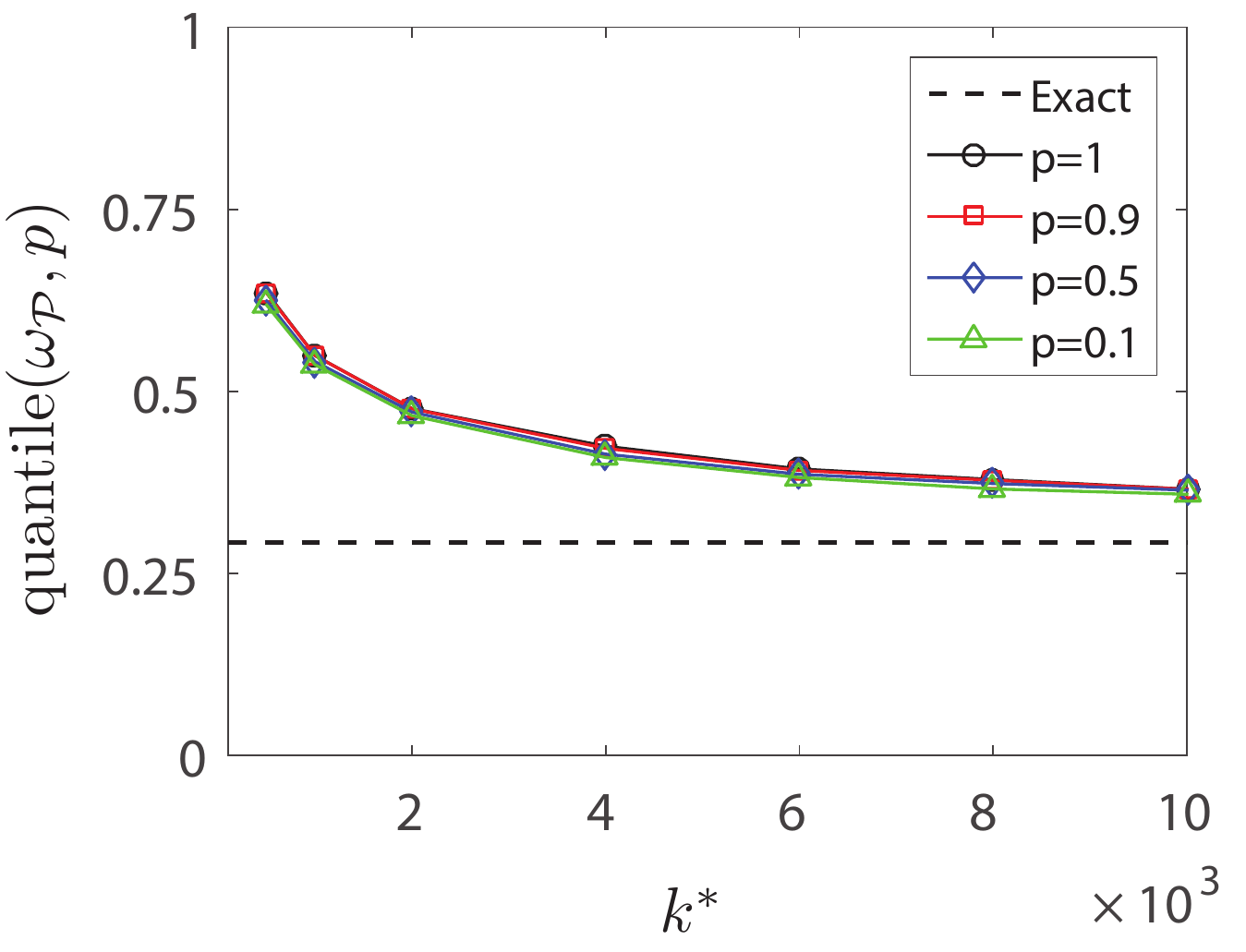}
						\caption{}
						\label{fig:Ex1_32a}
					\end{subfigure} \hspace{.01\textwidth}
					\begin{subfigure}[b]{.4\textwidth}
						\centering
						\includegraphics[width=\textwidth]{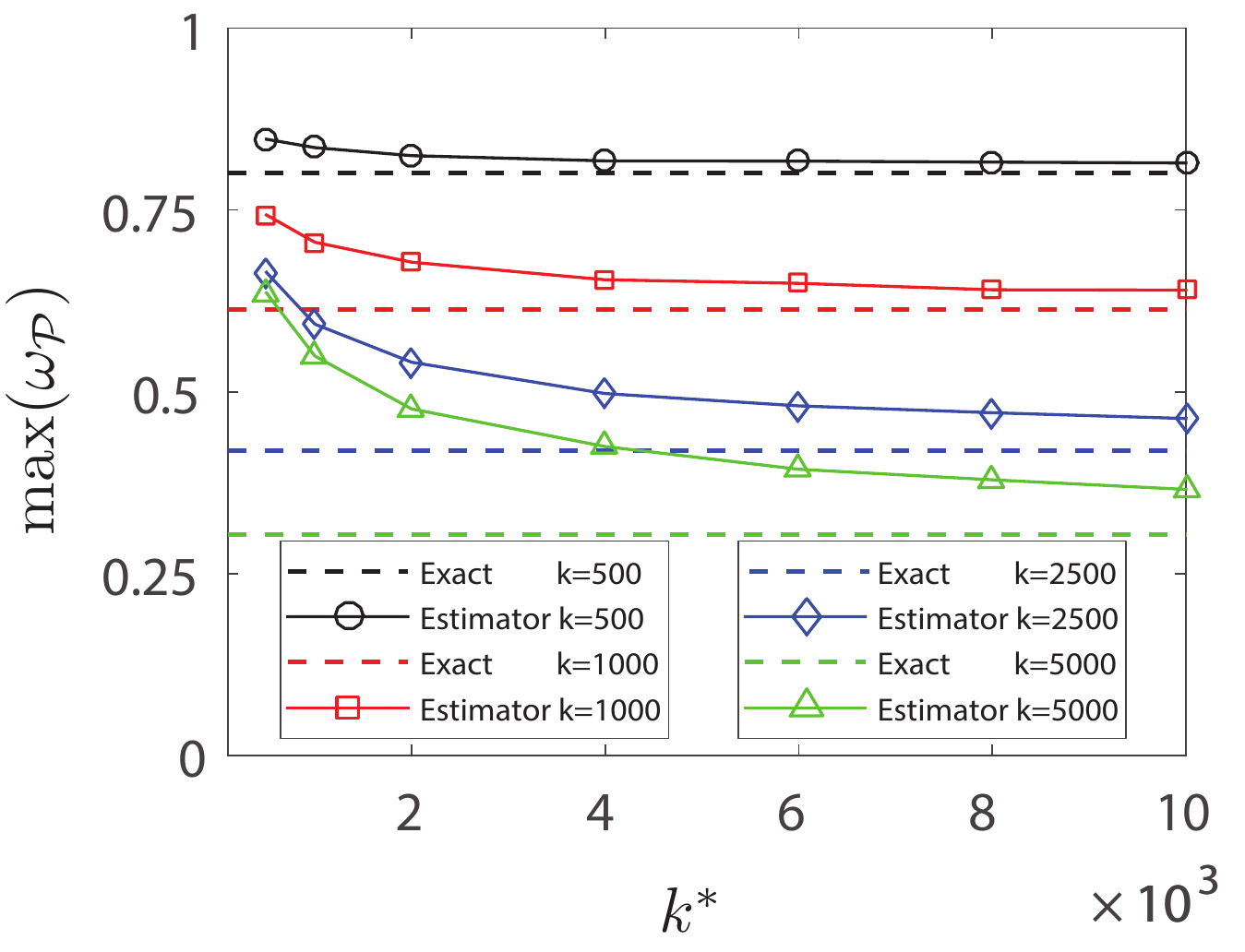}
						\caption{}
						\label{fig:Ex1_32b}
					\end{subfigure}	
					\caption{
						{The minimal value for $\varepsilon$ such that $(1-\varepsilon) \| \bx \|^2_{U} \leq (\| \bx \|^\bTheta_{U})^2$ holds for all $\bx \in V_r(\mu)$ and $\mu \in \mathcal{P}_\mathrm{test}$, and a posteriori random estimator of this value using $\bTheta^*$ with $k^*$ rows.} (a) The minimal value for $\varepsilon$ for $\bTheta$ with $k=5000$ rows and quantiles of probabilities $p=1,0.9,0.5$ and $0.1$ over $20$ realizations of the estimator, versus the size of $\bTheta^*$. (b) The minimal value for $\varepsilon$  and the maximum of $20$ realizations of the estimator versus the number of rows of $\bTheta^*$, for varying sizes of $\bTheta$.}
					\label{fig:Ex1_32}
				\end{figure}
				
			\end{remark}

			\emph{Computational costs.}
			For this benchmark, random sketching yielded drastic computational savings in the offline stage and considerably improved online efficiency. 
			To verify the gains for the offline stage, we executed two greedy algorithms for the generation of the reduced approximation space of dimension $r=150$ based on the minres projection and the sketched minres projection, respectively. The standard algorithm resulted in a computational burden after reaching $96$-th iteration due to exceeding the limit of RAM ($16$GB). {This took more than $3$ hours of runtime.}   Note that performing $r=150$ iterations in this case would require around $25$GB of RAM (mainly utilized for storage of the affine factors of $\bR^{-1}_U \bA(\mu) \bU_r$) {and more than $7$ hours of runtime}. In contrast to the standard method, conducting $r=150$ iterations of a greedy algorithm with random sketching using $\bTheta$ of size $k=2500$ (and $\bGamma$ of size $k'=500$) for the sketched minres projection and $\bTheta^*$ of size $k^*=250$ for the error certification, took only $0.65$GB of RAM. Moreover, the sketch required only a minor part ($0.2$GB) of the aforementioned amount of memory, while the major part was consumed by the {initialization and maintenance} of the full order model. The sketched greedy algorithm had a total runtime of $1.9$ hours{, which is less than the (expected) runtime for the standard greedy algorithm in more than $3.5$ times. From these $1.9$ hours,} $0.8$ hours was spent on computation of $150$ snapshots,  $0.2$ hours on provisional online solutions and $0.9$ hours on random projections. 
			
			Next the improvement of online computational cost of minres projection is addressed. For this, we computed the reduced solutions on the test set with {a} standard method, which consists in assembling the reduced system of equations (representing the normal equation) from its affine decomposition (precomputed in the offline stage) and its subsequent solution with built in Matlab\textsuperscript{\textregistered} {R2017b} linear solver. The online solutions on the test set were additionally computed with the sketched method for comparison of runtimes and storage requirements. For this, for each parameter value, the reduced least-squares problem was assembled from the precomputed affine decompositions of $\bV_r^\bPhi(\mu)$ and $\bb^\bPhi(\mu)$ and solved with the normal equation using {the} built in Matlab\textsuperscript{\textregistered} {R2017b} linear solver.
			Note that both methods proceeded with the normal equation. The difference was in the way how this equation was obtained. For the standard method it was directly assembled from the affine representation, while for the sketched method it was computed from the sketched matrices $\bV_r^\bPhi(\mu)$ and $\bb^\bPhi(\mu)$.
			
			Table\nobreakspace \ref {tab:onlineruntimes} depicts the runtimes and memory consumption taken by the standard and sketched {minres} online stages for varying sizes of the reduced space and $\bPhi$ (for the sketched method).  The sketch's sizes were picked such that the associated reduced solutions with high probability had almost (higher by at most a factor of $1.2$) optimal residual error. Our approach nearly {divided by 3} the online runtime for all values of $r$ from~Table\nobreakspace \ref {tab:onlineruntimes}. Furthermore, the improvement of memory requirements was even greater. For instance, for $r=150$ the online memory consumption was divided {by} $6.8$. {In~Table\nobreakspace \ref {tab:onlineruntimes} we also provide the cost of standard/sketched Galerkin online stage, which is more efficient than the other two due to less cost of forming the reduced system of equations (but can have worse quasi-optimality constants for non-coercive or ill-conditioned problems).}

			\begin{table}[tbhp]
				\caption{CPU times in seconds and amount of memory in MB taken by the standard {minres, sketched minres and Galerkin} online solvers for the solutions on the test set. }
				\label{tab:onlineruntimes}
				\centering
				\scalebox{0.85}{
					\begin{tabular}{|c|r|r|r|r|r|r|r|r|r|} \hline
						\multirow{2}{*}{} & \multicolumn{3}{c|}{Standard minres} & \multicolumn{3}{c|}{Sketched minres} & \multicolumn{3}{c|}{Galerkin}  \\  \cline{2-10}
						&$r=50$  & $r=100$ & $r=150$ & $\begin{array}{ll} r&=50 \\ k'&=300 \end{array}$ & $\begin{array}{ll} r&=100 \\ k'&=400 \end{array}$ &  $\begin{array}{ll} r&=150 \\ k'&=500 \end{array}$ & $r=50$  & $r=100$ & $r=150$  \\  [2pt]  \hline 
						CPU & $1$ & $3.6$ & $7.8$ & $0.4$ & $1.3$ & $2.2$ & $0.25$ & $0.5$ & $0.95$ \\ \cline{1-10}   
						Storage & $22$ & $87$ & $193$ & $5.8$ & $15$ & $28$ & $1$ & $3.8$ & $8.5$  \\ \cline{1-10}
						\hline 
				\end{tabular}}
			\end{table}

			\subsection{Advection-diffusion problem}
			The dictionary-based approximation method proposed in~Section\nobreakspace \ref {dbminres} is validated on a $2$D advection dominated advection-diffusion problem defined on a complex flow. This problem is governed by the following equations
			
			\begin{equation} \label{eq:BVP2}
			\left \{
			\begin{array}{rll}
			-\epsilon\mathrm{\Delta} u +\bm{\beta} \cdot \nabla u &= f,~~  & \textup{in } \Omega \\
			u &=0,~~ & \textup{on } \Gamma_{out} \\
			\frac{\partial u}{\partial \boldsymbol{n}} &= 0,~~ & \textup{on } \Gamma_{n},
			\end{array}
			\right.
			\end{equation}
			where $u$ is the unknown (temperature) field, $\epsilon=0.0001$ is the diffusion coefficient and $\bm{\beta}$ is the advection field. The geometry of the problem is as follows.
			First we have $5$ circular pores of radius $0.01$ located at points $\boldsymbol{x}_j = 0.5 \left( \cos (2 \pi j/5), \sin (2 \pi j/5) \right)$, $1 \leq j \leq 5$. The domain of interest is then defined as the square $[-10, 10]^2$ without the pores, i.e,   $\Omega:=[-10, 10]^2/\Omega_n$, with $\Omega_n:= \cup_{1\leq j \leq 5} \{ \boldsymbol{x} : \|\boldsymbol{x} - \boldsymbol{x}_j\| \leq 0.01 \}$. The boundaries $\Gamma_n$ and $\Gamma_{out}$ are taken as $\partial{\Omega_n}$ and $\partial{\Omega}/\partial{\Omega_n}$, respectively. Furthermore, $\Omega$ is (notationally) divided into the main region inside $[-1, 1]^2$, and the outer domain playing a role of a boundary layer. Finally, the force term $f$ is nonzero in the disc $\Omega_s:=\{ \boldsymbol{x} \in \Omega: \| \boldsymbol{x} \| \leq 0.025 \}$. The geometric setup of the problem is presented in~Figure\nobreakspace \ref {fig:Ex2_intial_problem_a}. 
			
			\begin{figure}[htp]
				\centering
				\begin{minipage}[c]{.4\textwidth}
					\centering
					\includegraphics[width=0.96\textwidth]{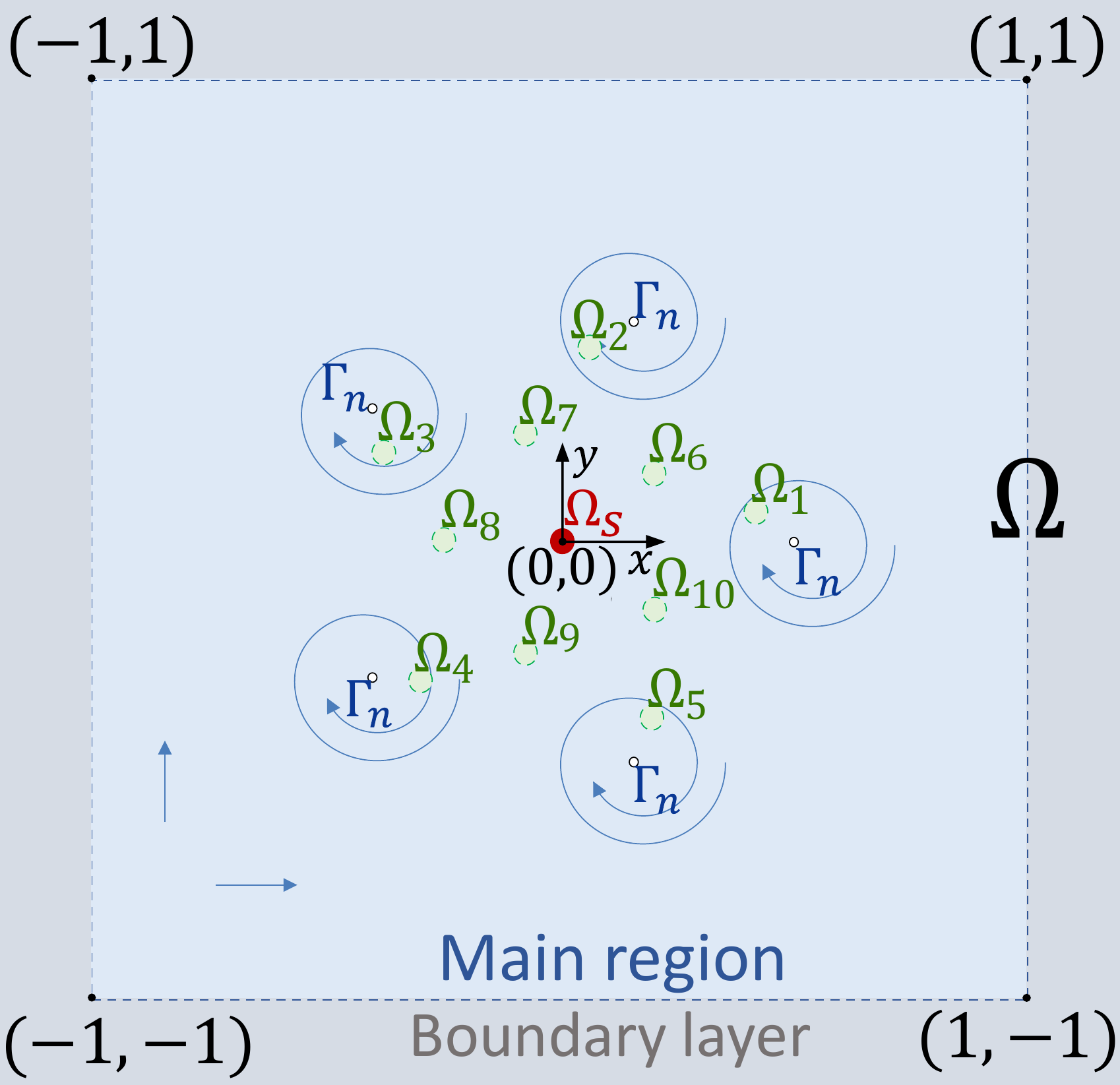}
				\end{minipage}%
				\hspace{0.05\textwidth}
				\begin{minipage}[c]{.4\textwidth}
					\centering
					\includegraphics[width=1\textwidth]{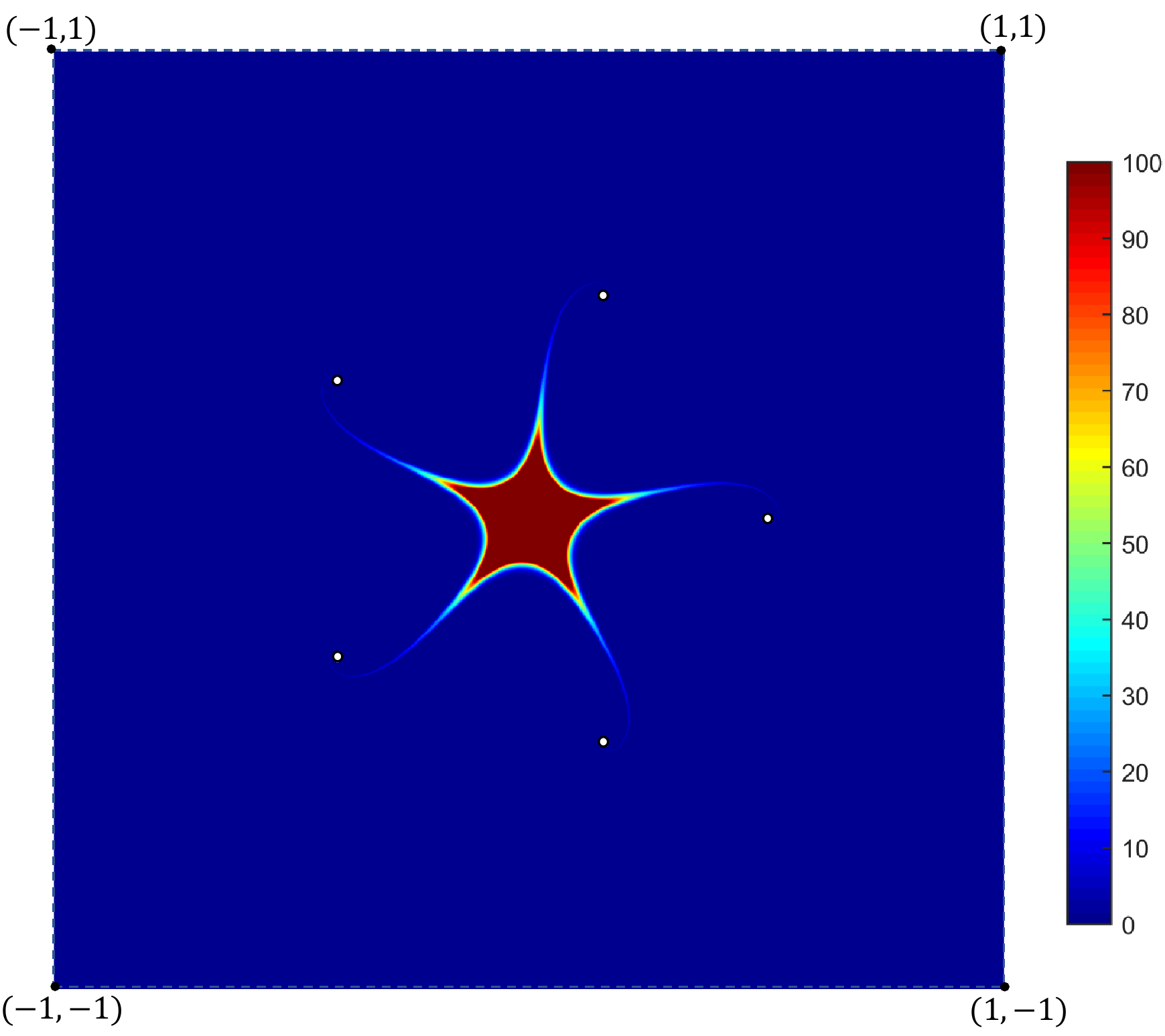}
				\end{minipage} \\ 
				\begin{minipage}[b]{.4\textwidth}
					\centering
					\subcaption{Geometry}
					\label{fig:Ex2_intial_problem_a}
				\end{minipage}%
				\begin{minipage}[b]{.4\textwidth}
					\subcaption{Snapshot at $\mu_s$}
					\label{fig:Ex2_intial_problem_b}
				\end{minipage} 
				\begin{minipage}[c]{.4\textwidth}
					\centering
					\includegraphics[width=1\textwidth]{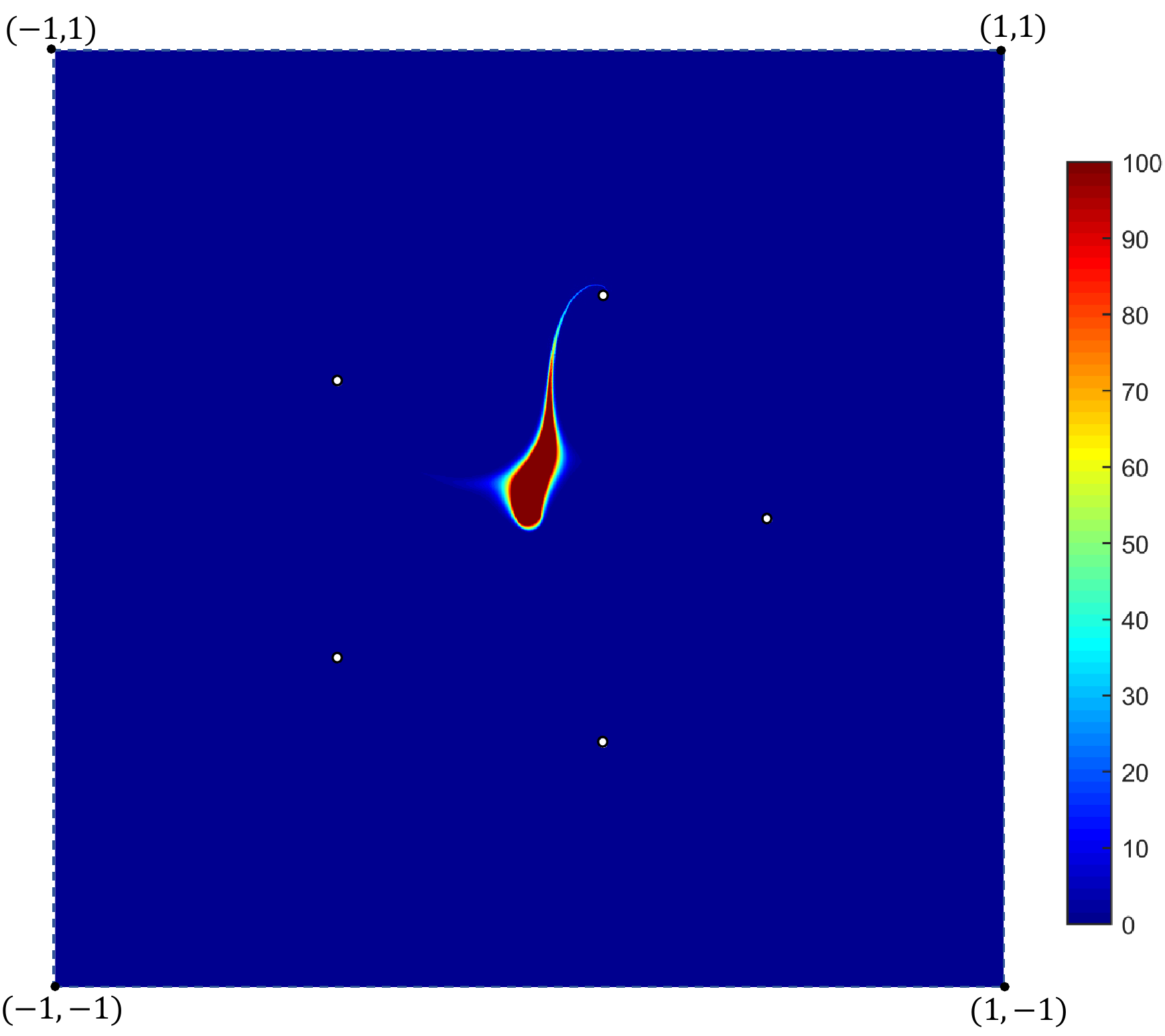}
				\end{minipage}%
				\hspace{0.073\textwidth}
				\begin{minipage}[c]{.4\textwidth}
					\centering
					\includegraphics[width=1\textwidth]{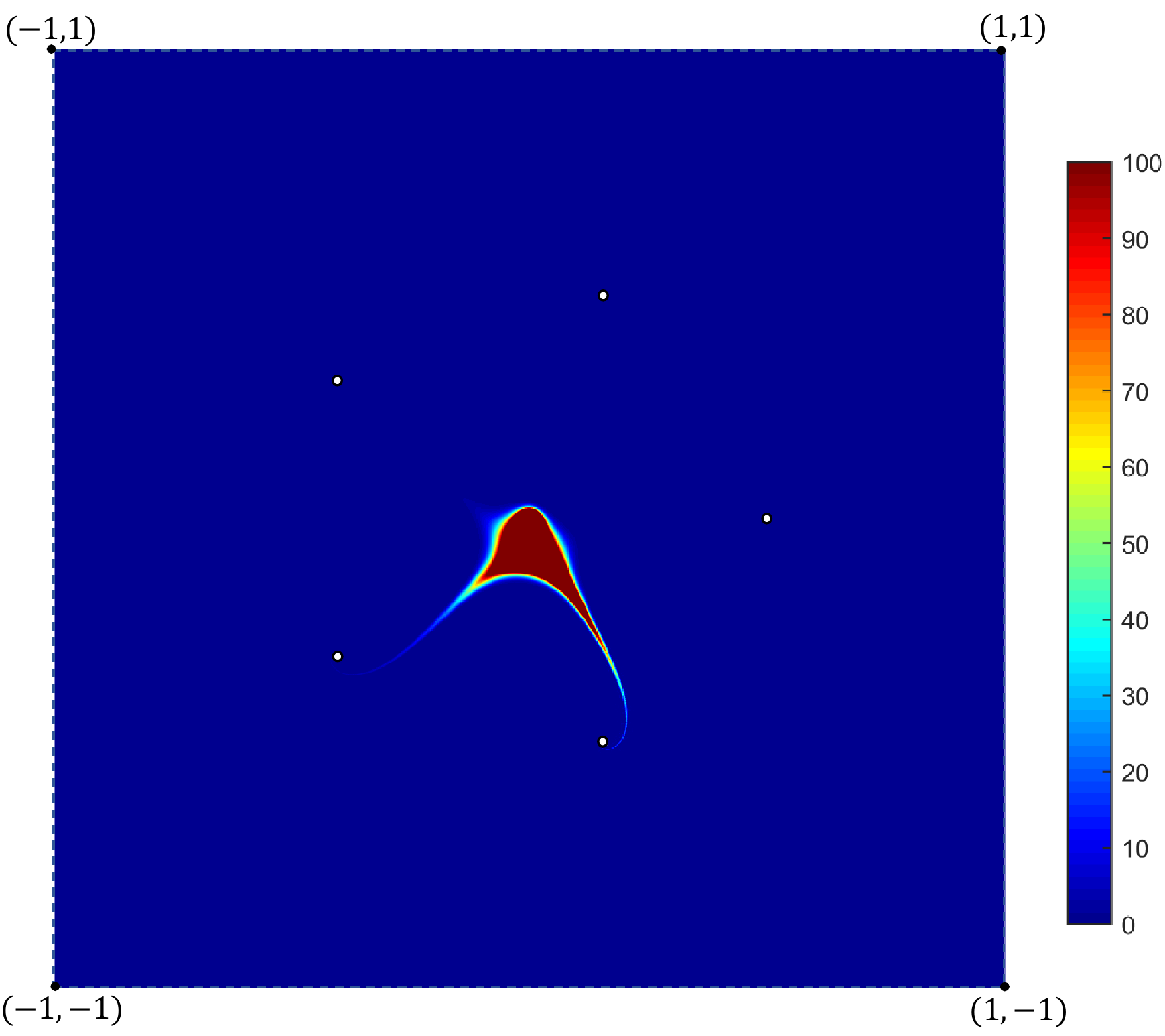}
				\end{minipage} \\ 
				\begin{minipage}[b]{.4\textwidth}
					\centering
					\subcaption{Random snapshot}
					\label{fig:Ex2_intial_problem_c}
				\end{minipage}%
				\begin{minipage}[b]{.4\textwidth}
					\subcaption{Random snapshot}
					\label{fig:Ex2_intial_problem_d}
				\end{minipage}
				\caption{(a) Geometry of the advection-diffusion problem. (b) The solution field $u$ for parameter value $\mu_s := (0,\allowbreak 0,\allowbreak 0.308,\allowbreak 0.308, \allowbreak 0.308,\allowbreak 0.308, \allowbreak 0.308,\allowbreak 0.616, \allowbreak 0.616,\allowbreak 0.616,\allowbreak 0.616,\allowbreak 0.616)$. (c)-(d) The solution field $u$  for two random samples from $\mathcal{P}$.}
				\label{fig:Ex2_intial_problem}
			\end{figure}

			The advection field is taken as a potential (divergence-free and curl-free) field  consisting of a linear combination of $12$ components,
			$$ \bm{\beta}(\boldsymbol{x}) = \mu_2 \cos(\mu_1) \hat{\boldsymbol{e}}_x + \mu_2 \sin(\mu_1)\hat{\boldsymbol{e}}_y +\sum^{10}_{i=1} \mu_i \bm{\beta}_i(\boldsymbol{x}),~~\boldsymbol{x} \in \Omega,$$
			where
			\begin{equation} \label{eq:beta_advection}
			\bm{\beta}_i (\boldsymbol{x}) = \left \{
			\begin{array}{rll}
			\frac{-\hat{\boldsymbol\bm{e}}_r(\boldsymbol{x}_{i})}{\|\boldsymbol{x} - \boldsymbol{x}_{i} \|} &\text{ for } 1 \leq i \leq 5\\
			\frac{-\hat{\boldsymbol{e}}_\theta(\boldsymbol{x}_{i-5})}{\|\boldsymbol{x} - \boldsymbol{x}_{i-5} \|}  &\text{ for } 6 \leq i \leq 10.
			\end{array}
			\right.
			\end{equation}
			The vectors $\hat{\boldsymbol{e}}_x$ and $\hat{\boldsymbol{e}}_y$ are the basis vectors of the Cartesian system of coordinates. The vectors  $\hat{\boldsymbol{e}}_r(\boldsymbol{x}_{j})$ and $\hat{\boldsymbol{e}}_\theta(\boldsymbol{x}_{j})$ are the basis vectors of the polar coordinate system with the origin at point $\boldsymbol{x}_{j}$, $1 \leq j \leq 5$. {From the physics perspective}, we have here a superposition of two uniform flows and five hurricane flows (each consisting of a sink and a rotational flow) centered at different locations. The source term is $$f(\boldsymbol{x})  =  \left \{ \begin{array}{rl}
			\frac{1}{\pi 0.025^2} &\text{ for } \boldsymbol{x} \in \Omega_s,\\ 0 &\text{ for } \boldsymbol{x} \in \Omega/ \Omega_s.\\
			\end{array} \right.$$

			We consider a multi-objective scenario, where one aims to approximate the average solution field $s^j(u)$, $1 \leq j \leq {10}$, inside sensor $\Omega_j$ having a form of a disc of radius $0.025$ located as in~Figure\nobreakspace \ref {fig:Ex2_intial_problem_a}.  The objective is to obtain sensor outputs for the parameter values $\mu := (\mu_1, \cdots, \mu_{12}) \in \mathcal{P}:= \{[0,~2 \pi] \times [0,~0.028] \times [0.308,~0.37]^5 \times [0.616,~0.678]^5 \}$.
			Figures\nobreakspace  \ref {fig:Ex2_intial_problem_a} to\nobreakspace  \ref {fig:Ex2_intial_problem_c}  present solutions $u(\mu)$ for few samples from $\mathcal{P}$.

			The discretization of the problem was performed with the classical finite element method. A nonuniform mesh was considered with finer elements near the pores of the hurricanes, and larger ones far from the pores such that each element's Peclet number inside $[-1,1]^2$ was larger than $1$ for any parameter value in $\mathcal{P}$. Moreover, it was revealed that for this benchmark the solution field outside the region $[-1,1]^2$ was practically equal to zero for all $\mu \in \mathcal{P}$. Therefore the outer region was discretized with coarse elements. For the discretization we used about $380000$ and $20000$ degrees of freedom in the main region and the outside boundary layer, respectively, which yielded approximately $400000$ degrees of freedom in total. 
			
			The solution space is equipped with {the} inner product
			\begin{equation*}
			{\langle\bv, \bw \rangle_U:= \langle {\nabla}v, {\nabla}w \rangle_{L^2},~\bv,\bw \in U,}
			\end{equation*}
			{that is the $H^1_0$ inner product for functions associated with vectors in $U$}. 
			
			For this problem, approximation of the solution with a fixed low-dimensional space is ineffective. 
			The problem has to be approached with non-linear approximation methods with  parameter-dependent approximation spaces. For this, the classical $hp$-refinement method is computationally intractable due to high dimensionality of the parameter domain, which makes the dictionary-based approximation to be the most pertinent choice.  
			
			The training and test sets $\mathcal{P}_\mathrm{train}$ and $\mathcal{P}_\mathrm{test}$ were respectively chosen as $20000$ and $1000$ uniform random samples from $\mathcal{P}$. Then, Algorithm\nobreakspace \ref {alg:sk_greedy_online} was employed to generate dictionaries of sizes $K=1500$, $K=2000$ and $K=2500$ for the dictionary-based approximation with $r=100$, $r=75$ and $r=50$ vectors, respectively. For comparison, we also performed a greedy reduced basis algorithm (based on sketched minres projection) to generate a fixed reduced approximation space, which in particular coincides with~Algorithm\nobreakspace \ref {alg:sk_greedy_online} with large enough $r$ (here $r=750$).  Moreover, for more efficiency (to reduce the number of online solutions) at $i$-th iteration of~Algorithm\nobreakspace \ref {alg:sk_greedy_online} and reduced basis algorithm instead of taking $\mu^{i+1}$ as a maximizer of  $\Delta^\bPhi(\bu_r(\mu);\mu)$ over $\mathcal{P}_{\mathrm{train}}$, we relaxed the problem to finding any parameter-value such that 
			\begin{equation} \label{eq:newargmax}
			\Delta^\bPhi(\bu_r(\mu^{i+1});\mu^{i+1}) \geq \max_{\mu \in \mathcal{P}_\mathrm{train}} \min_{1 \leq j\leq i} \Delta^\bPhi(\bu^j_r(\mu);\mu),
			\end{equation}
			where $\bu^j_r(\mu)$ denotes the solution obtained at the $j$-th iteration. Note that~\textup {(\ref {eq:newargmax})} improved the efficiency, yet yielding at least as accurate maximizer of the dictionary-based width (defined in~\textup {(\ref {eq:drwidth})}) as considering $\mu^{i+1}:=\arg \max_{\mu \in \mathcal{P}{\mathrm{train}}}\Delta^\bPhi(\bu_r(\mu);\mu)$. For the error certification purposes, each $250$ iterations the solution was computed on the whole training set and $\mu^{i+1}$ was taken as $\arg \max_{\mu \in \mathcal{P}{\mathrm{train}}}\Delta^\bPhi(\bu_r(\mu);\mu)$. Figure\nobreakspace \ref {fig:Ex2_1} depicts {the observed evolutions of the errors} in the greedy algorithms.
			
			\begin{figure}[h!]
				\centering
				\begin{subfigure}[b]{.4\textwidth}
					\centering
					\includegraphics[width=\textwidth]{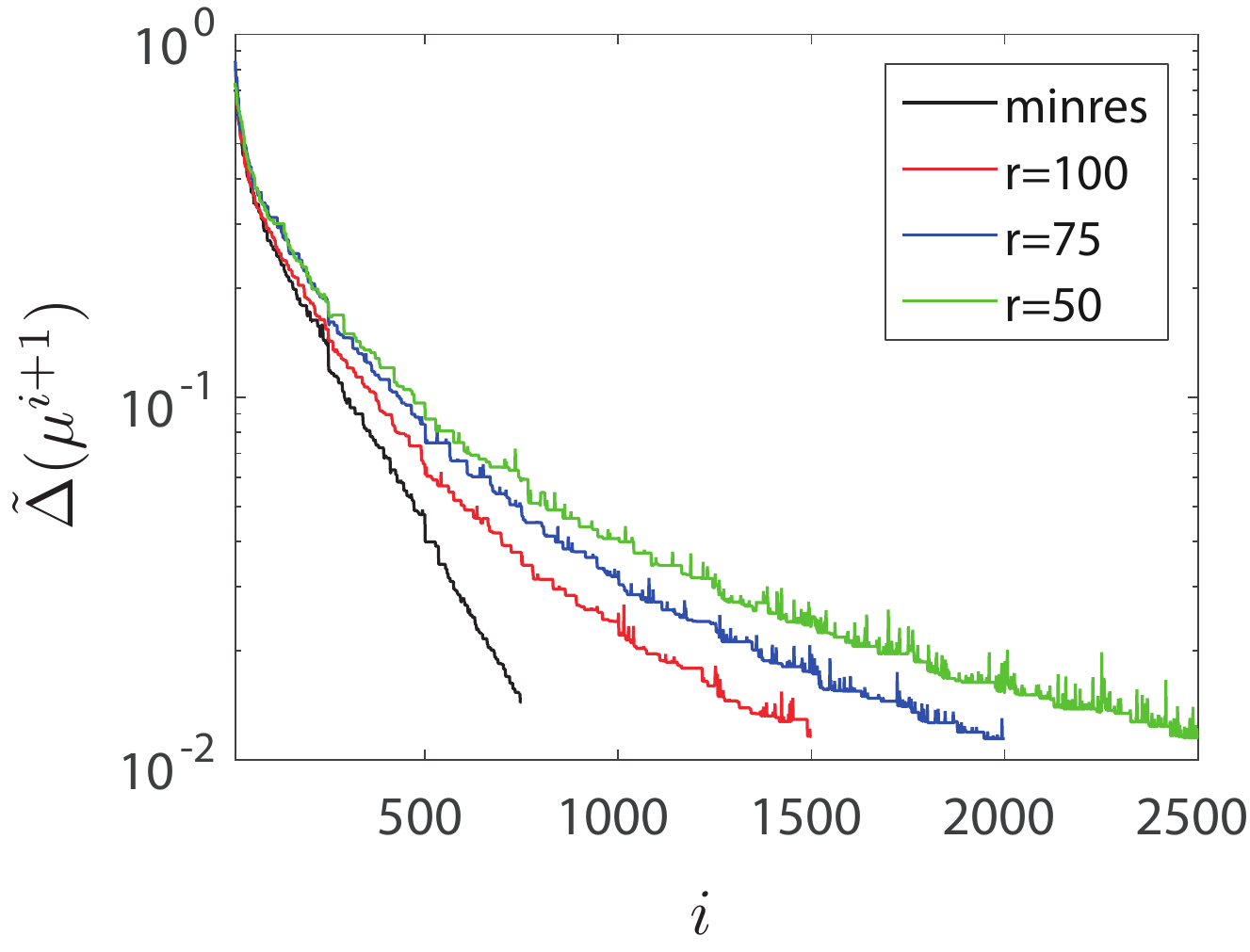}
					\caption{}
					\label{fig:Ex2_1a}
				\end{subfigure} \hspace{.01\textwidth}
				\begin{subfigure}[b]{.4\textwidth}
					\centering
					\includegraphics[width=\textwidth]{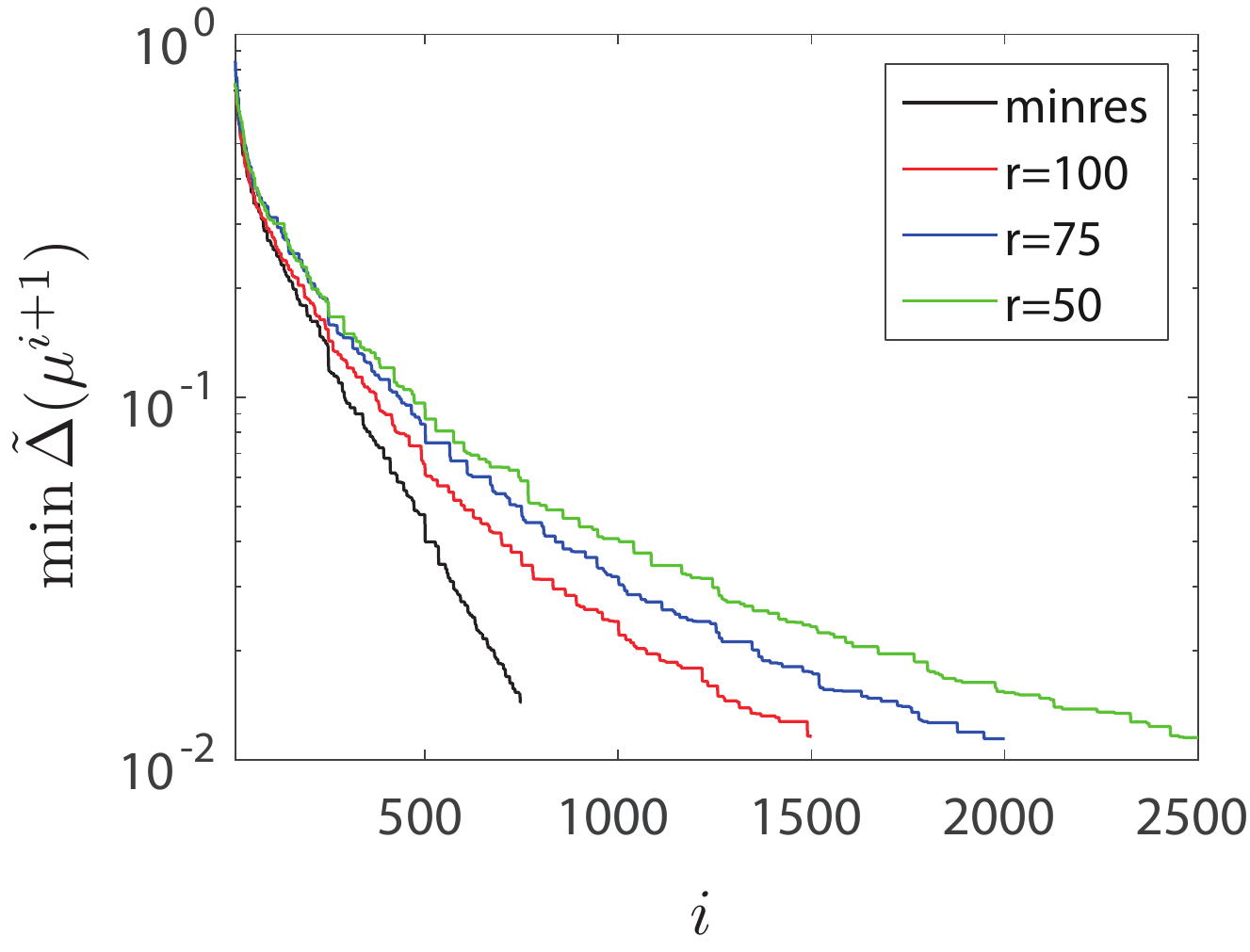}
					\caption{}
					\label{fig:Ex2_1b}
				\end{subfigure}	
				\caption{{Evolutions of the errors in}~Algorithm\nobreakspace \ref {alg:sk_greedy_online} for the dictionary generation for varying values of $r$, and the reduced basis greedy algorithm based on (sketched) minres projection. (a) The residual error  $\tilde{\Delta}(\mu^{i+1}):= \|\br(\bu_r(\mu^{i+1});\mu^{i+1})\|_{U'}/\| \bb \|_{U'}$. (b) The minimal value of the error at parameter value $\mu^{i+1}$ at the first $i$ iterations.}
				\label{fig:Ex2_1}
			\end{figure}
			
			We see that at the first $r$ iterations, the error decay for the dictionary generation practically coincides with the error decay of the reduced basis algorithm, which can be explained by the fact that the first $r$ iterations of the two algorithms coincide. {The convergence rates remain the same} for the reduced basis algorithm (even at high iterations), while they slowly subsequently degrade for dictionary-based approximation. The {latter} method still highly outperforms the former one, since its online computational cost scales only linearly with the number of iterations. Furthermore, for the dictionary-based approximation the convergence of the error is {moderately noisy.} The noise is primarily due to approximating the solutions of online sparse least-squares problems with the orthogonal greedy algorithm, for which the accuracy can be sensitive to the enrichment of the dictionary with new vectors.  The quality of online solutions could be improved by the usage of more sophisticated methods for sparse least-squares problems.    
			
			As it is clear from~Figure\nobreakspace \ref {fig:Ex2_1}, {the obtained dictionaries and the bases taken from them} provide approximations at least as accurate (on the training set) as the minres approximation with a fixed reduced space of dimension $r=750$. Yet, the dictionary-based approximations are much more online-efficient. Table\nobreakspace \ref {tab:runtimes} provides the online complexity and storage requirements for obtaining the dictionary-based solutions for all $\mu \in \mathcal{P}_\mathrm{test}$ (recall, $\#\mathcal{P}_\mathrm{test} = 1000$) with the orthogonal greedy algorithm (Algorithm\nobreakspace \ref {alg:skomp}) from a  sketch of size $k=8r$, and the sketched minres solutions with QR factorization with {Householder transformations (as in Matlab\textsuperscript{\textregistered} {R2017b} least-squares solver)} of the sketched reduced matrix in~\textup {(\ref {eq:skminresprojeff})} from a sketch of size $k=4r$. In particular, we see that the dictionary-based approximation with $r=75$ and $K=2000$ yields a gain in complexity by a factor of $15$ and memory consumption by a factor of $1.9$. In~Table\nobreakspace \ref {tab:runtimes} we also provide the associated runtimes and required RAM. It is revealed that the dictionary-based approximation with $K=2000$ and $r=75$ had an about $4$ times speedup. The difference between the gains in terms of complexity and runtime can be explained by {high} efficiency of the Matlab\textsuperscript{\textregistered} {R2017b} least-squares solver.
			It is important to note that even more considerable {enhance} of efficiency could be obtained by better exploitation of the structure of the dictionary-based reduced model, in particular, by representing the sketched matrix $\bV_K^\bTheta(\mu)$ in a format well suited for the  orthogonal greedy algorithm  (e.g., a product of a dense matrix by several sparse matrices similarly as in~\cite{le2016flexible,rubinstein2009double}).
			
			\begin{table}[tbhp]
				\caption{Computational cost of obtaining online solutions for all parameter values from the test set with the reduced basis method (based on sketched minres projection) and the dictionary-based approximations.}
				\label{tab:runtimes}
				\centering
				\scalebox{0.85}{
					\begin{tabular}{|l|r|r|r|r|} \hline
						& \multicolumn{1}{c|}{RB, $r=750$} & \multicolumn{1}{c|}{$K=1500,~r=100$} &  \multicolumn{1}{c|}{$K=2000,~r=75$} &  \multicolumn{1}{c|}{$K=2500,~r=50$}\\ \hline
						{Complexity in flops} & $3.1 \times 10^9$ & $0.27\times 10^9$  & $0.2\times 10^9$ & $0.12\times 10^9$ \\ [2pt] \hline
						{Storage in flns} & $2.9\times 10^8$ & $1.6\times 10^8$ & $1.6\times 10^8$& $1.3\times 10^8$  \\ \hline
						{CPU in $\mathrm{s}$} & $400$&  $124$ &  $113$ & $100$  \\ \hline
						{Storage in MB} & $234$ & $124$ & $124$& $104$  \\ \hline	
					\end{tabular}
				}
			\end{table}

			Further we provide statistical analysis of the dictionary-based approximation with $K=2000$ and $r=75$. For this we computed the associated dictionary-based solutions $\bu_r(\mu)$ for all parameter values in the test set, considering $\bTheta$ of varying sizes. The accuracy of an approximation is characterized by the quantities $\Delta_\mathcal{P}:=\max_{\mu \in \mathcal{P}_{\mathrm{test}}} \| \br(\bu_r(\mu); \mu) \|_{U'} /\|\bb\|_{U'}$, $e_\mathcal{P}:=\max_{\mu \in \mathcal{P}_{\mathrm{test}}} \|\bu(\mu) - \bu_r(\mu)\|_{U} / \max_{\mu \in \mathcal{P}_{\mathrm{test}}} \|\bu(\mu)\|_{U}$ and $e^{i}_\mathcal{P}=\max_{\mu \in \mathcal{P}_{\mathrm{test}}} |s^i(\bu(\mu))-s^i(\bu_r(\mu))|$, $1\leq i \leq {10}$. Figure\nobreakspace \ref {fig:Ex2_2} depicts the dependence of $\Delta_\mathcal{P}$, $e_\mathcal{P}$ and $e^{i}_\mathcal{P}$ (for few selected values of $i$) on the number of rows $k$ of $\bTheta$. For each value of $k$, the statistical properties of $\Delta_\mathcal{P}$, $e_\mathcal{P}$ and $e^{i}_\mathcal{P}$ were characterized with $20$ realizations of $\Delta_\mathcal{P}$, $e_\mathcal{P}$ and $e^{i}_\mathcal{P}$. It is observed that for $k=600$, the errors $\Delta_\mathcal{P}$ and $e_\mathcal{P}$ are concentrated around $0.03$ and $0.06$, respectively. 
			Moreover, for all tested $k \geq 600$ we obtained nearly the same errors, which suggests preservation of the quality of the dictionary-based approximation by its sketched version at $k=600$. A (moderate) deviation of the errors in the quantities of interest (even for very large $k$) can be explained by (moderately) low effectivity of representation of these errors with the error in $\| \cdot \|_U$, which we considered to control. 
			\begin{figure}[htp]
				\centering
				\begin{subfigure}[b]{.4\textwidth}
					\centering
					\includegraphics[width=\textwidth]{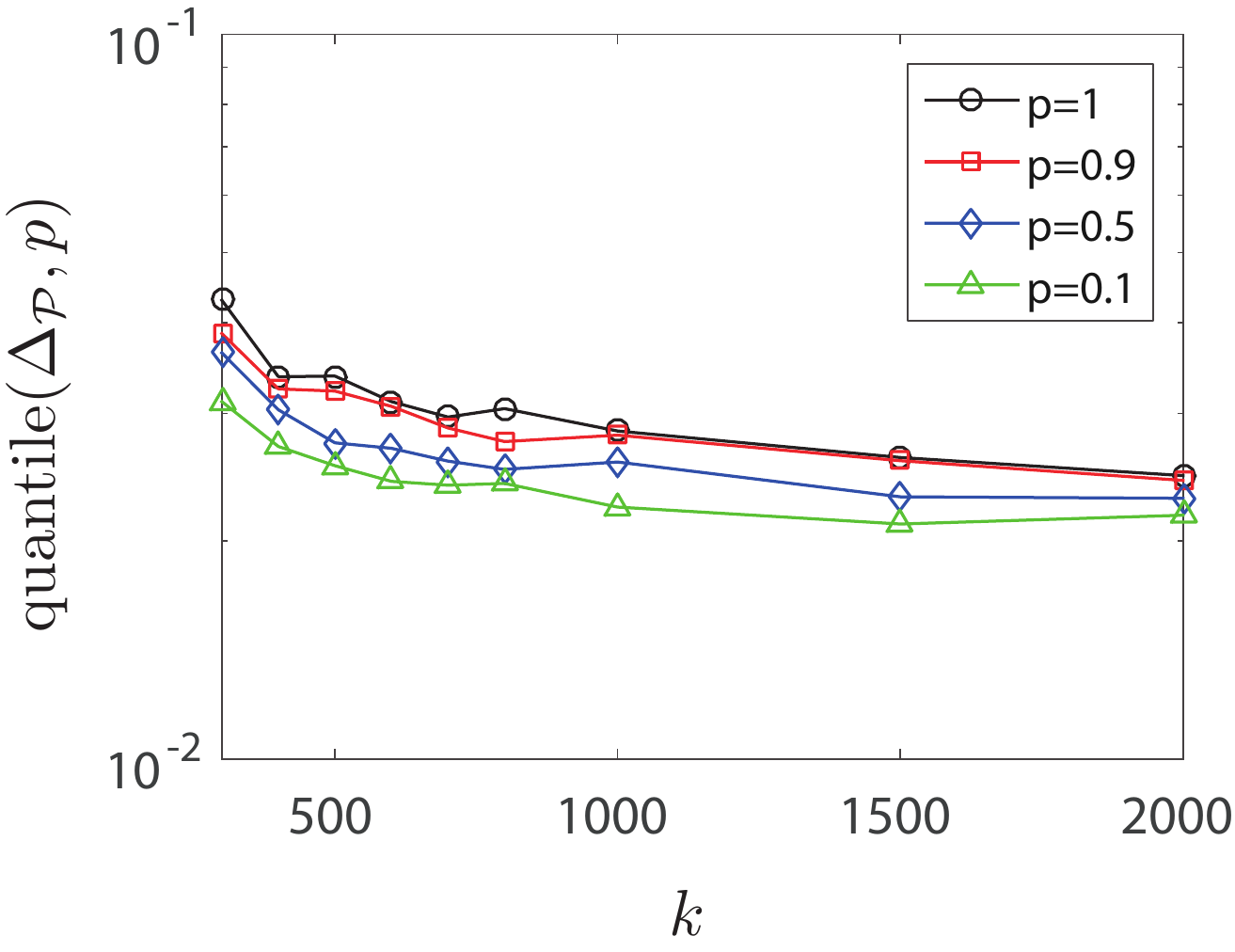}
					\caption{}
					\label{fig:Ex2_2a}
				\end{subfigure} \hspace{.01\textwidth}
				\begin{subfigure}[b]{.4\textwidth}
					\centering
					\includegraphics[width=\textwidth]{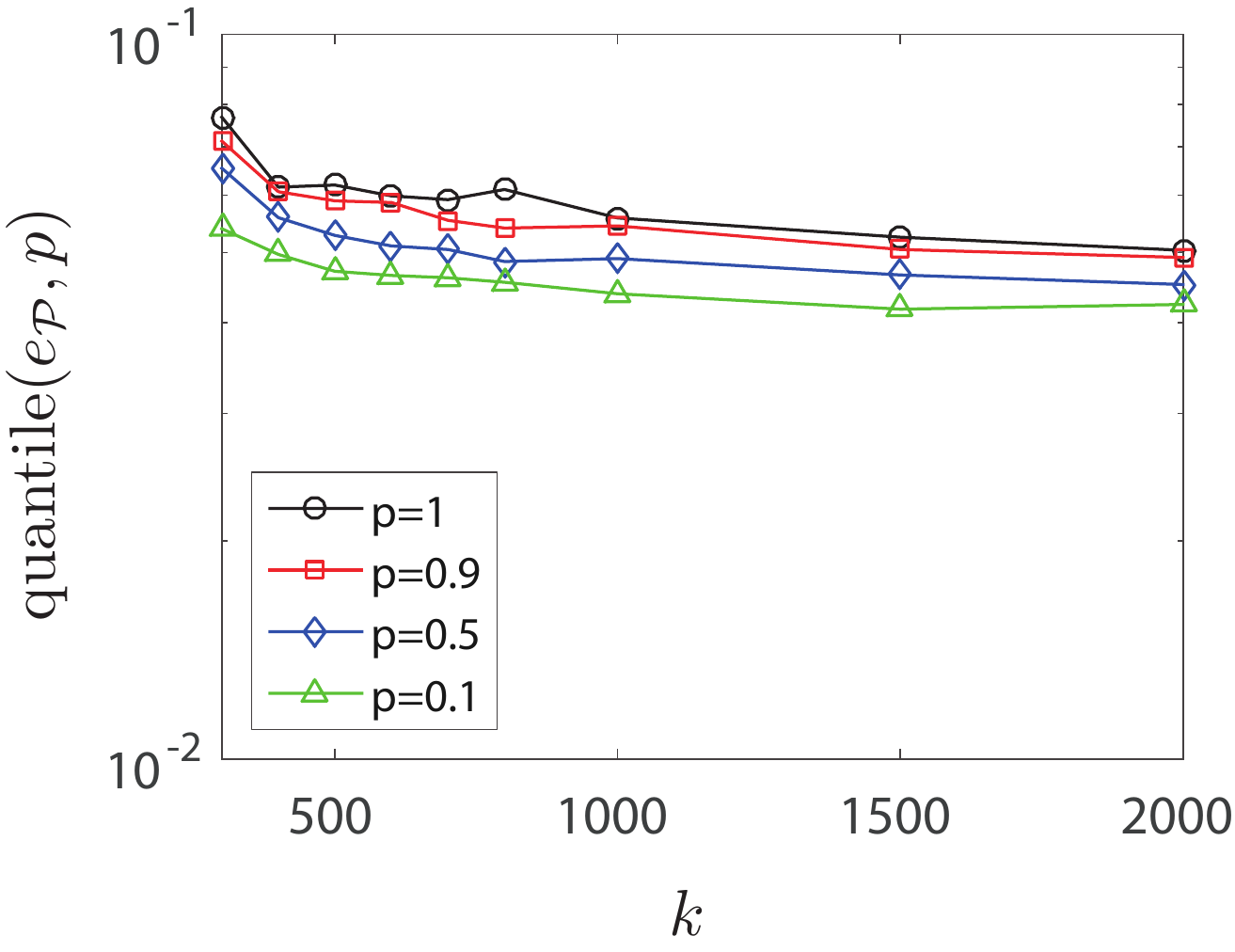}
					\caption{}
					\label{fig:Ex2_2b}
				\end{subfigure}	
				\begin{subfigure}[b]{.4\textwidth}
					\centering
					\includegraphics[width=\textwidth]{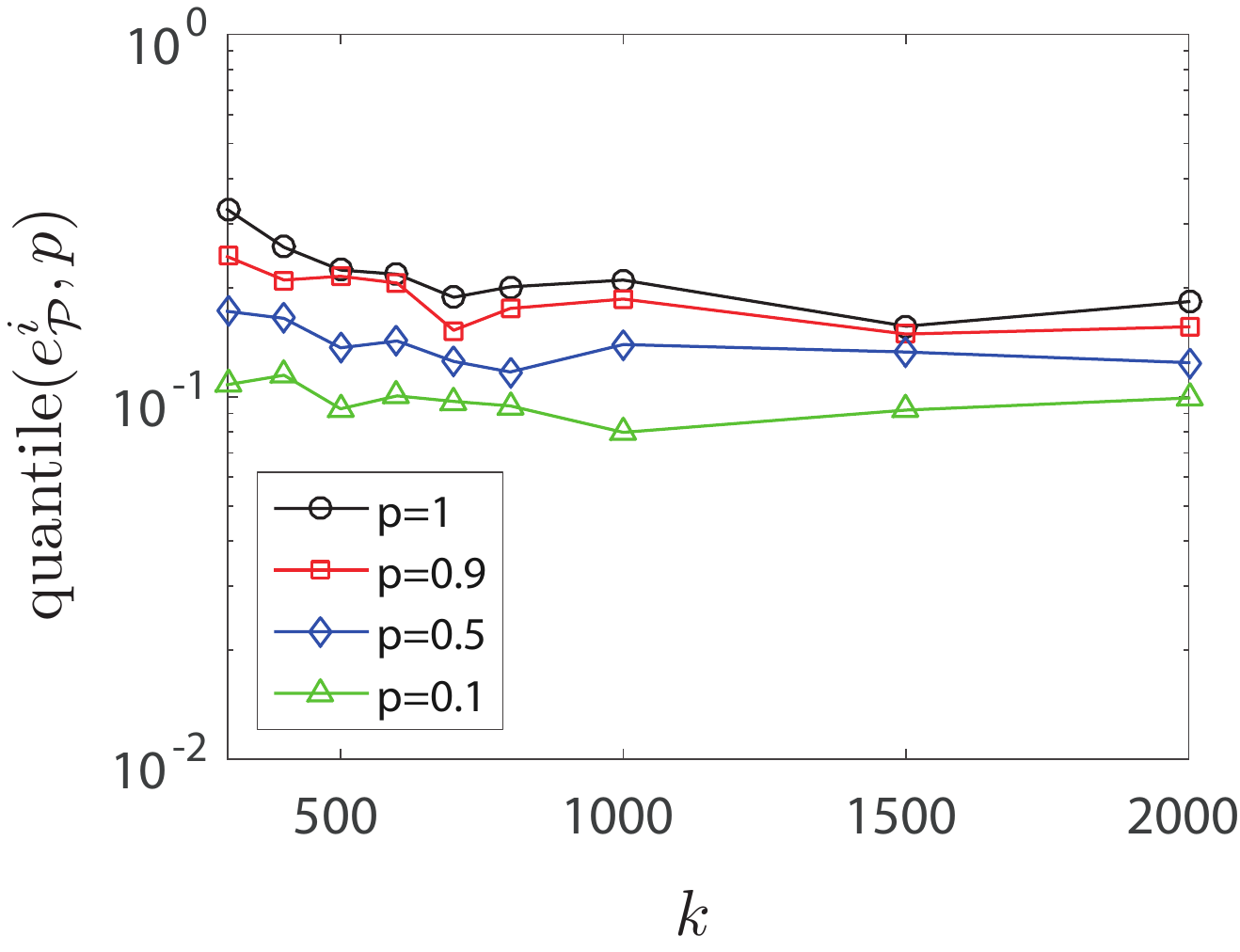}
					\caption{}
					\label{fig:Ex2_2c}
				\end{subfigure} \hspace{.01\textwidth}
				\begin{subfigure}[b]{.4\textwidth}
					\centering
					\includegraphics[width=\textwidth]{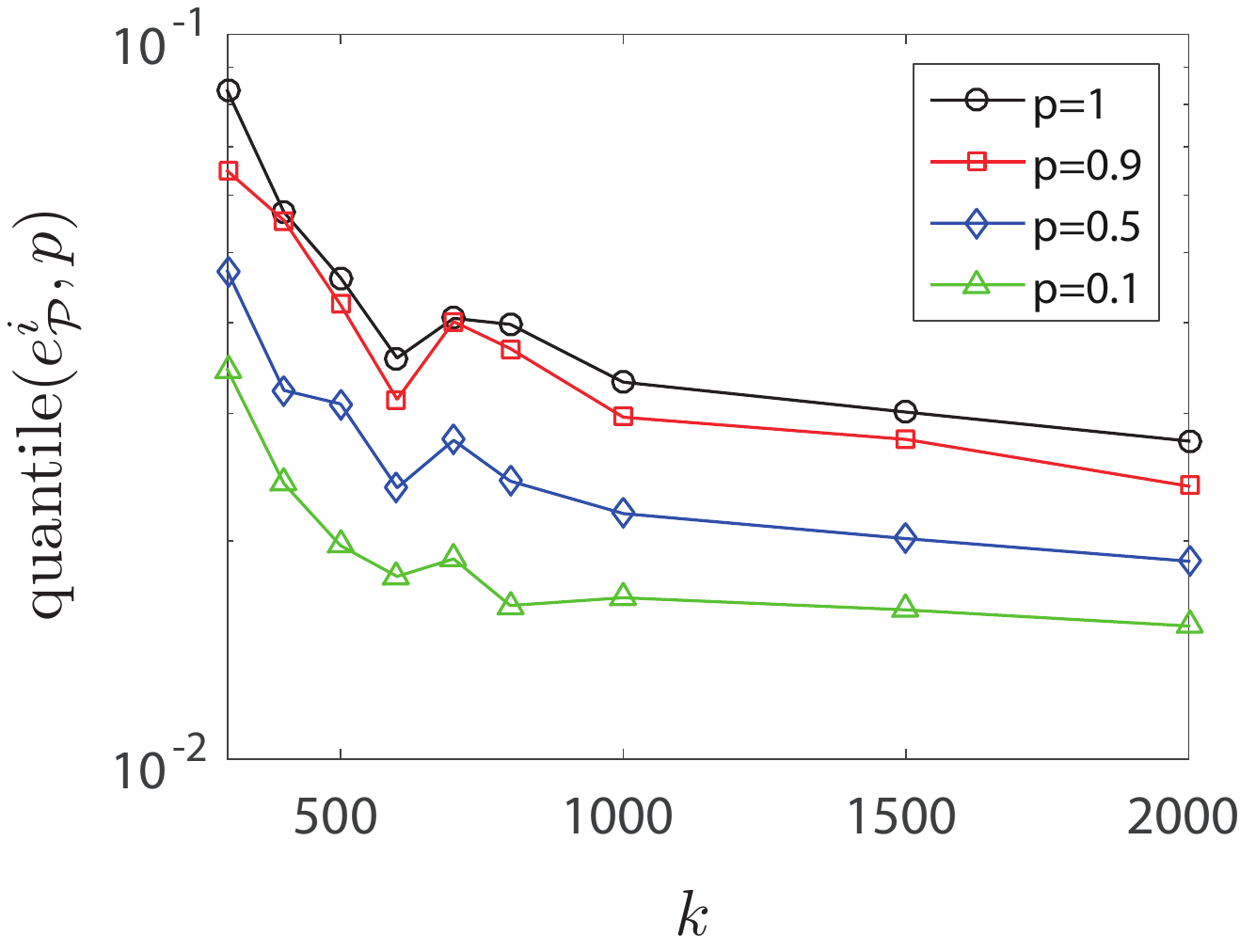}
					\caption{}
					\label{fig:Ex2_2e}
				\end{subfigure} \hspace{.01\textwidth}
				\caption{Quantiles of probabilities $p=1, 0.9, 0.5$ and $0.1$ over $20$ samples of the errors $\Delta_\mathcal{P}$, $e_\mathcal{P}$, $e^\mathrm{i}_\mathcal{P}$ of the dictionary-based approximation with $K=2000$ and $r=75$, versus the number of rows of $\bTheta$. (a) Residual error $\Delta_\mathcal{P}$. (b) Exact error $e_\mathcal{P}$. (c) Error $e^{i}_\mathcal{P}$ in the quantity of interest associated with sensor {$i=9$}. (d) Error $e^{i}_\mathcal{P}$ in the quantity of interest associated with sensor $i=1$.}
				\label{fig:Ex2_2}
			\end{figure}

			\section{Conclusion} \label{concl}	 
			
			In this article we have extended the methodology from~\cite{balabanov2019galerkin} to minres methods and proposed a novel nonlinear approximation method to tackle problems with a slow decay of Kolmogorov $r$-width.  
			The main ingredient of our approach is the approximation of the reduced model's solution from a random sketch, which entails drastic reduction of the computational costs and improvement of numerical stability. Precise conditions on the
			sketch to yield the (approximate) preservation of the quasi-optimality constants of the reduced model's solution are provided.  These conditions do not depend on the
			operator's properties, which implies robustness for ill-conditioned and non-coercive
			problems. Moreover, these
			conditions can be ensured with SRHT or
			Gaussian matrices of sufficiently large sizes depending only logarithmically on the probability of failure and on the cardinality of the dictionary (for dictionary-based approximation).
			
			We here also proposed efficient randomized methods for extraction of the quantity of interest {(see Appendix A)} and a posteriori certification of the reduced model's sketch. 
			These results along with the sketched minres projection can be used as a remedy of the drawbacks revealed in~\cite{balabanov2019galerkin}.

			The applicability of the proposed methodology was realized on two benchmark problems difficult to tackle with standard methods. The experiments on {the} invisibility cloak benchmark {confirmed} that random sketching {can indeed provide} high computational savings in both offline and online stages, and more numerical stability compared to the standard minres method while preserving the quality of the output.  It was {illustrated} experimentally that the random sketching technique {may be better suited} to minres methods than to Galerkin methods. Furthermore, the proposed procedure for the a posteriori certification of the sketch's quality was also experimentally tested. It yielded bounds for the dimension of random projections an order of magnitude less than the theoretical ones from~\cite{balabanov2019galerkin}.		
			In the advection-diffusion benchmark a slow-decay of the Kolmogorov $r$-width was revealed, which implied the necessity to use the dictionary-based approximation. It was verified that for this problem the dictionary-based approximation provided an {enhancement} of the online stage in more than an order of magnitude in complexity, in about $2$ times in terms of memory, and in about $4$ times in terms of runtime (compared to the sketched minres projection). Moreover, even higher computational savings could be obtained by representing the sketch of the dictionary in a more favorable format (e.g., as in~\cite{le2016flexible,rubinstein2009double}), which we leave for future research.

			\newpage
			
			\section*{Appendix A. Post-processing the reduced model's solution} \label{suppmaterial}
			\setcounter{section}{8}
			{In the paper} we presented a methodology for efficient computation of an approximate solution $\bu_r(\mu)$, or to be more precise, its coordinates in a certain basis, which can be the classical reduced basis for a fixed approximation space, or the dictionary vectors for dictionary-based approximation presented in~Section\nobreakspace \ref {dbminres}.  The approximate solution $\bu_r(\mu)$, however, is usually not what one should consider as the output. In fact, the amount of allowed online computations is highly limited and should be independent of the dimension of the full order model. Therefore outputting  $\mathcal{O}(n)$ bytes of data as $\bu_r(\mu)$ should be avoided when $\bu(\mu)$ is not the quantity of interest.
			
			Further, we shall consider an approximation with a single subspace $U_r$ noting that the presented approach can also be used for post-processing the dictionary-based  approximation from~Section\nobreakspace \ref {dbminres} (by taking $U_r$ as the subspace spanned by the dictionary vectors). Let $\bU_r$ be a matrix whose column vectors form a basis for $U_r$ and let $\ba_r(\mu)$ be the coordinates of $\bu_r(\mu)$ in this basis. 
			A general quantity of interest $s(\mu):=l(\bu(\mu); \mu)$ can be approximated by $s_r(\mu):=l(\bu_r(\mu); \mu)$. Further, let us assume a linear case where  $l(\bu(\mu); \mu) := \langle \bl(\mu), \bu(\mu) \rangle $ with $\bl(\mu) \in U'$ being the extractor of the quantity of interest. Then 
			\begin{equation} \label{eq:quantity}
			s_r(\mu)=\langle \bl(\mu), \bu_r(\mu) \rangle=\bl_r(\mu)^\mathrm{H} \ba_{r}(\mu),
			\end{equation} 
			where $\bl_r(\mu):=\bU_r^\mathrm{H} \bl(\mu)$. 
			
			\begin{remark} \label{rmk:postsolext}
				In general, our approach can be used for estimating an inner product between arbitrary parameter-dependent vectors. The possible applications include efficient estimation of the primal-dual correction and an extension to quadratic quantities of interest. In particular, the estimation of the primal-dual correction can be obtained by replacing $\bl(\mu)$ by $\br(\bu_r(\mu); \mu)$ and  $\bu_r(\mu)$ by $\bv_r(\mu)$ in~\textup {(\ref {eq:quantity})}, where $\bv_r(\mu) \in U$ is a reduced basis (or dictionary-based) approximate solution to the adjoint problem. A quadratic output quantity of interest has the form $l(\bu_r(\mu);\mu):=\langle \bL(\mu) \bu_r(\mu) + \bl(\mu), \bu_r(\mu) \rangle$, where $\bL(\mu):U \to U'$ and $\bl(\mu) \in U'$. Such $l(\bu_r(\mu);\mu)$ can be readily derived from~\textup {(\ref {eq:quantity})} by replacing $\bl(\mu)$ with $\bL(\mu) \bu_r(\mu) + \bl(\mu)$.
			\end{remark}
			The affine factors of $\bl_r(\mu)$ should be first precomputed in the offline stage and then used for online evaluation of $\bl_r(\mu)$ for each parameter value with a computational cost independent of the dimension of the original problem. The offline computations required for evaluating the affine factors of $\bl_r(\mu)$, however, can still be {too expensive or even unfeasible to perform.} Such a scenario may arise when using a high-dimensional approximation space (or a dictionary), when the extractor $\bl(\mu)$ has many (possibly expensive to maintain) affine terms, or when working in an extreme computational environment, e.g., with data streamed or distributed among multiple workstations. In addition, evaluating $\bl_r(\mu)$ from the affine expansion as well as evaluating $\bl_r(\mu)^\mathrm{H} \ba_{r}(\mu)$ itself can be subject to round-off errors {(especially when $\bU_r$ is ill-conditioned and may not be orthogonalized)}. Further, we shall provide a (probabilistic) way for estimating $s_r(\mu)$ with a reduced computational cost and better numerical stability. As the core we take the idea from~\cite[Section 4.3]{balabanov2019galerkin} proposed as a workaround to expensive offline computations for the evaluation of the primal-dual correction. 
			
			\begin{remark} \label{rmk:semiinner2}
				{
					The spaces $U$ and $U'$ are equipped with inner products $\langle \cdot, \cdot \rangle_U$ and $\langle \cdot, \cdot \rangle_{U'}$ (defined by matrix $\bR_U$), which are used for controlling the accuracy of the approximate solution $\bu_r(\mu)$. In general, $\bR_U$ is chosen according to both the operator $\bA(\mu)$ and the extractor $\bl(\mu)$ of the quantity of interest. The goal of this section, however, is only the estimation of the quantity of interest from the given $\bu_r(\mu)$. Consequently, for many problems it can be more pertinent to use here a different $\bR_U$ than the one employed for obtaining and characterizing $\bu_r(\mu)$. {The choice for $\bR_U$ should be done according to  $\bl(\mu)$ (independently of $\bA(\mu)$).} For instance, for discretized parametric PDEs, if $\bl(\mu)$ represents an integral of the solution field over the spatial domain then it is natural to choose $\langle \cdot, \cdot \rangle_U$ corresponding to {the} $L^2$ inner product. Moreover, $\langle \cdot, \cdot \rangle_U$ is required to be an inner product only on a certain subspace of interest, {which means that $\bR_U$ may be a positive semi-definite  matrix.} This consideration can be particularly helpful when the quantity of interest depends only on the restriction of the solution field to a certain subdomain. In such a case, $\langle \cdot, \cdot \rangle_U$ can be chosen to correspond with an inner product between restrictions of functions to this subdomain. The extension of random sketching for estimation of semi-inner products is straightforward (see~Remark\nobreakspace \ref {rmk:semiinner}).  }
			\end{remark}
			\begin{remark} \label{rmk:semiinner}
				{{Let us outline the extension of the methodology to the case where $\langle \cdot, \cdot \rangle_{U}$ is not definite. Let us assume that $\langle \cdot, \cdot \rangle_{U}$ is an inner product on a subspace $W \subseteq U$ of interest.  Then, it follows that $W':=\{ \bR_U \bx : \bx \in W \}$ can be equipped with $\langle \cdot, \cdot \rangle_{U'}:= \langle  \cdot, \bR_U^{\dagger} \cdot \rangle$, where $\bR_U^{\dagger}$ is {a} pseudo-inverse of $\bR_U$. {Such products} $\langle \cdot, \cdot \rangle_{U}$ and $\langle \cdot, \cdot \rangle_{U'}$ can be approximated by
						\begin{equation} \label{eq:thetadef2}
						\langle \cdot, \cdot \rangle^{\bTheta}_{U}:= \langle \bTheta \cdot, \bTheta \cdot \rangle, \textup{ and }
						\langle \cdot, \cdot \rangle^{\bTheta}_{U'} := \langle \bTheta \bR_U^{\dagger} \cdot, \bTheta \bR_U^{\dagger} \cdot \rangle.
						\end{equation} 
						This can be useful for the estimation of a (semi-)inner product between parameter-dependent vectors (see~Remark\nobreakspace \ref {rmk:semiinner2}).  
				}}
			\end{remark}

			\subsection{Approximation of the quantity of interest} 
			
			An efficiently computable and accurate estimation of $s_r(\mu)$ can be obtained in two phases. In the first phase, the manifold  $\mathcal{M}_r := \{ \bu_r(\mu) : \mu \in \mathcal{P} \}$ is (accurately enough) approximated with a subspace $W_p:=\mathrm{span}(\bW_p) \subset U$, which is spanned by an efficient to multiply {(i.e., sparse or low-dimensional)} matrix $\bW_p$. This matrix can be selected a priori or obtained depending on $\mathcal{M}_r$. 
			In~Section\nobreakspace \ref {effspaces} we shall provide some strategies for choosing or computing the columns for $\bW_p$. The appropriate strategy should be selected depending on the particular problem and computational architecture. 
			Further, the solution vector $\bu_r(\mu)$ is approximated by its orthogonal projection $\bw_p(\mu):= \bW_p \bc_p(\mu)$ on $W_p$.  The coordinates $\bc_p(\mu)$ can be obtained from $\ba_r(\mu)$ by
			\begin{equation}
			\bc_p(\mu) = \bH_p \ba_r(\mu),
			\end{equation}
			where $\bH_p := [\bW^\mathrm{H}_p \bR_U \bW_p]^{-1} \bW^\mathrm{H}_p \bR_U \bU_r$. Note that since $\bW_p$ is efficient to multiply by, the matrix $\bH_p$ can be efficiently precomputed in the offline stage. We arrive {at} the following estimation of $s_r(\mu)$: 
			\begin{equation} \label{eq:approxquantity}
			s_r(\mu) \approx  \langle \bl(\mu), \bw_p(\mu) \rangle = \bl^\star_r(\mu)^\mathrm{H} \ba_r(\mu), 
			\end{equation}
			where $\bl^\star_r(\mu)^\mathrm{H}:= \bl(\mu)^\mathrm{H} \bW_p \bH_p$. Unlike $\bl_r(\mu)$, the affine factors of $\bl^\star_r(\mu)$ can now be efficiently precomputed thanks to the structure of $\bW_p$.
			
			In the second phase of the algorithm, the precision of~\textup {(\ref {eq:approxquantity})} is improved with a sketched (random) correction associated with an $U \to \ell_2$ subspace embedding $\bTheta$:  
			\begin{equation}
			\begin{split}
			s_r(\mu) &= \langle \bl(\mu), \bw_p(\mu) \rangle +\langle \bl(\mu), \bu_r(\mu) - \bw_p(\mu) \rangle \\ &\approx \langle \bl(\mu), \bw_p(\mu) \rangle + \langle \bR^{-1}_{U} \bl(\mu), \bu_r(\mu) - \bw_p(\mu) \rangle^\bTheta_U =: s^\star_r(\mu).
			\end{split}
			\end{equation}
			In practice, $s^\star_r(\mu)$ can be efficiently evaluated using the following relation:
			\begin{equation} \label{eq:quantityapprox}
			s^\star_r(\mu) = [\bl^\star_r(\mu)^\mathrm{H} + {}_{\Delta}\bl^\star_r(\mu)^\mathrm{H} ]\ba_r(\mu), 
			\end{equation}
			where the affine terms of ${}_{\Delta}\bl^\star_r(\mu)^\mathrm{H} := \bl^\bTheta(\mu)^\mathrm{H} (\bU^\bTheta_r - \bW^\bTheta_p \bH_p)$ can be precomputed from the $\bTheta$-sketch of $\bU_r$, a sketched matrix $\bW^\bTheta_p:=\bTheta \bW_p$ and the matrix $\bH_p$ with a negligible computational cost. 
			\begin{proposition} \label{thm:errorbounds_r}
				If $\bTheta$ is an  $( \varepsilon, \delta, 1)$ oblivious $U \to \ell_2$ subspace embedding, 
				\begin{equation} \label{eq:quantitybound}
				|s_r(\mu) - s^\star_r(\mu)| \leq \varepsilon\|\bl (\mu)\|_{U'} \| \bu_r(\mu) - \bw_p(\mu) \|_U
				\end{equation}	
				holds for a fixed parameter $\mu \in \mathcal{P}$ with probability at least $1-2\delta$.
			\end{proposition}

			\begin{proposition} \label{thm:errorbounds_r2}
				Let $L \subset U$ denote a subspace containing $\{\bR^{-1}_{U} \bl(\mu): \mu \in \mathcal{P} \}$. 
				If $\bTheta$ is an $\varepsilon$-embedding for {$L+U_r+W_p$}, then~\textup {(\ref {eq:quantitybound})} holds for all $\mu \in \mathcal{P}$.
			\end{proposition}

			It follows that the accuracy of $s^\star_r(\mu)$ can be controlled through the quality of $W_p$ for approximating $\mathcal{M}_r$, the quality of $\bTheta$ as an $U \to \ell_2$ $\varepsilon$-embedding, or both. {Note that choosing $\bTheta$ as a null matrix (i.e., an $\varepsilon$-embedding for $U$ with $\varepsilon =1$)} leads to a single first-phase approximation~\textup {(\ref {eq:approxquantity})}, while letting $W_p := \{ \bnull \}$ corresponds to a single sketched (second-phase) approximation. Such particular choices for $\bTheta$ or $W_p$ can be pertinent when the subspace $W_p$ is highly accurate so that there is practically no benefit to use a sketched correction or, the other way around, when the computational environment or the problem does not permit a sufficiently accurate approximation of $\mathcal{M}_r$ with $W_p$, therefore making the use of a non-zero $\bw_p(\mu)$ unjustified.
			
			\begin{remark}
				{When interpreting random sketching as a Monte Carlo method for the estimation of the inner product $\langle \bl(\mu), \bu_r(\mu) \rangle$, the proposed approach can be interpreted as  a control variate method where $\bw_p(\mu)$ plays the role of the control variate. A multileveled Monte Carlo method with different control variates should further improve the efficiency of post-processing.}
			\end{remark}
			
			\subsection{Construction of reduced subspaces} \label{effspaces}
			
			Further we address the problem of computing the basis vectors for $W_p$. In general, the strategy for obtaining $W_p$ has to be chosen according to the problem's structure and the constraints due to the computational environment.
			
			A simple way, used in~\cite{balabanov2019galerkin}, is to choose $W_p$ as the span of samples of $\bu(\mu)$ either chosen randomly or during the first few iterations of the reduced basis (or dictionary) generation with a greedy algorithm. Such $W_p$, however, may be too costly to operate with. Then we propose more sophisticated constructions of $W_p$.

			{\emph{Approximate Proper Orthogonal Decomposition.}}
			A subspace $W_p$ can be obtained by an (approximate) POD of the reduced vectors evaluated on a training set $\mathcal{P}_\mathrm{train} \subseteq \mathcal{P}$. Here, randomized linear algebra can be again employed for improving efficiency. The computational cost of the proposed POD procedure shall mainly consist of the solution of $m = \# \mathcal{P}_\mathrm{train}$ reduced problems and the multiplication of $\bU_r$ by $p=\mathrm{dim}(W_p) \ll r$ small vectors. Unlike the classical POD, our methodology does not require computation or maintenance of the full solution's samples and therefore allows large training sets.
			
			Let $L_{m} = \{ \ba_r (\mu^i)\}^m_{i=1}$ be a training sample of the coordinates of $\bu_r(\mu)$ in a basis $\bU_r$. We look for a POD subspace $W_r$ associated with the snapshot matrix $$\bW_{m}:=[ \bu_r (\mu^1), \bu_r (\mu^2), \hdots, \bu_r (\mu^m) ]= \bU_r  \bL_{m},$$
			where $\bL_{m}$ is a matrix whose columns are the elements from $L_{m}$. 
			
			An accurate estimation of POD can be efficiently computed via the sketched method of snapshots introduced in~\cite[Section 5.2]{balabanov2019galerkin}. More specifically, a quasi-optimal (with high probability) POD basis can be calculated as 
			
			\begin{equation} \label{eq:approxpod}
			\bW_p :=  \bU_r  \bT^*_p,
			\end{equation}
			where $$\bT^*_p:= \bL_m [\bt_1, \dots, \bt_p],$$ 
			with $\bt_1, \dots, \bt_p$ being the $p$ dominant singular vectors of $\bU^\bTheta_r \bL_m$.
			Note that {the} matrix $\bT^*_p$ can be efficiently obtained with a computational cost independent of the dimension of the full order model. The dominant cost is the multiplication of $\bU_r$ by $\bT^*_p$, which is also expected to be inexpensive since $\bT^*_p$ has a small number of columns. Guarantees for the quasi-optimality of $W_p$ can be readily derived from~\cite[Theorem 5.5]{balabanov2019galerkin}.
			
			{\emph{Sketched greedy algorithm.}}
			A greedy search over the training set $\{ \bu_r(\mu): \mu \in \mathcal{P}_{\mathrm{train}} \}$ of approximate solutions is another way to construct $W_p$. At the $i$-th iteration, $\bW_i$ is enriched with a vector $\bu_r (\mu^{i+1})$ with the largest distance to $W_i$ over the training set. Note that in this case the resulting matrix $\bW_p$ has the form~\textup {(\ref {eq:approxpod})}, where $\bT^*_p = [\ba_r(\mu^{1}), \dots, \ba_r(\mu^{p})]$.  The efficiency of the greedy selection can be improved by employing random sketching technique. At each iteration, the distance to $W_i$ can be measured with the sketched norm $\| \cdot \|^\bTheta_U$, which can be computed from sketches $\bTheta \bu_r(\mu) =  \bU^\bTheta_r \ba_r(\mu)$ of the approximate solutions with no need to operate with large matrix $\bU_r$ but only its sketch. This allows efficient computation of the quasi-optimal interpolation points $\mu^{1}, \dots, \mu^{p}$ and the associated matrix $\bT^*_p$. Note that for numerical stability an orthogonalization of $\bW_i$ with respect to $\langle \cdot, \cdot \rangle^\bTheta_U$ can be performed, that can be done by modifying $\bT^*_i$ so that $\bU^\bTheta_r \bT^*_i$ is an orthogonal matrix. Such $\bT^*_i$ can be obtained with standard QR factorization. When $\bT^*_p$ has been computed, the matrix $\bW_p$ can be calculated by multiplying $\bU_r$ with $\bT^*_p$. The quasi-optimality of  $\mu^{1}, \dots, \mu^{p}$ and approximate orthogonality of $\bW_p$ is guaranteed if $\bTheta$ is an $\varepsilon$-embedding for all subspaces from the set $\{ W_p + \mathrm{span}(\bu_r(\mu^i)) \}^m_{i=1}$. This property of $\bTheta$ can be guaranteed a priori by considering $(\varepsilon, \binom{m}{p}^{-1} \delta, p+1)$ oblivious $U \to \ell_2$ subspace embeddings, or certified a posteriori with the procedure explained in~Section\nobreakspace \ref {sketchcert}. 
			
			{\emph{Coarse approximation.}}
			Let us notice that the online cost of evaluating $s_r^\star(\mu)$ does not depend on the dimension $p$ of $W_p$. Consequently, if $W_p$ is spanned by structured (e.g., sparse) basis vectors then a rather high dimension is allowed (possibly larger than $r$). 
			
			{For classical numerical methods for PDEs, the resolution of the mesh (or grid) is usually chosen to guarantee both an approximability of the solution manifold by the approximation space and the stability. For many problems the latter factor is dominant and one choses the mesh primary to it. This is a typical situation for wave problems, advection-diffusion-reaction problems and many others. For these problems, the resolution of the mesh can be much higher than needed for the estimation of the quantity of interest from the given solution field. In these cases, the quantity of interest can be efficiently yet accurately approximated using a coarse-grid interpolation of the solution.} 
			
			Suppose that each vector $\bu \in U$ represents a function $u: \Omega \to \mathbb{K}$ in a finite-dimensional approximation space spanned by basis functions $\{ \psi_i(x) \}^{n}_{i=1}$ associated with a fine mesh of $\Omega$. The function $u(x)$ can be approximated by a projection on a coarse-grid approximation space spanned by basis functions $\{ \phi_i(x) \}^{p}_{i=1}$. For simplicity assume that each $\phi_i(x) \in \mathrm{span} \{ \psi_j(x) \}^{n}_{j=1}$. Then the $i$-th basis vector for $W_p$ can be obtained simply by evaluating the coordinates of $\phi_i(x)$ in the basis $\{ \psi_j(x) \}^{n}_{j=1}$. Note that for the classical finite element approximation, each basis function has a local support and the resulting matrix $\bW_p$ is sparse. 
			
			
			\subsection{Numerical validation}
			{The proposed methodology was validated experimentally on an invisibility cloak benchmark from~Section\nobreakspace \ref{cloak}.}
			
			{As in~Section\nobreakspace \ref{cloak} in this experiment we used an approximation space $U_r$ of dimension $r=150$ constructed with a greedy algorithm (based on sketched minres projection).  We used a training set of $50000$ uniform random samples  in $\mathcal{P}$, while the test set $\mathcal{P}_{test} \subset \mathcal{P}$ was taken as $1000$ uniform random samples in $\mathcal{P}$. } {The} experiment was performed for a fixed approximation $\bu_r(\mu)$ obtained with sketched minres projection on $U_r$ using $\bTheta$ with $1000$ rows. For such $\bu_r(\mu)$, {an} approximate extraction of the quantity $s_r(\mu) = l(\bu_r(\mu); \mu)$ from $\bu_r(\mu)$ (represented by coordinates in reduced basis) was considered.
			
			{The post-processing procedure was performed by choosing $\bl(\mu)$ and $\bu_r(\mu)$ in~\textup {(\ref {eq:quantity})} as $\bu_r(\mu) - \bu^{in}(\mu)$.} Furthermore, for better accuracy the solution space was here equipped with {(semi-)}inner product $\langle \cdot, \cdot \rangle_U := \langle \cdot, \cdot \rangle_{L^2(\Omega_1)}$ that is different from the inner product~\textup {(\ref {eq:ex1innerU})} considered for ensuring quasi-optimality of the minres projection and error estimation (see~Remark\nobreakspace \ref {rmk:semiinner2}).  
			For such choices of $\bl(\mu),~\bu_r(\mu)$ and $\langle \cdot, \cdot \rangle_U$, we employed a greedy search with error indicator $\min_{\bw \in W_i} \|\bu_r(\mu) - \bw \|^\bTheta_U$ over training set of $50000$ uniform samples in $\mathcal{P}$ to find $W_p$. Then $s_r(\mu)$ was efficiently approximated by $s^\star_r(\mu)$ given in~\textup {(\ref {eq:quantityapprox})}. In this experiment, the error is characterized by $e^\mathrm{s}_\mathcal{P}=\max_{\mu \in \mathcal{P}_{\mathrm{test}}} |s(\mu)-\tilde{s}_r(\mu)| / A_s$, where $\tilde{s}_r(\mu) = s_r(\mu)$ or $s^\star_r(\mu)$.
			The statistical properties of $e^\mathrm{s}_\mathcal{P}$ for each value of $k$ and $\dim(W_p)$ were obtained with $20$ realizations of $\bTheta$.
			Figure\nobreakspace \ref {fig:Ex1_2a} exhibits the dependence of $e^\mathrm{s}_\mathcal{P}$ on the size of $\bTheta$ with $W_p$ of dimension $\dim{(W_p)} =15$. Furthermore, in Figure\nobreakspace \ref {fig:Ex1_2b} we provide the maximum value $e^\mathrm{s}_\mathcal{P}$ from the computed samples for different sizes of $\bTheta$ and $W_p$.  The accuracy of $s^\star_r(\mu)$ can be controlled by the quality of $W_p$ for the approximation of $\bu_r(\mu)$ and the quality of $\langle \cdot, \cdot \rangle^\bTheta_U$ for the approximation of $\langle \cdot, \cdot \rangle_U$. When $W_p$ approximates well $\bu_r(\mu)$, one can use a random correction with $\bTheta$ of rather small size, while in the alternative scenario the usage of a large random sketch is required. In this experiment we see that the quality of the output is nearly preserved with high probability when using $W_p$ of dimension $\dim{(W_p)} =20$ and a sketch of size $k=1000$, or $W_p$ of dimension $\dim{(W_p)} =15$ and a sketch of size $k=10000$. For less accurate $W_p$, with $\dim{(W_p)} \leq 10$, the preservation of the quality of the output requires larger sketches of sizes $k \geq 30000$. For {optimizing} efficiency the dimension for $W_p$ and the size for $\bTheta$ should be picked depending on the dimensions $r$ and $n$ of $U_r$ and $U$, respectively, and the particular computational architecture. The increase of the considered dimension of $W_p$ entails storage and operation with more high-dimensional vectors, while the increase of the sketch entails higher computational cost associated with storage and operation with the sketched matrix $\bU_r^\bTheta = \bTheta \bU_r$. Let us finally note that for this benchmark the approximate extraction of the quantity of interest with {our approach}  using $\dim{(W_p)} =15$ and a sketch of size $k=10000$, required in about $10$ times less amount of storage and complexity than the classical exact extraction.  
			\begin{figure}[h!]
				\centering
				\begin{subfigure}[b]{.4\textwidth}
					\centering
					\includegraphics[width=\textwidth]{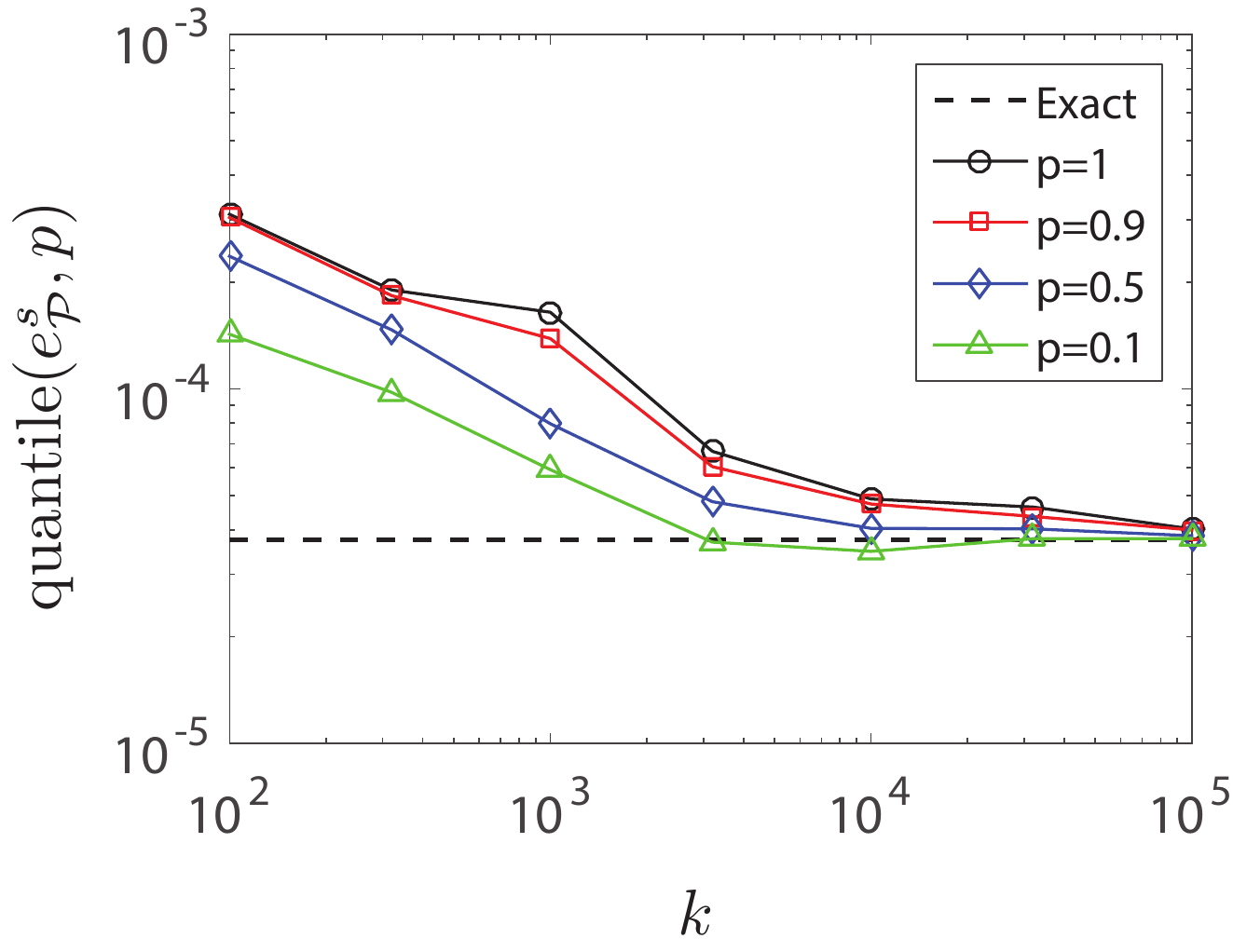}
					\caption{}
					\label{fig:Ex1_2a}
				\end{subfigure} \hspace{.01\textwidth}
				\begin{subfigure}[b]{.4\textwidth}
					\centering
					\includegraphics[width=\textwidth]{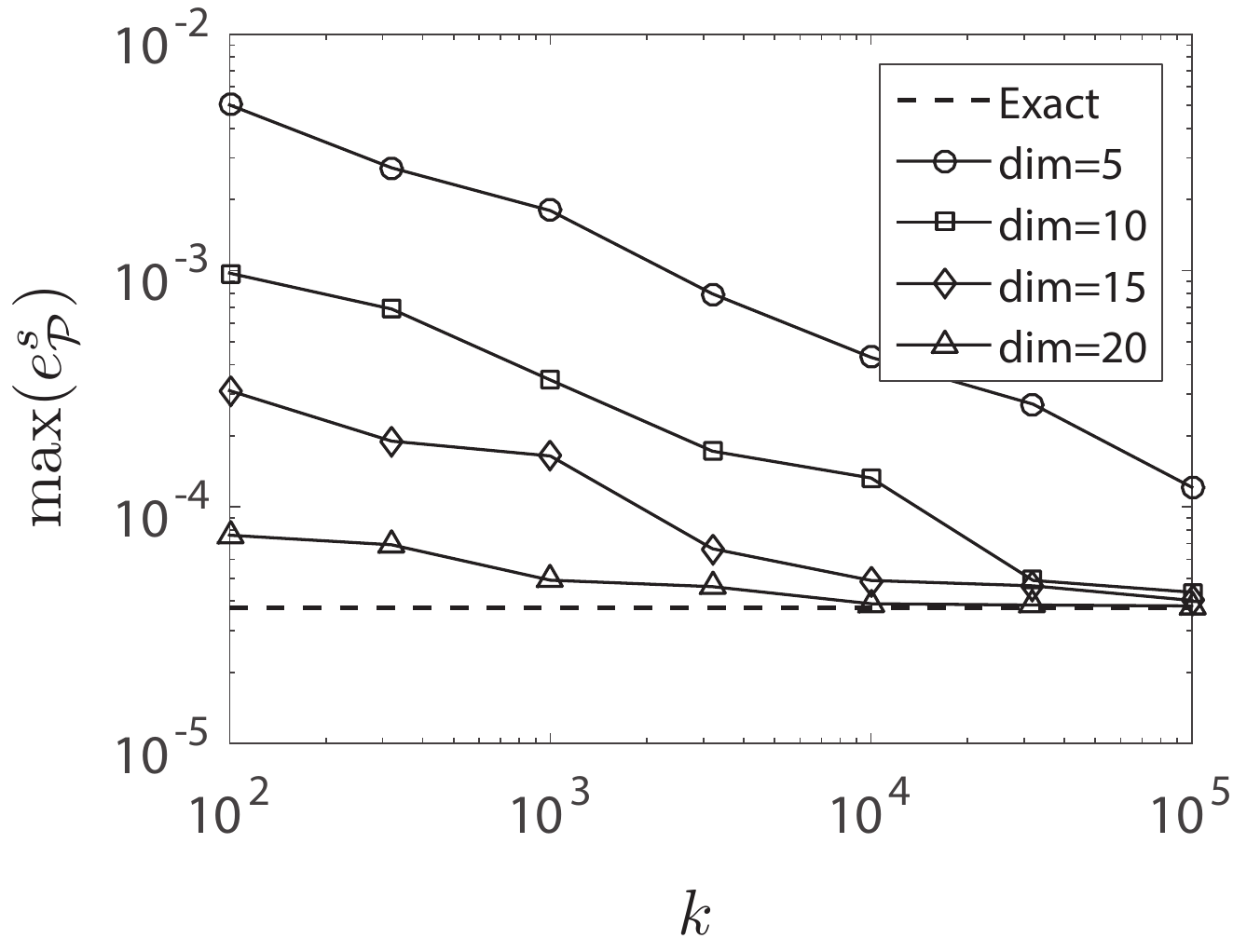}
					\caption{}
					\label{fig:Ex1_2b}
				\end{subfigure}	
				\caption{The error $e^\mathrm{s}_\mathcal{P}$ of $s_r(\mu)$ or its efficient approximation $s_r^\star(\mu)$ using $W_p$ and $\bTheta$ of varying sizes. (a) The error of $s_r(\mu)$ and quantiles of probabilities $p=1,0.9,0.5$ and $0.1$ over $20$ realizations of $e^\mathrm{s}_\mathcal{P}$ associated
					with $s_r^\star(\mu)$ using $W_p$ with $\dim{(W_p)}=10$, versus sketch's size $k$. (b) The error of $s_r(\mu)$ and maximum over $20$ realizations of $e^\mathrm{s}_\mathcal{P}$ associated
					with $s_r^\star(\mu)$, versus sketch's size $k$ for $W_p$ of varying dimension. }
				\label{fig:Ex1_2}
			\end{figure}
			
			\newpage
			\section*{Appendix B. Proofs of propositions.}
			
			\begin{proof} [Proof of Proposition\nobreakspace \ref{thm:cea}]
				The statement of the proposition follows directly from the definitions of the constants $\zeta_{r}(\mu)$ and $\iota_{r}(\mu)$, that imply 
				\begin{equation*}
				\zeta_{r}(\mu)\|\bu(\mu) - \bu_r(\mu) \|_U	\leq \|\br(\bu_r(\mu); \mu)\|_{U'} \leq \|\br(\bP_{U_r} \bu(\mu); \mu)\|_{U'} \leq \iota_{r}(\mu)\|\bu(\mu) - \bP_{U_r} \bu(\mu) \|_U.
				\end{equation*} 
				~
			\end{proof}

			\begin{proof} [Proof of Proposition\nobreakspace \ref{thm:skcea}]
				The proof follows the one of~Proposition\nobreakspace \ref {thm:cea}. 
			\end{proof}
			
			\begin{proof} [Proof of Proposition\nobreakspace \ref{thm:skminresopt}]
				{By the assumption on $\bTheta$,} we have that  $$\sqrt{1-\varepsilon}  \| \bA(\mu) \bx  \|_{U'} \leq \| \bA(\mu) \bx  \|^\bTheta_{U'} \leq \sqrt{1+\varepsilon}  \| \bA(\mu) \bx  \|_{U'} $$ holds for all $\bx \in \mathrm{span} \{ \bu(\mu) \}+ U_r$. Then~\textup {(\ref {eq:skalphabetabounds})} follows immediately. 
			\end{proof}
			
			\begin{proof} [Proof of Proposition\nobreakspace \ref{thm:skminresstab}]
				Let $\ba \in \mathbb{K}^{r}$ and $\bx:= \bU_r\ba$.	Then		
				\begin{align*}
				\frac{\| \bV^\bTheta_r(\mu) \ba \|}{\| \ba \|}&=  \frac{\| \bTheta\bR_U^{-1}\bA(\mu)\bU_r\ba\|}{ \| \ba \|} = \frac{\|\bA(\mu)\bx\|^\bTheta_{U'}}{ \| \bx \|^\bTheta_U}. 
				\end{align*}
				{Since $\bTheta$ is an $\varepsilon$-embedding for $U_r$, we have}
				\begin{equation*} 
				\sqrt{1-\varepsilon} \| \bx \|_U \leq \| \bx \|^{\bTheta}_U \leq  \sqrt{1+\varepsilon} \| \bx \|_U.
				\end{equation*}
				The statement of the proposition follows immediately. 
			\end{proof}

			\begin{proof} [Proof of Proposition\nobreakspace \ref{prop:widthsuper}]	
				Take arbitrary $\tau>0$, and let $\mathcal{D}^{(i)}_K$ be dictionaries with $\# \mathcal{D}^{(i)}_K = K_i$, such that
				$$ \sup_{\bu \in \mathcal{M}^{(i)}} \min_{W_{r_i} \in \mathcal{L}_{r_i}(\mathcal{D}^{(i)}_K)} \| \bu - \bP_{W_{r_i}} \bu\|_U \leq  \sigma_{r_i}(\mathcal{M}^{(i)}; K_i) + \tau,~~1 \leq i \leq l.   $$		
				The existence of $\mathcal{D}^{(i)}_K$ follows directly from the definition~\textup {(\ref {eq:drwidth})} of the dictionary-based width. 	
				Define
				$ {\mathcal{D}}^*_K = \bigcup^l_{i=1} \mathcal{D}^{(i)}_K.$ 			
				Then the following relations hold: 
				\begin{align*}
				\sum^l_{i=1} \sigma_{r_i}(\mathcal{M}^{(i)}; K_i) 
				&\geq \sum^l_{i=1} \left ( \sup_{\bu \in \mathcal{M}^{(i)}} \min_{W_{r_i} \in \mathcal{L}_{r_i}(\mathcal{D}^{(i)}_K)} \| \bu - \bP_{W_{r_i}} \bu\|_U - \tau \right) \\ 
				&\geq  \left (\sup_{\mu \in \mathcal{P}} \sum^l_{i=1} \min_{W_{r_i} \in \mathcal{L}_{r_i}(\mathcal{D}^{(i)}_K)} \| \bu^{(i)}(\mu) - \bP_{W_{r_i}} \bu^{(i)}(\mu)\|_U \right) - l \tau  \\ 
				& \geq \left ( \sup_{\mu \in \mathcal{P}}  \min_{W_{r} \in \mathcal{L}_{r}(\mathcal{D}^*_K)} \sum^l_{i=1} \| \bu^{(i)}(\mu) - \bP_{W_{r}} \bu^{(i)}(\mu)\|_U \right ) - l \tau \\
				& \geq  \sup_{\mu \in \mathcal{P}}  \min_{W_{r} \in  \mathcal{L}_{r}(\mathcal{D}^*_K)} \| \sum^l_{i=1} \bu^{(i)}(\mu) - \bP_{W_{r}} \sum^l_{i=1} \bu^{(i)}(\mu)\|_U - l \tau \geq \sigma_{{r}}({\mathcal{M}}; {K}) - l \tau.
				\end{align*} 
				Since the above relations hold for arbitrary positive $\tau$, we conclude that $$\sum^l_{i=1} \sigma_{r_i}(\mathcal{M}^{(i)}; K_i) \geq \sigma_{{r}}({\mathcal{M}}; {K}).$$
			\end{proof}

			\begin{proof} [Proof of Proposition\nobreakspace \ref{thm:Vmcea}]	
				Let $U^*_r(\mu) := \arg \min_{W_r \in \mathcal{L}_r(\mathcal{D}_K) } \| \bu(\mu) - \bP_{W_r} \bu(\mu)\|_U$. By definition of $\bu_r(\mu)$ and constants $\zeta_{r,K}(\mu)$ and $\iota_{r,K}(\mu)$,
				\begin{align*}
				{\zeta_{r,K}(\mu)\|\bu(\mu) - \bu_r(\mu) \|_U} 	& {\leq \|\br(\bu_r(\mu); \mu)\|_{U'} 
					\leq {D} \|\br(\bP_{U^*_r(\mu)} \bu(\mu); \mu)\|_{U'} + {\tau \|  \bb(\mu) \|_{U'}}} \\ &{\leq \iota_{r,K}(\mu) \left( {D} \|\bu(\mu) - \bP_{U^*_r(\mu)} \bu(\mu) \|_U + \tau \| \bu(\mu) \|_U  \right),}
				\end{align*}
				which ends the proof. 
			\end{proof}
			
			\begin{proof} [Proof of Proposition\nobreakspace \ref{thm:sksparsecea}]	
				The proof exactly follows the one of~Proposition\nobreakspace \ref {thm:Vmcea} by replacing $\| \cdot \|_{U'}$ with $\| \cdot \|^\bTheta_{U'}$. 
			\end{proof}
			
			\begin{proof} [Proof of Proposition\nobreakspace \ref{ssminresopt}]	
				We have that  $$\sqrt{1-\varepsilon}  \| \bA(\mu) \bx  \|_{U'} \leq \| \bA(\mu) \bx  \|^\bTheta_{U'} \leq \sqrt{1+\varepsilon}  \| \bA(\mu) \bx  \|_{U'} $$ 
				holds for all $\bx \in \mathrm{span} \{ \bu(\mu) \}+ W_r$ with $W_r \in \mathcal{L}_r(\mathcal{D}_K)$. The statement of the proposition then follows directly from the definitions of $\zeta_{r,K}(\mu), \iota_{r,K}(\mu), \zeta^\bTheta_{r,K}(\mu)$ and $\iota^\bTheta_{r,K}(\mu)$. 	
			\end{proof}
			
			\begin{proof} [Proof of Proposition\nobreakspace \ref{prop:stabalgsksparseminresproj}]
				{Let matrix $\bU_K$ have columns $\{\bw_i: i \in \{1 \cdots K \} \}$ and matrix $\bU_r(\mu)$ have columns  $\{\bw_i: i \in \mathcal{I}(\mu) \}$, with $\mathcal{I}(\mu) \subseteq \{1, \cdots,K\}$ being a subset of $r$ indices. 
					Let $\bx \in \mathbb{K}^{r}$ be an arbitrary vector. Define a sparse vector $\bz(\mu) = \left (z_i(\mu) \right )_{i \in \{1 \cdots K \}} $ with $\left (z_i(\mu) \right )_{i \in \mathcal{I}(\mu)} = \bx$ and zeros elsewhere.			 
					Let $\bw(\mu):= \bU_r(\mu)\bx= \bU_K\bz(\mu) $.}	Then		
				\begin{align*}
				\frac{\| \bV^\bTheta_r(\mu) \bx \|}{\| \bx \|}&=  \frac{\| \bTheta\bR_U^{-1}\bA(\mu)\bU_r(\mu)\bx\|}{ \| \bx \|} = \frac{\|\bA(\mu)\bw(\mu)\|^\bTheta_{U'}}{ \| \bx \|} =\frac{\|\bA(\mu)\bw(\mu)\|^\bTheta_{U'}}{ \| \bw(\mu) \|_U} \frac{\| \bw(\mu) \|_U}{\|\bx \|}   \\ 
				&\geq \zeta^\bTheta_{r,K}(\mu) \frac{\| \bU_r(\mu) \bx \|_U}{\|\bx \|} =
				\zeta^\bTheta_{r,K}(\mu) \frac{\| \bU_K \bz(\mu) \|_U}{\|\bz(\mu) \|} \geq \zeta^\bTheta_{r,K}(\mu) \Sigma^\mathrm{min}_{r,K}.
				\end{align*}
				Similarly, 
				\begin{align*}
				\frac{\| \bV^\bTheta_r(\mu) \bx \|}{\| \bx \|}  \leq \iota^\bTheta_{r,K}(\mu) \frac{\| \bU_r(\mu) \bx \|_U}{\|\bx \|} \leq \iota^\bTheta_{r,K}(\mu) \Sigma^\mathrm{max}_{r,K}.
				\end{align*}
				The statement of the proposition follows immediately. 
			\end{proof}

			\begin{proof} [Proof of Proposition\nobreakspace \ref{thm:certvectors}]	
				Using~Proposition\nobreakspace \ref {thm:errorbounds_r} with $\bl(\mu) := \bR_U \bx$, $\bu_r(\mu) := \by$, $\bw_p(\mu):=\bnull$, $\bTheta:=\bTheta^*$, $\varepsilon:= \varepsilon^*$ and $\delta:= \delta^*$, we have
				\begin{equation} \label{eq:certvectors1}
				\mathbb{P} (| \langle \bx, \by \rangle_{U} - \langle \bx, \by \rangle^{\bTheta^*}_{U} | \leq \varepsilon^* \| \bx \|_{U}  \| \by\|_{U}) \geq 1- 2\delta^*,
				\end{equation}
				from which we deduce that
				\begin{equation} \label{eq:certvectors2}
				\begin{split}
				| \langle \bx, \by \rangle^{\bTheta^*}_{U} - \langle \bx, \by \rangle^{\bTheta}_{U} | - {\varepsilon^*} \| \bx \|_{U}  \| \by\|_{U}  
				&\leq {| \langle \bx, \by \rangle_{U} - \langle \bx, \by \rangle^{\bTheta}_{U}|} \\ 
				&\leq | \langle \bx, \by \rangle^{\bTheta^*}_{U} - \langle \bx, \by \rangle^{\bTheta}_{U} | + {\varepsilon^*} \| \bx \|_{U}  \| \by\|_{U}
				\end{split}
				\end{equation}
				holds with probability at least $1- 2\delta^*$.
				In addition,
				\begin{equation} \label{eq:certvectors3}
				\mathbb{P} (| \| \bx \|^2_{U} -  ( \| \bx \|^{\bTheta^*}_{U}  )^2| \leq \varepsilon^* \| \bx \|^2_{U}) \geq 1-\delta^*
				\end{equation}
				and
				\begin{equation} \label{eq:certvectors4}
				\mathbb{P} (| \| \by \|^2_{U} -  ( \| \by \|^{\bTheta^*}_{U}  )^2| \leq \varepsilon^* \| \by \|^2_{U}) \geq 1-\delta^*.
				\end{equation}
				The statement of the proposition can be now derived by combining~ \textup {(\ref {eq:certvectors2})} to\nobreakspace  \textup {(\ref {eq:certvectors4})}  and using a union bound argument. 
			\end{proof}
			
			\begin{proof} [Proof of Proposition\nobreakspace \ref{thm:omegaUB}]	
				Observe that 
				$$\omega = \max \left \{ 1- \min_{\bx \in V / \{ \bnull \}} \left (\frac{\| \bx\|^{\bTheta}_U}{\| \bx\|_U} \right )^2, \max_{\bx \in V / \{ \bnull \}} \left (\frac{\| \bx\|^{\bTheta}_U}{\| \bx\|_U}\right )^2 -1  \right \}.$$
				Let us make the following assumption: $$   1- \min_{\bx \in V / \{ \bnull \}} \left (\frac{\| \bx\|^{\bTheta}_U}{\| \bx\|_U} \right )^2 \geq 	\max_{\bx \in V / \{ \bnull \}} \left (\frac{\| \bx\|^{\bTheta}_U}{\| \bx\|_U}\right )^2 -1. $$ For the alternative case the proof is similar. 
				
				Next, we show that $\bar{\omega}$ is an upper bound for $\omega$ with probability at least $1-\delta^*$. Define $\bx^*:= \arg \min_{\bx \in V / \{ \bnull \}, \| \bx \|_U=1} \| \bx\|^{\bTheta}_U$. By definition of $\bTheta^*$, 
				\begin{equation} \label{eq:obldef}
				1-\varepsilon^* \leq \left( \|\bx^*\|^{\bTheta^*}_{U} \right )^2
				\end{equation} 
				holds with probability at least $1-\delta^*$. If~\textup {(\ref {eq:obldef})} is satisfied, we have 
				$$\bar{\omega} \geq 1- (1-\varepsilon^*) \min_{\bx \in V / \{ \bnull \}} \left (\frac{\| \bx\|^{\bTheta}_U}{\| \bx\|^{\bTheta^*}_U} \right )^2 \geq  1- (1-\varepsilon^*) \left (\frac{\| \bx^*\|^{\bTheta}_U}{\| \bx^*\|^{\bTheta^*}_U} \right )^2 \geq 1 - (\| \bx^*\|^{\bTheta}_U)^2 = \omega.$$ 
				\quad 
			\end{proof}
			
			\begin{proof} [Proof of Proposition\nobreakspace \ref{thm:omegaUBoptimality}]	
				{By definition of $\omega$ and the assumption on $\bTheta^*$}, for all $\bx \in V$, it holds 
				$$ | \|\bx \|^2_{U} - ( \|\bx\|^{\bTheta^*}_{U}  )^2| \leq \omega^*\|\bx \|^2_{U},~\textup{ and } ~| \|\bx \|^2_{U} - ( \|\bx\|^{\bTheta}_{U}  )^2| \leq \omega\|\bx \|^2_{U}.$$
				The above relations and the definition~\textup {(\ref {eq:omegaUB})} of $\bar{\omega}$ yield
				\begin{equation*} 
				\bar{\omega} \leq \max \left \{ 1- (1-\varepsilon^*)\frac{1-\omega}{1+\omega^*}, (1+\varepsilon^*)\frac{1+\omega}{1-\omega^*} -1  \right \} = (1+\varepsilon^*)\frac{1+\omega}{1-\omega^*} -1,
				\end{equation*} 
				which ends the proof. 
			\end{proof}
			\begin{proof} [Proof of Proposition\nobreakspace \ref{thm:errorbounds_r}]					
			Denote $\bx := \bR^{-1}_U \bl (\mu) / \| \bl(\mu) \|_{U'}$ and $\by :=  (\bu_r(\mu) - \bw_p(\mu))/\|\bu_r(\mu) - \bw_p(\mu)\|_{U} $.
			Let us consider $\mathbb{K} = \mathbb{C}$, {which also accounts for the real case, $\mathbb{K} = \mathbb{R}$.} 
			Let  $$\omega:= \frac{ \langle \bx , \by \rangle_{U} - \langle \bx , \by \rangle^\bTheta_{U}}{|\langle \bx , \by \rangle_{U} - \langle \bx , \by \rangle^\bTheta_{U}|}. $$ Observe that $|\omega|=1$ and $\langle \bx , \omega \by \rangle_{U} - \langle \bx , \omega \by \rangle^\bTheta_{U}$ is a real number.
		
			By a union bound for the probability of success,  $\bTheta$ is an $\varepsilon$-embedding for $\mathrm{span}(\bx+\omega\by)$ and $\mathrm{span}(\bx-\omega\by)$ with probability at least $1-2\delta$. Then, using the parallelogram identity we obtain
			\begin{align*}
			4|\langle \bx , \by \rangle_{U} - \langle \bx , \by \rangle^\bTheta_{U}| &=|4\langle \bx , \omega \by \rangle_{U} - 4\langle \bx , \omega \by \rangle^\bTheta_{U}| \\
			& = | \| \bx +\omega \by \|_{U}^2 - \| \bx - \omega \by \|_{U}^2 +  4\mathrm{Im}(\langle \bx , \omega \by \rangle_{U}) \\
			&- \left( (\| \bx +\omega \by \|^\bTheta_{U})^2 - (\| \bx - \omega \by \|^\bTheta_{U})^2 + 4 \mathrm{Im}(\langle \bx , \omega \by \rangle^\bTheta_{U}) \right )| \\
			&= | \| \bx +\omega \by \|_{U}^2 -  (\| \bx +\omega \by \|^\bTheta_{U})^2  - \left( \| \bx - \omega \by \|_{U}^2 - (\| \bx - \omega \by \|^\bTheta_{U})^2 \right) \\
			&- 4\mathrm{Im}(\langle \bx , \omega \by \rangle_{U} - \langle \bx , \omega \by \rangle^\bTheta_{U}) |\\
			&\leq  \varepsilon \| \bx +\omega \by \|_{U}^2 + \varepsilon \| \bx - \omega \by \|_{U}^2 = 4 \varepsilon.
			\end{align*} 
			We conclude that relation~\textup {(\ref {eq:quantitybound})} holds with probability at least $1-2\delta$. 		
			\end{proof}

			\begin{proof} [Proof of Proposition\nobreakspace \ref{thm:errorbounds_r2}]					
			We can use a similar proof as in~Proposition\nobreakspace \ref {thm:errorbounds_r} with the fact that if $\bTheta$ is an $\varepsilon$-embedding for {$L+U_r +W_p$}, then it satisfies the $\varepsilon$-embedding property for  $\mathrm{span}(\bx+\omega\by)$ and $\mathrm{span}(\bx-\omega\by)$.  
			\end{proof}	
		\end{document}